\documentclass[11pt,a4paper,reqno]{amsart}

\renewcommand{\ldots}{\hspace{0.9pt}.\hspace{0.3pt}.\hspace{0.3pt}.\hspace{1.5pt}}

\newcommand{\falling}[2]{#1^{\underline{#2}}}

\newcommand{\HIDE}[1]{#1} 
\usepackage{xstring}
\newcommand{\longreal}[1]{%
  \StrLeft{#1}{12}%
} 

\usepackage[utf8]{inputenc}
\usepackage[T1]{fontenc}
\usepackage{mathrsfs}

\usepackage{amsmath}
\usepackage{amsthm}
\usepackage{amssymb}
\usepackage{mathtools}

\usepackage[foot]{amsaddr}

\usepackage{graphicx}
\usepackage{comment}
\usepackage{tikz}
\usetikzlibrary{calc, shapes, chains, positioning, fit}
\usepackage{verbatim}

\usepackage[margin=1in]{geometry}

\definecolor{darkblue}{rgb}{0,0,0.5}
\definecolor{cerisepink}{rgb}{0.93, 0.23, 0.51}
\definecolor{chocolate(traditional)}{rgb}{0.48, 0.25, 0.0}
\definecolor{chromeyellow}{rgb}{1.0, 0.8, 0.0}
\definecolor{bluegreen}{rgb}{0.0, 0.8, 0.9}
\definecolor{darkorange}{rgb}{0.8, 0.4, 0.0}
\definecolor{aonthergreen}{rgb}{0.0, 0.8, 0.4}
\definecolor{harvard-crimson}{rgb}{0.64, 0.06, 0.2}
\definecolor{dark-red}{rgb}{0.39, 0.15, 0.07}

\usepackage[unicode,
    bookmarksopen=true,
    bookmarksopenlevel=1,
    colorlinks=true,
    linkcolor=darkblue,
    linktoc=page,
    citecolor=darkblue,
    urlcolor=harvard-crimson,
    hypertexnames=false
]{hyperref}
\usepackage[capitalise]{cleveref}

\usepackage[shortlabels]{enumitem}
\setlist{topsep=.5em, itemsep=.3em}
\usepackage{subcaption}

\usepackage{thmtools}
\usepackage{thm-restate}

\usepackage[normalem]{ulem}
\usepackage{etoolbox}
\usepackage{xparse} 
\usepackage[normalem]{ulem}

\makeatletter
\patchcmd{\subsection}{\bfseries}{\bfseries\color{dark-red}}{}{} 
\patchcmd{\@begintheorem}{\thmhead}{\color{dark-red}\thmhead}{}{} 
\makeatother

\crefname{equation}{}{}
\crefname{theorem}{Theorem}{Theorems}
\crefname{claim}{Claim}{Claims}
\crefname{lemma}{Lemma}{Lemmas}
\crefname{corollary}{Corollary}{Corollaries}

\numberwithin{equation}{section}

\declaretheorem[name=Theorem, numberwithin=section]{theorem}

\declaretheorem[name=Lemma, sibling=theorem]{lemma}
\declaretheorem[name=Proposition, sibling=theorem]{proposition}
\declaretheorem[name=Conjecture, sibling=theorem]{conjecture}

\declaretheorem[name=Notation, numbered=no, style=definition]{notation*} 
\declaretheorem[name=Organisation, numbered=no, style=definition]{organization*} 

\declaretheorem[name=Claim, numberwithin=theorem]{claim}

\newenvironment{claimproof}[1][Proof of Claim]{%
  \begin{proof}[#1]%
}{%
  \end{proof}%
}

\newcommand{\eps}{\ensuremath{\varepsilon}}
\newcommand{\naturals}{\ensuremath{\mathbb N}}

\newcommand{\reals}{\ensuremath{\mathbb R}}

\newcommand{\cF}{\protect{\ensuremath{\mathcal F}}}
\newcommand{\cG}{\protect{\ensuremath{\mathcal G}}}

\ExplSyntaxOn
\cs_set:Npn \splitatcommas #1 {
  \seq_set_split:Nnn \l_tmpa_seq { , } { #1 }
  \seq_use:Nn \l_tmpa_seq { , \penalty0 \hspace{0pt plus 1em} }
}
\ExplSyntaxOff

\newcommand{\set}[1]{\{#1\}}

\newcommand{\paren}[1]{\left(#1\right)}

\NewDocumentCommand{\pr}{g}{%
  \IfNoValueTF{#1}{\mathbb{P}}{\mathbb{P}\left[#1\right]}%
}
\NewDocumentCommand{\ex}{g}{%
  \IfNoValueTF{#1}{\mathbb{E}}{\mathbb{E}\left(#1\right)}%
}
\NewDocumentCommand{\var}{g}{%
  \IfNoValueTF{#1}{\textup{\textrm{Var}}}{\textup{\textrm{Var}}\left(#1\right)}%
}
\NewDocumentCommand{\extremal}{g}{%
  \IfNoValueTF{#1}{\textup{\textrm{ex}}}{\textup{\textrm{ex}}\left(#1\right)}%
}

\newcommand{\APPENDIX}[1]{}
\newcommand{\NOTAPPENDIX}[1]{#1}
\renewcommand{\APPENDIX}[1]{#1}  
\renewcommand{\NOTAPPENDIX}[1]{} 

\AtBeginDocument{%
   \def\MR#1{}
}

\newcommand{\blow}[2]{\ifstrempty{#2}{#1{()}}{#1(#2)}}
    
\usepackage{bbm}
\usepackage{bm}
\newcommand{\B}[1]{{\boldsymbol #1}}
\newcommand{\I}[1]{{\mathbbm #1}}
\newcommand{\V}[1]{\bm{#1}}
\newcommand{\C}[1]{{\protect\mathcal{#1}}}
\renewcommand{\O}[1]{\overline{#1}}
\newcommand{\p}{p}
\renewcommand{\P}{P}
\newcommand{\N}{N}

\newcommand{\dedit}{\delta_{\mathrm{edit}}}

\usepackage{ifthen}
\newcommand{\LPSS}[1]{\ifthenelse{\equal{#1}{}}
{\cite{LiuPikhurkoSharifzadehStaden23}}
{\cite[#1]{LiuPikhurkoSharifzadehStaden23}}}

\newcommand{\hide}[1]{}

\usepackage{pgffor}
\usepackage{pgf}

\makeatletter
\newdimen\probgraph@tmpa
\newdimen\probgraph@tmpb

\pgfkeys{
  /probgraph/.is family, /probgraph,
  layout radius/.initial = 5.3mm,
  node radius/.initial   = 3.5mm,
  start angle/.initial   = 0,
  edge width factor/.initial = 0.9,
  node line width/.initial = 0.1pt,
  edge opacity/.initial  = 1
}

\newcommand{\ProbGraph}[3]{
  \begingroup
  \def\probgraph@K{0} 
  \foreach[count=\i from 1] \p in {#1}{
    \xdef\probgraph@K{\i}
  }

  \begin{tikzpicture}[baseline={(0,0)}, outer sep=0pt, inner sep=0pt]
    \foreach[count=\i from 0] \p in {#1}{ 
      \pgfmathsetmacro{\ang}{\pgfkeysvalueof{/probgraph/start angle} - 360*\i/(\probgraph@K)}
      \coordinate (probgraph-v\i) at (\ang:\pgfkeysvalueof{/probgraph/layout radius});
      \pgfmathsetmacro{\s}{sqrt(\p)}
      \pgfmathsetlengthmacro{\r}{\s*\pgfkeysvalueof{/probgraph/node radius}}
      \expandafter\xdef\csname probgraph@r@\i\endcsname{\r}
      \expandafter\xdef\csname probgraph@fill@\i\endcsname{white}
    }
    \foreach \ii/\jj/\vv in {#3}{
      \pgfmathtruncatemacro{\pct}{round(100*(\vv))}
      \ifnum\ii=\jj\relax 
        \expandafter\xdef\csname probgraph@fill@\ii\endcsname{black!\pct}
      \else 
        \probgraph@tmpa=\csname probgraph@r@\ii\endcsname\relax
        \probgraph@tmpb=\csname probgraph@r@\jj\endcsname\relax
        \ifdim\probgraph@tmpa<\probgraph@tmpb \probgraph@tmpb=\probgraph@tmpa\fi
        \edef\minrad{\the\probgraph@tmpb}
        \pgfmathsetlengthmacro{\lw}{\pgfkeysvalueof{/probgraph/edge width factor}*\minrad}
        \draw[
          line cap=round,
          line width=\lw,
          draw=black!\pct,
          opacity=\pgfkeysvalueof{/probgraph/edge opacity}
        ] (probgraph-v\ii) -- (probgraph-v\jj);
      \fi
    }
    \foreach[count=\i from 0] \p in {#1}{ 
      \filldraw[
        draw=black,
        line width=\pgfkeysvalueof{/probgraph/node line width},
        fill=\csname probgraph@fill@\i\endcsname
      ] (probgraph-v\i) circle[radius=\csname probgraph@r@\i\endcsname];
    }
    \foreach \p in {#2}{ 
      \node[font=\tiny, inner sep=0pt, outer sep=0pt] at (probgraph-v\p) {R};
    }
  \end{tikzpicture}
  \endgroup
}
\makeatother

\tikzset{
  vertex/.style={
    circle,
    draw,
    minimum size=0.2em,
    inner sep=0pt
  },
  fill overlay/.style = {circle, 
    fill=black, 
    minimum size=0.2em,
    inner sep=0pt
  },
  edge/.style={}
}

\newcommand{\DrawGraph}[2]{
\ifnum#1>0
  \begin{tikzpicture}[baseline={($(current bounding box.center)$)}]
    \def\graphradius{0.4cm}
    \foreach \i in {0, ..., \the\numexpr #1-1 \relax} {
      \pgfmathsetmacro{\angle}{90 - \i * (360/#1)}
      \node[vertex] (v\i) at (\angle:\graphradius) {};
    }
    \foreach \u / \v in {#2} {
      \ifnum\u=\v
        \node[fill overlay] at (v\u) {};
      \else
        \draw[edge] (v\u) -- (v\v);
      \fi
    }
  \end{tikzpicture}
\else 
\hspace{5mm}
\fi
}

\newcommand{\DrawRecursiveGraph}[3]{
\ifnum#1>0
  \begin{tikzpicture}[baseline={($(current bounding box.center)$)}]
    \def\graphradius{0.4cm}
    \foreach \i in {0, ..., \the\numexpr #1-1 \relax} {
      \pgfmathsetmacro{\angle}{90 - \i * (360/#1)}
      \node[vertex] (v\i) at (\angle:\graphradius) {};
      \coordinate (rv\i) at (\angle:0.56cm);
    }
    \foreach \u / \v in {#2} {
      \draw[edge] (v\u) -- (v\v);
    }
    \foreach \u in {#3} {
      \node[font=\tiny, inner sep=0pt, outer sep=0pt] at (rv\u) {$R$};
    }
  \end{tikzpicture}
\else
\hspace{5mm}
\fi
}

\newcommand{\ConjRow}[9]{%
  #1 & 
  \DrawGraph{6}{#2} & 
  \DrawGraph{6}{#3} & 
  \ProbGraph{#4}{#5}{#6} & 
  #7 & 
  #8 & 
  \ConjRowCont
}

\newcommand{\ConjRowCont}[2]{
   #2 \\
}

\newcommand{\GuessRow}[9]{%
  #1 & 
  \DrawGraph{6}{#2} & 
  \DrawGraph{6}{#3} & 
  \ProbGraph{#4}{#5}{#6} & 
  #7 & 
  #8 & 
  \GuessRowCont
}

\newcommand{\GuessRowCont}[2]{
   #2 \\
}

\newsavebox{\checktoolongbox}
\newcommand{\checktoolong}[2]{
  \begingroup
  \sbox{\checktoolongbox}{\(#1\)}
  \ifdim\wd\checktoolongbox<#2\relax
    \endgroup
    \(#1\)
  \else
    \endgroup
    \[#1\]
  \fi
}

\usepackage{longtable}
\usepackage{array}
\newcolumntype{C}[1]{>{\centering\arraybackslash}m{#1}}

\newcommand{\TableRow}[8]{
#1 & \DrawGraph{#2}{#3} & \DrawGraph{#2}{#4} & \DrawGraph{#5}{#6} & $#7$ & $#8$ & \TableRowContinued
}

\newcommand{\TableRowContinued}[2]{
$#1$ & $#2$ \\
}

\newcommand{\jared}[1]{{}{\sf\textcolor{bluegreen}{JL: #1}}}
\newcommand{\XL}[1]{{}{\sf\textcolor{magenta}{XL: #1}}}

\begin{document}

\title{The inducibility of $\mathbf{6}$-vertex graphs}

\author[L.Bodnár]{Levente Bodnár}
\address{(LB, JG, JL, OP, SS): Mathematics Institute and DIMAP, University of Warwick, Coventry, UK}
\email{\{Levente.Bodnar, Jun.Gao, Jared.Leon, O.Pikhurko, Shumin.Sun\}@warwick.ac.uk}
\thanks{(LB, JG, OP, SS): research supported by ERC Advanced Grant 101020255.}

\author[J.Gao]{Jun Gao}

\author[J.León]{Jared León}
\thanks{(JL): research supported by the Warwick Mathematics Institute Centre for Doctoral Training}

\author[X.Liu]{Xizhi Liu}
\address{(XL): School of Mathematical Sciences, USTC, Hefei, China.}
\email{liuxizhi@ustc.edu.cn}
\thanks{(XL): research supported by the Excellent Young Talents Program (Overseas) of the National Natural Science Foundation of China}

\author[O.Pikhurko]{Oleg Pikhurko}
\author[S.Sun]{Shumin Sun}



  

\newcounter{Solved}
\setcounter{Solved}{36} 
\newcounter{NewSolved}
\setcounter{NewSolved}{30}
\newcounter{Conjectures}
\setcounter{Conjectures}{12}

\begin{abstract}
  The \emph{inducibility constant} $\lambda_{F}$ of a graph $F$ is the asymptotically maximum induced density of $F$ in a growing sequence of graphs. 
  This paper systematically investigates the case when $F$ has 6 vertices (and there are 78 cases to consider up to isomorphism and complementation). We show that flag algebras can compute the sharp upper bound on $\lambda_F$ in \theSolved\  cases of which, as far as the authors know, \theNewSolved\ are new results. 
  In each of the solved cases, we also prove results about the structure of large (almost) extremal graphs. In particular, we establish perfect stability in all \the\numexpr\value{Solved}-4\relax\  cases when the extremal construction has no quasirandom parts. We also present conjectures about the value of $\lambda_{F}$ for \theConjectures\ further cases (where the upper and lower bounds are very close to each other).
\end{abstract}

\maketitle




\section{Introduction}
\label{sec:introduction}

For graphs $F$ and $G$ with $\kappa\le n$ vertices
respectively, let $P(F,G)$ denote the number of $\kappa$-subsets of
$V(G)$ that induce a graph isomorphic to $F$ and let
$p(F,G):=P(F,G)/\binom n{\kappa}$ be the \emph{(induced) density} of $F$ in~$G$.
For a given graph $F$, the \emph{$F$-inducibility problem} (or the \emph{$F$-problem} for short) asks for $p(F,n)$,
the maximum of $p(F,G)$ over all graphs $G$ with $n$
vertices. By an easy averaging argument, one can show that the limit
\[
\lambda_{F}:=\lim_{n\to\infty} p(F,n)
\]
 exists; we call it the
\emph{inducibility constant} of~$F$. 

The inducibility problem has drawn a great amount of interest since it
was introduced by Pippenger and Golumbic~\cite{PippengerGolumbic75} in
1975. For some sample of results, see
e.g.~\cite{Siran84,BollobasNaraTachibana86,BrownSidorenko94,BollobasEgawaHarrisJin95,Hirst14,HatamiHirstNorine14,EvenZoharLinial15,BaloghHuLidickyPfender16,HefetzTyomkyn18,KralNorinVolec19,Yuster19,LidickyMattesPfender23,Ueltzen24}.
Also, it was considered for many other types of structures: for example, oriented
graphs~\cite{Huang14,ChoiLidickyPfender20,BozykGrzesikKielak22,HuMaNorinWu24}, 
hypercubes~\cite{GoldwasserHansen21,GoldwasserHansen24,BodnarPikhurko25}, various tree-like structures~\cite{AlonNavesSudakov16,CzabarkaSzekelyWagner17,CzabarkaDossouOlorySzekelyWagner20,BroschPuges25}, 
constrained graphs (for example, triangle-free graphs~\cite{Gyori89,Grzesik12,HatamiHladkyKralNorine13}
), etc.

Since the inducibility constant
$\lambda_{F}$ does not change if we replace $F$ by its complement $\O F$, it is
enough to consider only one graph from each complementary pair. If $F$ is complete partite then a result by Brown and
Sidorenko~\cite[Proposition~1]{BrownSidorenko94} implies that, in order to determine
the value of $p(F,n)$, it is enough to consider only complete partite
graphs on $[n]$ and the problem in the limit reduces to maximising 
some polynomial on the space of part ratios. 
\hide{If the latter
is fully solved, with the description of all extremal ratios, then the
method of Liu, Pikhurko, Sharifzadeh and
Staden~\cite{LiuPikhurkoSharifzadehStaden23} can often be applied to
describe the structure of large (almost) extremal graphs.}%

Each $3$-vertex graph or its complement is complete partite, so they are covered e.g.\ by the above result of Brown and
Sidorenko~\cite{BrownSidorenko94}.

All $4$-vertex graphs $F$ were resolved by the results
in~\cite{BollobasNaraTachibana86,Exoo86,BrownSidorenko94,Hirst14} except when $F=P_4$ is
the $4$-vertex path. The best known lower bound
$\lambda_{P_4}\ge 1173/5824 = 0.2014\ldots$ is due to~Even-Zohar and
Linial~\cite{EvenZoharLinial15} while the best known upper bound
$\lambda_{P_4}\le 0.204513\ldots$ comes from flag algebras.

Even-Zohar and Linial~\cite[Table 2] {EvenZoharLinial15} produced the summary of
known and new results for $5$-vertex graphs $F$.
After the publication of~\cite{EvenZoharLinial15}, the values of inducibility constant for some new $5$-vertex graphs $F$ were determined in~\cite{PikhurkoSliacanTyros19,LiuMubayiReiher23,LiuPikhurkoSharifzadehStaden23,BodnarPikhurko26} but there are still $5$ (non-equivalent) remaining open cases, described in~\cite{BodnarPikhurko26}.

In this paper, we systematically investigate the inducibility problem
for graphs on $6$ vertices, using flag algebras to prove upper bounds. There are $156$ non-isomorphic $6$-vertex graphs, which we
group into $78$ complementary pairs $\{F_i,\O{F}_i\}$ for $0\le i\le 77$. (Note that, since $\binom{6}{2}$ is odd, no $6$-vertex graph is self-complementary.)
Our results are as follows; they are also summarised in Table~\ref{ta:AllResults} below.

Using the package \texttt{FlagAlgebraToolbox} of Bodn\'ar~\cite{Bodnar26}, we could produce the
sharp upper bound on  the inducibility constant $\lambda_{F}$ in \theSolved\  cases (out of $78$) of which \theNewSolved\ cases are, as far as we can see, new results.
There are three further cases where the package did not give the required sharp result but the exact value was previously known by other methods, and we include these cases in the table: $F_1$, whose inducibility constant was determined by Liu, Mubayi and Reiher~\cite[Theorem~1.13]{LiuMubayiReiher23}, the cycle case $F_{51}=C_6$, treated by Brandt, Lidick\'y and Pfender as recorded in Brandt's thesis~\cite{Brandt16}, and the net graph $F_{43}$, determined by Blumenthal and Phillips~\cite{BlumenthalPhillips25}.
In all cases solved here, we also describe almost extremal graphs of large order $n$ within $o(n^2)$ adjacency edits. It happens in each case that, informally speaking, there is a unique limit construction. The latter statement can be formalised using the language of graphons (see Lov\'asz' book~\cite{Lovasz12}), where a graphon $W$ is called \emph{$F$-extremal} if the appropriately defined density $p(F,W)$ is equal to its maximum possible value~$\lambda_{F}$. Then, in all cases solved here, the extremal graphon $W$ is unique and happens to be a step graphon (with finitely many steps).

If each value of the extremal graphon $W$ is 0 or 1 (which happens in \the\numexpr\value{Solved}-4\relax\ cases) then it is natural to represent constructions as blowups of a \emph{pattern} $B$ (a graph, with loops allowed). In all of these  \the\numexpr\value{Solved}-4\relax\  cases, we prove furthermore that the so-called \emph{perfect $B$-stability} holds using criteria from~\cite{PikhurkoSliacanTyros19}  (see Section~\ref{sec:perfect-stability} for details). This is a rather strong structural property that, in particular, implies that every large extremal graph is exactly a blowup of~$B$. Thus in order to determine $p(F,n)$ exactly for large $n$, it is enough to maximise the density of $F$ in an $n$-vertex blowup of $B$, which amounts to maximising an explicit polynomial in integers summing up to~$n$. We do not attempt to do this integer optimisation. However, we do determine the (unique in all solved cases) limiting vector of optimal part ratios. 

There are four cases, namely, $i\in\{13, 70, 75, 77\}$ in Table~\ref{ta:AllResults}, when almost extremal graphs are quasirandom balanced bipartite graphs of appropriate edge density $p_i$, and we show that every large almost optimal graph is close in the edit distance to this form. 
Thus, in each of these cases, the (unique) extremal graphon $W$ consists of two parts $V_0$ and $V_1$, each of measure $1/2$, with $W$ being $0$ on all pairs inside a part and $p_i$ on all cross-pairs. It is unclear how to define perfect stability when there are  quasirandom pairs. Nonetheless, we also establish some finer structure results about almost extremal graphs that, in particular, imply the uniqueness of the extremal graphon.
\hide{Nonetheless, we also establish some finer structure results about large extremal graphs in each of these four cases (see Section~\ref{sec:quasirandom}).}

There are eleven $6$-vertex graphs that are complete partite. 
\hide{part sizes 111111 (complement of 6), 15, 24, 33, 114, 123, 222, 1113, 1122, 11112}%
The package could determine the inducibility constant (and prove results enough for establishing perfect stability) for all of them except for $\O{F}_1=T_{2,1,1,1,1}$,
where $T_{\ell_0,\ldots,\ell_{m-1}}$ denote the complete $m$-partite graph with parts of sizes $\ell_0,\ldots,\ell_{m-1}$. 
In the case of $F_1$, the inducibility constant was determined by Liu, Mubayi and Reiher~\cite[Theorem~1.13]{LiuMubayiReiher23} and perfect stability was proved by Liu, Ma and Zhu~\cite{LiuMaZhu26}, and we note this result in our table.  Given the general result of Brown and Sidorenko~\cite{BrownSidorenko94} for complete partite $F$,  the inducibility constant (or even the whole inducibility function $p(F,n)$) was computed without flag algebras for many other $6$-vertex complete partite graphs~$F$. Of course, the problem for the empty and complete graphs, $F_0=T_{6}$ and $\O{F}_0=T_{1,1,1,1,1,1}$, is trivial. 
Already Pippenger and Golumbic~\cite[Theorem~10]{PippengerGolumbic75} determined the inducibility function of $T_{\ell,\ell}$ for every $\ell$. Brown and Sidorenko~\cite[Proposition~2]{BrownSidorenko94} proved that, for every $n\ge \ell_0+\ell_1$, the value of $p(T_{\ell_0,\ell_1},n)$ is attained by a complete partite graph with exactly two parts while the corresponding perfect stability was proved by Liu, Pikhurko, Sharifzadeh and
Staden~\cite[Theorem 1.3]{LiuPikhurkoSharifzadehStaden23}; in particular, the solution for $\O{F}_{32}=T_{5,1}$, $\O{F}_{56}=T_{4,2}$, $\O{F}_{53}=T_{3,3}$ easily follows.  Bollob{\'a}s, Egawa, Harris and Jin~\cite{BollobasEgawaHarrisJin95} determined the inducibility function for $\O{F}_8=T_{2,2,2}$ with perfect stability proved by Liu et al \cite[Theorem 1.4]{LiuPikhurkoSharifzadehStaden23}. The value of the inducibility constant of $\O{F}_3=T_{2,2,1,1}$ was determined by Yuster~\cite{Yuster26} with perfect stability proved by Liu, Ma and Zhu~\cite{LiuMaZhu26}. Solutions to the three remaining 6-vertex complete partite graphs $F$, namely $\O{F}_{37}=T_{4,1,1}$, $\O{F}_{17}=T_{3,2,1}$ and $\O{F}_6=T_{3,1,1,1}$, do not seem to appear anywhere in the literature.

Table~\ref{ta:AllResults} below summarises our findings (and also lists the answer in the special case $i=1$ for completeness). Recall that all graphs are grouped into $78$ complementary pairs indexed by $i$ running from $0$ to $77$. Row~$i$ of the table plots both $F_i$ and its complement.
The fifth column presents either the exact value of $\lambda_{F_i}$ or an upper
bound on it. The seventh column lists $N$, which is the number of vertices in $0$-flags that we needed in order to prove the upper bound (see Section~\ref{sec:flag-algebras} for details).

If we have a gap between the bounds, then we list the upper bound as the floating-point real that was returned by computer using $8$-vertex $0$-flags (with $N=8$ being the maximum value which is feasible for computer calculations). While it is routine to convert computer's output into a mathematical proof with only very small loss in the bound, we do not do this here as this would not provide any new insights.

If we are able to prove the sharp upper bound on $\lambda_{F_i}$ via flag algebras, then we also list the pattern $B$ (which we prove to be unique in all solved cases) whose blowups give the matching lower bound. For the known recursive cases $F_{43}$ and $F_{51}$, the fourth column instead displays the recursive graphon $W$ with equal parts and with $R$ indicating recursion inside the corresponding part. That is, $W$ is a self-similar graphon, where the domain is partitioned into a (possibly non-uniform) grid and some of the diagonal blocks, namely those indicated by $R$, map to linearly rescaled copies of $W$ itself. Optimal limiting part ratios (which happen to be unique up to an automorphism of $B$ in all solved cases) are not listed in the table due to space limitations but can be found in the corresponding theorems of Section~\ref{sec:exact-results}. When the value of $\lambda_{F_i}$ is an irrational number, we list it as a polynomial of $\alpha$ (where $\alpha$ is usually an optimal ratio for some collection of parts of $B$) with the sixth column listing the minimal monic polynomial having $\alpha$ as a root. The last column indicates whether perfect
stability was established (which, remarkably, is `Yes' in all solved cases apart from the four special cases with quasirandom pairs).
For the rows where this entry is a pattern $B$, extremal constructions can be obtained from the graph $B$ listed in the
fourth column as follows. We replace each hollow vertex with an
independent set and each solid vertex with a clique (and $R$ means a recursive part, see Section~\ref{sec:conjectures} for details). For each edge, we
add either a complete bipartite graph or a binomial random
bipartite graph between the corresponding sets. The precise sizes of
the independent sets and cliques, as well as further details of the
construction, are provided in Section~\ref{sec:exact-results}.

As far as we know, apart from the complete partite graphs and their complements (that we discussed above) and the previously known values for $F_{43}$ and $F_{51}$, all other cases solved here are new. 

There are \theConjectures\ further cases where the upper and lower bounds come very close to each other, discussed in Section~\ref{sec:conjectures}. We conjecture that each of the constructions presented in Section~\ref{sec:conjectures} gives the exact value of the inducibility constant.

In the remaining cases, we did computer search for constructions, using gradient descent methods in the space of step graphons with at most 4 parts and $\{0,1\}$-valued step graphons with at most 6 steps (equivalently, patterns with at most 6 vertices) also allowing for recursive parts. Table~\ref{tab:guess6} of Section~\ref{sec:Concluding} presents the found lower bounds. 
\hide{
\begin{organization*}
  The remainder of this paper is organised as follows. We introduce
  some necessary definitions and auxiliary results in
  Section~\ref{sec:preliminaries}, including brief overviews of the
  flag algebra method and the perfect
  stability. Section~\ref{sec:exact-results} contains precise
  statements and proofs of the exact results for $F_i$ whose value of
  $\lambda_{F_i}$ is determined.
  In Section~\ref{sec:quasirandom}, we \ldots
  In Section~\ref{sec:conjectures}, we make several
  conjectures concerning the remaining unsolved cases.
\end{organization*}
}
\renewcommand{\arraystretch}{1.8}
\setlength{\tabcolsep}{3pt}
\begin{longtable}{c c c c c c c c}
\caption{Inducibility constants of $6$-vertex graphs\label{ta:AllResults}} \\
\HIDE{\hline
  $i$ & $F_i$ & $\overline{F_i}$ & $B/W$ & $\lambda_{F_i}$  & $\alpha$ & $N$ & P.~Stability\\
  \hline
  \endfirsthead 
  \hline
  $i$ & $F_i$ & $\overline{F_i}$ & $B/W$ & $\lambda_{F_i}$  & $\alpha$ & $N$ & P.~Stability\\
  \hline
  \endhead 
  \TableRow{0}{6}{}{0/1,0/2,0/3,0/4,0/5,1/2,1/3,1/4,1/5,2/3,2/4,2/5,3/4,3/5,4/5}{1}{}{1}{-}{6}{Y}
  \TableRow{1}{6}{0/1}{0/2,0/3,0/4,0/5,1/2,1/3,1/4,1/5,2/3,2/4,2/5,3/4,3/5,4/5}{13}{0/0,1/1,2/2,3/3,4/4,5/5,6/6,7/7,8/8,9/9,10/10,11/11,12/12}{\frac{178200}{371293} \mbox{ (\cite{LiuMubayiReiher23})}}{-}{-}{Y \mbox{ (\cite{LiuMaZhu26})}}
  \TableRow{2}{6}{0/1,0/4}{0/2,0/3,0/5,1/2,1/3,1/4,1/5,2/3,2/4,2/5,3/4,3/5,4/5}{5}{0/1,2/3}{-\frac{255}{8} \alpha^{2} + \frac{591}{10} \alpha - \frac{129}{5}}{\alpha^{3} + 2 \alpha^{2} - \frac{32}{5} \alpha + \frac{16}{5}}{7}{Y}
  \TableRow{3}{6}{0/4,1/5}{0/1,0/2,0/3,0/5,1/2,1/3,1/4,2/3,2/4,2/5,3/4,3/5,4/5}{6}{0/0,1/1,2/2,3/3,4/4,5/5}{\frac{25}{72} \mbox{ (\cite{Yuster26})}}{-}{10}{Y\mbox{ (\cite{LiuMaZhu26})}}
  \TableRow{4}{6}{0/1,0/2,0/4}{0/3,0/5,1/2,1/3,1/4,1/5,2/3,2/4,2/5,3/4,3/5,4/5}{3}{1/2}{\frac{40}{243}}{-}{6}{Y}
  \TableRow{5}{6}{0/1,0/4,1/5}{0/2,0/3,0/5,1/2,1/3,1/4,2/3,2/4,2/5,3/4,3/5,4/5}{0}{}{\le 0.1121327654}{-}{8}{-}
  \TableRow{6}{6}{0/1,0/5,1/5}{0/2,0/3,0/4,1/2,1/3,1/4,2/3,2/4,2/5,3/4,3/5,4/5}{2}{0/0}{\frac{5}{16}}{-}{6}{Y}
  \TableRow{7}{6}{0/2,0/4,1/5}{0/1,0/3,0/5,1/2,1/3,1/4,2/3,2/4,2/5,3/4,3/5,4/5}{9}{0/2,0/1,1/2,3/4,3/5,4/5,6/7,6/8,7/8}{\frac{160}{729}}{-}{7}{Y}
  \TableRow{8}{6}{0/2,1/3,4/5}{0/1,0/3,0/4,0/5,1/2,1/4,1/5,2/3,2/4,2/5,3/4,3/5}{3}{0/0,1/1,2/2}{\frac{10}{81}\mbox{ (\cite{BollobasEgawaHarrisJin95})}}{-}{6}{Y\mbox{ (\cite{LiuPikhurkoSharifzadehStaden23})}}
  \TableRow{9}{6}{0/1,0/2,0/4,0/5}{0/3,1/2,1/3,1/4,1/5,2/3,2/4,2/5,3/4,3/5,4/5}{0}{}{\le 0.1677211391}{-}{8}{-}
  \TableRow{10}{6}{0/1,0/2,0/4,1/5}{0/3,0/5,1/2,1/3,1/4,2/3,2/4,2/5,3/4,3/5,4/5}{0}{}{\le 0.1026174578}{-}{8}{-}
  \TableRow{11}{6}{0/1,0/3,0/5,1/5}{0/2,0/4,1/2,1/3,1/4,2/3,2/4,2/5,3/4,3/5,4/5}{0}{}{\le 0.1265831119}{-}{8}{-}
  \TableRow{12}{6}{0/2,0/3,0/4,1/5}{0/1,0/5,1/2,1/3,1/4,2/3,2/4,2/5,3/4,3/5,4/5}{3}{0/0,1/2}{\frac{40}{243}}{-}{6}{Y}
  \TableRow{13}{6}{0/2,0/4,1/2,1/5}{0/1,0/3,0/5,1/3,1/4,2/3,2/4,2/5,3/4,3/5,4/5}{2}{0/1}{\frac{5440}{729} \alpha - \frac{2560}{729}}{\alpha^{2} - \frac{7}{3} \alpha + \frac{8}{9}}{6}{N/A}
  \TableRow{14}{6}{0/2,0/5,1/2,1/5}{0/1,0/3,0/4,1/3,1/4,2/3,2/4,2/5,3/4,3/5,4/5}{3}{1/2}{\frac{10}{81}}{-}{6}{Y}
  \TableRow{15}{6}{0/3,0/4,1/2,1/5}{0/1,0/2,0/5,1/3,1/4,2/3,2/4,2/5,3/4,3/5,4/5}{4}{0/2,1/3}{\frac{45}{256}}{-}{6}{Y}
  \TableRow{16}{6}{0/1,0/2,1/3,4/5}{0/3,0/4,0/5,1/2,1/4,1/5,2/3,2/4,2/5,3/4,3/5}{0}{}{\le 0.0685473237}{-}{8}{-}
  \TableRow{17}{6}{0/1,0/3,1/3,4/5}{0/2,0/4,0/5,1/2,1/4,1/5,2/3,2/4,2/5,3/4,3/5}{3}{0/0,1/1,2/2}{\frac{40}{81}}{-}{6}{Y}
  \TableRow{18}{6}{0/1,0/2,0/4,0/5,1/5}{0/3,1/2,1/3,1/4,2/3,2/4,2/5,3/4,3/5,4/5}{0}{}{\le 0.0966144422}{-}{8}{-}
  \TableRow{19}{6}{0/1,0/2,0/3,0/4,1/5}{0/5,1/2,1/3,1/4,2/3,2/4,2/5,3/4,3/5,4/5}{0}{}{\le 0.0873819112}{-}{8}{-}
  \TableRow{20}{6}{0/1,0/2,0/4,1/2,1/5}{0/3,0/5,1/3,1/4,2/3,2/4,2/5,3/4,3/5,4/5}{0}{}{\le 0.0543206983}{-}{8}{-}
  \TableRow{21}{6}{0/1,0/2,0/5,1/2,1/5}{0/3,0/4,1/3,1/4,2/3,2/4,2/5,3/4,3/5,4/5}{6}{0/1,0/2,0/3,0/4,1/2,1/3,1/4,2/3,2/4,3/4}{\frac{128}{675}}{-}{8}{Y}
  \TableRow{22}{6}{0/1,0/3,0/4,1/2,1/5}{0/2,0/5,1/3,1/4,2/3,2/4,2/5,3/4,3/5,4/5}{0}{}{\le 0.0617310538}{-}{8}{-}
  \TableRow{23}{6}{0/2,0/4,0/5,1/2,1/5}{0/1,0/3,1/3,1/4,2/3,2/4,2/5,3/4,3/5,4/5}{0}{}{\le 0.1116212593}{-}{8}{-}
  \TableRow{24}{6}{0/2,0/3,0/4,1/2,1/5}{0/1,0/5,1/3,1/4,2/3,2/4,2/5,3/4,3/5,4/5}{6}{0/1,1/2,2/3,3/4,4/5,5/0}{\frac{5}{54}}{-}{6}{Y}
  \TableRow{25}{6}{0/1,0/4,0/5,1/3,4/5}{0/2,0/3,1/2,1/4,1/5,2/3,2/4,2/5,3/4,3/5}{0}{}{\le 0.1061720738}{-}{8}{-}
  \TableRow{26}{6}{0/1,0/2,0/5,1/3,4/5}{0/3,0/4,1/2,1/4,1/5,2/3,2/4,2/5,3/4,3/5}{0}{}{\le 0.0432004682}{-}{8}{-}
  \TableRow{27}{6}{0/1,0/2,0/3,1/3,4/5}{0/4,0/5,1/2,1/4,1/5,2/3,2/4,2/5,3/4,3/5}{8}{0/1,1/2,2/3,3/0,4/5,5/6,6/7,7/4,0/0,1/1,2/2,3/3,4/4,5/5,6/6,7/7}{\frac{135}{1024}}{-}{7}{Y}
  \TableRow{28}{6}{0/2,0/3,1/4,1/5,4/5}{0/1,0/4,0/5,1/2,1/3,2/3,2/4,2/5,3/4,3/5}{3}{0/0,1/2}{\frac{15}{64}}{-}{6}{Y}
  \TableRow{29}{6}{0/3,0/4,1/3,1/5,4/5}{0/1,0/2,0/5,1/2,1/4,2/3,2/4,2/5,3/4,3/5}{0}{}{\le 0.0155317467}{-}{8}{-}
  \TableRow{30}{6}{0/2,0/4,1/3,1/5,4/5}{0/1,0/3,0/5,1/2,1/4,2/3,2/4,2/5,3/4,3/5}{0}{}{\le 0.0432388288}{-}{8}{-}
  \TableRow{31}{6}{0/2,0/3,1/2,1/3,4/5}{0/1,0/4,0/5,1/4,1/5,2/3,2/4,2/5,3/4,3/5}{3}{0/0,1/2}{\frac{10}{81}}{-}{6}{Y}
  \TableRow{32}{6}{0/1,0/2,0/3,0/4,1/2,1/3,1/4,2/3,2/4,3/4}{0/5,1/5,2/5,3/5,4/5}{2}{0/0,1/1}{\frac{20}{3} \alpha^{2} - \frac{20}{3} \alpha + \frac{4}{3}\mbox{ (\cite{BrownSidorenko94})}}{\alpha^{4} - 2 \alpha^{3} + \frac{7}{3} \alpha^{2} - \frac{4}{3} \alpha + \frac{1}{6}}{7}{Y\mbox{ (\cite{LiuPikhurkoSharifzadehStaden23})}}
  \TableRow{33}{6}{0/1,0/2,0/4,0/5,1/2,1/5}{0/3,1/3,1/4,2/3,2/4,2/5,3/4,3/5,4/5}{0}{}{\le 0.0759651322}{-}{8}{-}
  \TableRow{34}{6}{0/1,0/2,0/3,0/4,1/2,1/5}{0/5,1/3,1/4,2/3,2/4,2/5,3/4,3/5,4/5}{0}{}{\le 0.0650311105}{-}{8}{-}
  \TableRow{35}{6}{0/2,0/4,0/5,1/2,1/3,1/5}{0/1,0/3,1/4,2/3,2/4,2/5,3/4,3/5,4/5}{0}{}{\le 0.0668123579}{-}{8}{-}
  \TableRow{36}{6}{0/2,0/3,0/5,1/2,1/3,1/5}{0/1,0/4,1/4,2/3,2/4,2/5,3/4,3/5,4/5}{4}{0/1,2/3}{\frac{25}{6} \alpha^{2} - \frac{25}{6} \alpha + \frac{5}{6}}{\alpha^{4} - 2 \alpha^{3} + \frac{7}{3} \alpha^{2} - \frac{4}{3} \alpha + \frac{1}{6}}{7}{Y}
  \TableRow{37}{6}{2/3,2/4,2/5,3/4,3/5,4/5}{0/1,0/2,0/3,0/4,0/5,1/2,1/3,1/4,1/5}{2}{0/0}{\frac{80}{243}}{-}{6}{Y}
  \TableRow{38}{6}{0/1,0/2,0/3,0/4,1/5,4/5}{0/5,1/2,1/3,1/4,2/3,2/4,2/5,3/4,3/5}{5}{0/1,1/2,2/3,3/4,4/0}{\frac{72}{625}}{-}{6}{Y}
  \TableRow{39}{6}{0/1,0/3,0/4,0/5,1/3,4/5}{0/2,1/2,1/4,1/5,2/3,2/4,2/5,3/4,3/5}{0}{}{\le 0.0628954466}{-}{8}{-}
  \TableRow{40}{6}{0/1,0/2,0/4,0/5,1/3,4/5}{0/3,1/2,1/4,1/5,2/3,2/4,2/5,3/4,3/5}{0}{}{\le 0.0655364348}{-}{8}{-}
  \TableRow{41}{6}{0/1,0/2,0/5,1/4,1/5,4/5} {0/3,0/4,1/2,1/3,2/3,2/4,2/5,3/4,3/5}{0}{}{\le 0.0772859012}{-}{8}{-}
  \TableRow{42}{6}{0/1,0/2,0/3,1/4,1/5,4/5}{0/4,0/5,1/2,1/3,2/3,2/4,2/5,3/4,3/5}{0}{}{\le 0.0657267553}{-}{8}{-}
  43 & \DrawGraph{6}{0/1,0/2,0/5,1/3,1/5,4/5} & \DrawGraph{6}{0/3,0/4,1/2,1/4,2/3,2/4,2/5,3/4,3/5} & \DrawRecursiveGraph{6}{0/1,0/2,0/5,1/3,1/5,4/5}{0,1,2,3,4,5} & $\frac{24}{1555}\mbox{ (\cite{BlumenthalPhillips25})}$ & $-$ & $-$ & $-$ \\
  \TableRow{44}{6}{0/1,0/3,0/4,1/3,1/5,4/5}{0/2,0/5,1/2,1/4,2/3,2/4,2/5,3/4,3/5}{0}{}{\le 0.0408582489}{-}{8}{-}
  \TableRow{45}{6}{0/1,0/2,0/4,1/3,1/5,4/5}{0/3,0/5,1/2,1/4,2/3,2/4,2/5,3/4,3/5}{0}{}{\le 0.0481709303}{-}{8}{-}
  \TableRow{46}{6}{0/1,0/2,0/3,1/3,1/5,4/5}{0/4,0/5,1/2,1/4,2/3,2/4,2/5,3/4,3/5}{0}{}{\le 0.0602613357}{-}{8}{-}
  \TableRow{47}{6}{0/2,0/3,0/4,1/3,1/5,4/5}{0/1,0/5,1/2,1/4,2/3,2/4,2/5,3/4,3/5}{0}{}{\le 0.0432004547}{-}{8}{-}
  \TableRow{48}{6}{0/1,0/2,0/3,1/2,1/3,4/5}{0/4,0/5,1/4,1/5,2/3,2/4,2/5,3/4,3/5}{0}{}{\le 0.2178297374}{-}{8}{-}
  \TableRow{49}{6}{0/3,0/4,0/5,1/2,1/3,4/5}{0/1,0/2,1/4,1/5,2/3,2/4,2/5,3/4,3/5}{6}{0/0,1/1,2/2,3/3,4/4,5/5,0/1,1/2,2/3,3/4,4/5,5/0}{\frac{5}{54}}{-}{6}{Y}
  \TableRow{50}{6}{0/2,0/3,0/5,1/2,1/3,4/5}{0/1,0/4,1/4,1/5,2/3,2/4,2/5,3/4,3/5}{6}{0/1,1/2,2/3,3/4,4/5,5/0}{\frac{5}{54}}{-}{6}{Y}
  51 & \DrawGraph{6}{0/2,0/4,1/2,1/3,3/5,4/5} & \DrawGraph{6}{0/1,0/3,0/5,1/4,1/5,2/3,2/4,2/5,3/4} & \DrawRecursiveGraph{6}{0/2,0/4,1/2,1/3,3/5,4/5}{0,1,2,3,4,5} & $\frac{24}{1555}\mbox{ (\cite{Brandt16})}$ & $-$ & $-$ & $-$ \\
  \TableRow{52}{6}{0/1,0/3,0/4,0/5,1/3,1/4,1/5,3/5,4/5}{0/2,1/2,2/3,2/4,2/5,3/4}{0}{}{\le 0.2067706556}{-}{8}{-}
  \TableRow{53}{6}{0/4,0/5,1/2,1/3,2/3,4/5}{0/1,0/2,0/3,1/4,1/5,2/4,2/5,3/4,3/5}{2}{0/0,1/1}{\frac{5}{16}\mbox{ (\cite{BrownSidorenko94})}}{-}{6}{Y\mbox{ (\cite{LiuPikhurkoSharifzadehStaden23})}}
  \TableRow{54}{6}{0/1,0/2,0/4,0/5,1/2,1/3,1/5}{0/3,1/4,2/3,2/4,2/5,3/4,3/5,4/5}{0}{}{\le 0.0469245022}{-}{8}{-}
  \TableRow{55}{6}{0/1,0/2,0/3,0/5,1/2,1/3,1/5}{0/4,1/4,2/3,2/4,2/5,3/4,3/5,4/5}{4}{0/1,0/0,2/3,3/3}{\frac{288}{125} \alpha^{2} - \frac{288}{125} \alpha + \frac{288}{625}}{\alpha^{4} - 2 \alpha^{3} + \frac{7}{3} \alpha^{2} - \frac{4}{3} \alpha + \frac{1}{6}}{7}{Y}
  \TableRow{56}{6}{0/1,2/3,2/4,2/5,3/4,3/5,4/5}{0/2,0/3,0/4,0/5,1/2,1/3,1/4,1/5}{2}{0/0,1/1}{\frac{15}{32}\mbox{ (\cite{BrownSidorenko94})}}{-}{6}{Y\mbox{ (\cite{LiuPikhurkoSharifzadehStaden23})}}
  \TableRow{57}{6}{0/1,0/3,0/4,0/5,1/4,1/5,4/5}{0/2,1/2,1/3,2/3,2/4,2/5,3/4,3/5}{0}{}{\le 0.1297872777}{-}{8}{-}
  \TableRow{58}{6}{0/1,0/2,0/3,0/5,1/4,1/5,4/5}{0/4,1/2,1/3,2/3,2/4,2/5,3/4,3/5}{4}{0/1,1/2,2/3,1/1,2/2}{\frac{45}{512}}{-}{6}{Y}
  \TableRow{59}{6}{0/1,0/3,0/4,0/5,1/3,1/5,4/5}{0/2,1/2,1/4,2/3,2/4,2/5,3/4,3/5}{0}{}{\le 0.0582381881}{-}{8}{-}
  \TableRow{60}{6}{0/1,0/2,0/4,0/5,1/3,1/5,4/5}{0/3,1/2,1/4,2/3,2/4,2/5,3/4,3/5}{0}{}{\le 0.0311475323}{-}{8}{-}
  \TableRow{61}{6}{0/1,0/2,0/3,0/4,1/3,1/5,4/5}{0/5,1/2,1/4,2/3,2/4,2/5,3/4,3/5}{0}{}{\le 0.0511213762}{-}{8}{-}
  \TableRow{62}{6}{0/1,0/3,0/4,0/5,1/2,1/3,4/5}{0/2,1/4,1/5,2/3,2/4,2/5,3/4,3/5}{0}{}{\le 0.0650311132}{-}{8}{-}
  \TableRow{63}{6}{0/1,0/2,0/3,0/5,1/2,1/3,4/5}{0/4,1/4,1/5,2/3,2/4,2/5,3/4,3/5}{0}{}{\le 0.0905122900}{-}{8}{-}
  \TableRow{64}{6}{0/3,0/4,0/5,1/2,1/4,1/5,4/5}{0/1,0/2,1/3,2/3,2/4,2/5,3/4,3/5}{6}{0/0,1/1,2/2,3/3,4/4,5/5,0/1,1/2,2/3,3/4,4/5,5/0}{\frac{5}{108}}{-}{6}{Y}
  \TableRow{65}{6}{0/2,0/4,0/5,1/2,1/4,1/5,4/5}{0/1,0/3,1/3,2/3,2/4,2/5,3/4,3/5}{12}{0/0,1/1,2/2,3/3,4/4,5/5,0/1,1/2,2/3,3/4,4/5,5/0,0/3,1/4,2/5,6/6,7/7,8/8,9/9,10/10,11/11,6/7,7/8,8/9,9/10,10/11,11/6,6/9,7/10,8/11}{\frac{50}{27}\alpha^{2}-\frac{50}{27}\alpha+\frac{10}{27}}{\alpha^{4}-2\alpha^{3}+\frac{7}{3}\alpha^{2}-\frac{4}{3}\alpha+\frac{1}{6}}{7}{Y}
  \TableRow{66}{6}{0/1,0/2,0/4,0/5,1/2,1/4,1/5,4/5}{0/3,1/3,2/3,2/4,2/5,3/4,3/5}{0}{}{\le 0.1120734210}{-}{8}{-}
  \TableRow{67}{6}{0/2,0/3,0/5,1/2,1/4,1/5,4/5}{0/1,0/4,1/3,2/3,2/4,2/5,3/4,3/5}{0}{}{\le 0.0390894387}{-}{8}{-}
  \TableRow{68}{6}{0/2,0/3,0/5,1/2,1/3,1/5,4/5}{0/1,0/4,1/4,2/3,2/4,2/5,3/4,3/5}{5}{0/1,1/2,2/3,3/4,4/0}{\frac{72}{625}}{-}{6}{Y}
  \TableRow{69}{6}{0/2,0/3,0/4,1/2,1/3,1/5,4/5}{0/1,0/5,1/4,2/3,2/4,2/5,3/4,3/5}{5}{0/1,1/2,2/3,3/4,4/0}{\frac{72}{625}}{-}{6}{Y}
  \TableRow{70}{6}{0/1,0/2,0/3,0/4,1/5,3/5,4/5}{0/5,1/2,1/3,1/4,2/3,2/4,2/5,3/4}{2}{0/1}{\frac{12353145}{67108864}}{-}{6}{N/A}
  \TableRow{71}{6}{0/1,0/2,0/3,0/4,1/2,3/5,4/5}{0/5,1/3,1/4,1/5,2/3,2/4,2/5,3/4}{4}{0/1,1/2,2/3,0/0,3/3}{\frac{45}{512}}{-}{6}{Y}
  \TableRow{72}{6}{0/1,0/2,0/3,0/4,1/2,1/5,3/5,4/5}{0/5,1/3,1/4,2/3,2/4,2/5,3/4}{6}{0/1,1/2,2/0,3/4,4/5,5/3,0/5,1/4,2/3}{\frac{5}{54}}{-}{7}{Y}
  \TableRow{73}{6}{0/1,0/2,0/4,1/2,1/3,3/5,4/5}{0/3,0/5,1/4,1/5,2/3,2/4,2/5,3/4}{0}{}{\le 0.0289172243}{-}{8}{-}
  \TableRow{74}{6}{0/1,0/3,0/4,1/2,1/3,3/5,4/5}{0/2,0/5,1/4,1/5,2/3,2/4,2/5,3/4}{0}{}{\le 0.0487738849}{-}{8}{-}
  \TableRow{75}{6}{0/2,0/3,0/4,1/2,1/3,3/5,4/5}{0/1,0/5,1/4,1/5,2/3,2/4,2/5,3/4}{2}{0/1}{\frac{4117715}{86093442}}{-}{6}{N/A}
  \TableRow{76}{6}{0/1,0/3,0/4,0/5,1/3,1/4,3/5,4/5}{0/2,1/2,1/5,2/3,2/4,2/5,3/4}{0}{}{\le 0.1490854571}{-}{8}{-}
  \TableRow{77}{6}{0/2,0/3,0/4,1/2,1/3,1/4,3/5,4/5}{0/1,0/5,1/5,2/3,2/4,2/5,3/4}{2}{0/1}{\frac{5242880}{43046721}}{-}{6}{N/A}
  \hline
  \caption{Inducibility constants of $6$-vertex graphs} \\
  }
\end{longtable}

\section{Preliminaries}
\label{sec:preliminaries}

Let $\reals$ denote the set of real numbers, and let $\naturals$ denote the
set of non-negative integers. For $n \in \naturals$, we define
$[n] := \set{0, \ldots, n - 1}$. We may abbreviate an unordered pair
$\{u,w\}$ as $uw$, including the case when $u$ and $w$ are
single-digit numbers. For an integer $\kappa\ge 0$, let $\falling{n}{\kappa}:=n(n-1)\ldots(n-\kappa+1)$ denote
the \emph{falling $\kappa$-th factorial} of $n$. 
For a set $X$ and an integer $\kappa \geq 0$, the set
of all $\kappa$-subsets of $X$ is denoted by $\binom{X}{\kappa}$.
For $a, b, \eps \in \reals$, we write $a=b\pm \eps$ if
$b-\eps \leq a\leq b+\eps$. We may omit ceiling/floor signs when they
are not essential.

A \emph{pattern} is a pair $B = (V(B), E(B))$ where $V(B)$ is a set (of \emph{vertices}) and  $E(B)$ is a symmetric subset of $V(B)^2$ (of \emph{edges}). Thus we allow loops on
vertices but no multiple edges. Its \emph{order} is $v(B) := |V(B)|$.
A loop on vertex $u$ is indicated by $\{u,u\} \in E(B)$ or
$uu \in E(B)$.
A \emph{pattern automorphism} is a bijection $f: V(B) \to V(B)$ such
that for all $u, w \in V(B)$, $uw \in E(B)$ if and only if
$f(u)f(w) \in E(B)$. (Thus it preserves loops and non-loops.)
The \emph{neighbourhood} of a vertex $u$ in $B$ is
\[
\Gamma_B(u) := \{w \in V(B) \mid \{u,w\} \in E(B)\}.
\] 
Note that
$u \in \Gamma_B(u)$ if and only if $u$ has a loop. The \emph{degree}
of $u$ is $\deg_B(u) := |\Gamma_B(u)|$.
These definitions also apply to graphs, which are patterns without loops.

For a graph $F$, its \emph{complement} is
$\O F := \left(V(F), \binom{V(F)}{2} \setminus E(F)\right)$. The
\emph{induced subgraph} on $X \subseteq V(F)$ is
\[
F[X] := (X, \{uw \in E(F) \mid u, w \in X\}),
\] 
and for disjoint
$X, Y \subseteq V(F)$,
$F[X,Y] := \{(u,w) \in X \times Y \mid uw \in E(F)\}$ denotes the set of edges
between $X$ and $Y$. When we say \emph{subgraph}, we mean \emph{induced subgraph}.

We will need the following special graphs: $K_n$ is the complete graph on $n$ vertices, $P_n$ is the
path with $n$ vertices, and $T_{n_0,\ldots,n_{m-1}}$ is the complete
$m$-partite graph with parts of sizes $n_0,\ldots,n_{m-1}$. The graph
$T_{2,1} \cong P_3$ is called the \emph{cherry}. The vertex-disjoint union of
graphs $F$ and $H$ is denoted $F \sqcup H$. Small graphs or patterns
can be written as $(m, E)$, meaning the vertex set is $[m]$; for
example, the $4$-vertex path is $(4, \{01,12,23\})$ and the pattern consisting of two isolated
loops is $(2, \{00,11\})$.

Let $B$ be a pattern with vertex set $[m]$. For pairwise disjoint sets
$V_0, \ldots, V_{m-1}$ (some possibly empty), the \emph{blowup}
$\blow{B}{V_0,\ldots,V_{m-1}}$ is the (loopless) graph on
$V = \bigcup_{i=0}^{m-1} V_i$ where distinct $x \in V_i$ and
$y \in V_j$ are adjacent if and only if $\{i,j\} \in E(B)$. In
particular, $V_i$ induces a clique if $i$ has a loop in $B$, and an
independent set otherwise. The family of all blowups of $B$ is denoted
by $\blow{B}{}$.
The standard \emph{$(m-1)$-dimensional simplex} is
\begin{equation}\label{eq:Simplex}
  \I S_m:=\left\{(x_0,\ldots,x_{m-1})\in \I R^m\mid
    x_0+\ldots+x_{m-1}=1\mbox{ and } \forall i\in [m]\ x_i\ge
    0\right\}.
\end{equation}
For $\V x=(x_0,\ldots,x_{m-1})\in \I S_m$,
let $p(F, \blow{B}{\V x})$ be the limit as $n\to\infty$ of
$p(F, \blow{B}{V_0,\ldots,V_{m-1}})$, where $|V_i|=(x_i+o(1))n$ for
$i\in [m]$. This is a continuous function on $\I S_m$ (in fact, a
polynomial). Let
\[
  p(F, \blow{B}{}) := \sup\{p(F, \blow{B}{\V x})\mid \V x\in\I S_m\}.
\]
By the compactness of $\I S_m$ the supremum is attained by at least
one $\V x\in\I S_m$; such vectors will be called
\emph{$(F,B)$-maximisers}. The pattern $B$ is called \emph{$F$-minimal}
if, for every pattern $B'$ (of order $m-1$) obtained from $B$ by
removing one vertex, it holds that
$p(F, \blow{B'}{}) < p(F, \blow{B}{})$. By compactness and continuity,
this holds if and only if the (closed) set of $(F,B)$-maximisers is
disjoint from the boundary of the simplex~$\I S_m$. The pattern $B$ is
called \emph{$F$-optimal} if $\lambda_{F} = p(F, \blow{B}{})$, that is,
we can attain the asymptotically optimal constant $\lambda_{F}$ by some
blowups of~$B$.  If $F$ and/or $B$ is understood, we may omit them
from the above notation.

A \emph{homomorphism} from a graph $F$ to a pattern $B$ is a function
$f: V(F) \to V(B)$ (not necessarily injective) such that for all
distinct $x, y \in V(F)$, $\{x,y\} \in E(F)$ if and only if
$\{f(x), f(y)\} \in E(B)$. This corresponds to assigning vertices of
$F$ to parts in a blowup of $B$ to form an induced copy of $F$. 
 Note that a homomorphism in our definition has to preserve both edges and
non-edges. An 
\emph{embedding} $f$ of a graph $F$ into a graph $G$, denoted $f: F \hookrightarrow G$,  is an injective homomorphism from $F$ to $G$ that preserves edges and non-edges. (Thus $f$ is an isomorphism of $F$ onto its
image.)  For graphs $F$ and
$G$ with $\kappa$ and $n$ vertices respectively, the \emph{embedding density}
$t(F,G)$ is the probability that a random injective map
$V(F) \to V(G)$ is an embedding. Note that
\begin{equation}\label{eq:t}
    p(F,G) = \frac{\kappa!}{|\mathrm{aut}(F)|}\, t(F,G),
\end{equation}
where $\mathrm{aut}(F)$ is the automorphism group of $F$.

  
The \emph{edit distance} $\delta_{\mathrm{edit}}(G,H)$ between graphs
$G$ and $H$ of the same order is the minimum of
$|E(G) \bigtriangleup f(E(H))|$ over all bijections
$f: V(H) \to V(G)$, that is, the smallest number of \emph{edits} (i.e.\ adjacency changes) needed to make
one graph isomorphic to the other. The distance from $G$ to a graph family
$\cG$ is
\[
\delta_{\mathrm{edit}}(G,\cG) := \min\{ \delta_{\mathrm{edit}}(G,H)
\mid H \in \cG \text{ and } v(H) = v(G) \}.
\]
We often consider
$\cG = \blow{B}{}$, so $\delta_{\mathrm{edit}}(G, \blow{B}{})$ is the
minimum number of edits required to transform $G$ into a blowup of~$B$.

\hide{
A sequence of bipartite graphs $(G_n)_{n\in\I N}$ with almost equal parts is
\emph{$c$-quasirandom} if for every bipartite graph $F$, the
\emph{bipartite non-induced homomorphism density}
$t_{\mathrm{bip}}(F, G_n)$, the probability that a random
part-preserving map $V(F) \to V(G_n)$ sends edges to edges, satisfies
$t_{\mathrm{bip}}(F, G_n) = c^{|E(F)|} + o(1)$ as $n \to \infty$. This
matches the density in a typical $c$-random subgraph of $T_{n,n}$. By
adapting the proof by Chung--Graham--Wilson~\cite{ChungGrahamWilson89} (see e.g.~\cite[Lemma 14]{CooleyKangPikhurko22}), it suffices to verify this for $F$ being an edge and a $4$-cycle.
}

We call a sequence of graphs $(G_n)_{n\in\I N}$
whose orders tend to infinity \emph{almost $F$-extremal} (resp.\
\emph{almost $c$-regular}) if for every $\eps>0$ there is $n_0$ such that, for
every $n\ge n_0$, it holds that $p(F, G_n) \ge \lambda_{F}-\eps$ (resp.\
at least $(1-\eps)v(G_n)$ vertices $u$ of $G_n$ satisfy
$\deg_{G_n}(u)=(c\pm\eps)v(G_n)$).

Asymptotic notation, such as $o(1)$, is taken with respect to
$n\to\infty$ (where $n$ is usually the order of the unknown graph
$G$).

We will also need the following routine facts.

 \begin{proposition}
\label{pr:ClaimF3} The maximum value of $p(x):=x^5 (1 - x) + x (1 - x)^5$ for $x\in [0,1]$ 
is attained if and only if $x\in\{y_0,1-y_0\}$, where
\begin{equation}\label{eq:y0}
 y_0:=\frac12- \sqrt{\frac{2\sqrt{10} - 5}{12}}=0.16776\ldots\ .
 \end{equation}     
 \end{proposition}
\begin{proof}
Define $Q(x) := p(x + 1/2)$ for $x\in\I R$. It can be verified that 
$Q(x) = -2x^6 - \frac{5}{2}x^4 + \frac{5}{8}x^2 + \frac{1}{32}$. Also, 
the derivative $Q'(x)$ is $x\paren{-12x^4 - 10x^2 + \frac{5}{4}}$ and its five (simple) roots are 0 and when $x^2$ is $(-5\pm 2\sqrt{10})/12$, including two complex roots. Thus the two square roots of $(2\sqrt{10}-5)/12>0$ are the only two local maxima of $Q$, both are in $(-1/2,1/2)$ and give the same value of the symmetric polynomial~$Q$. This proves the proposition.
\end{proof}

Let us remark that the minimal monic polynomial of the real $y_0$ from~\eqref{eq:y0} is 
\[
y^4-2 y^3+\frac{7 }{3}\,y^2 -\frac{4}{3}\,y+\frac{1}{6}=\frac{p'(y)}{6(1-2y)},
\]  and this polynomial will appear in a number of our theorems.

\begin{proposition}\label{pr:F2} For any fixed reals $a,b> 0$, the maximum of $r(x):=x^a(1-x)^b$ for $x\in [0,1]$ is attained if and only if $x=a/(a+b)$.\end{proposition}

\begin{proof} The derivative of $r$ is $r'(x)=x^{a-1}(1-x)^{b-1}(a-(a+b)x)$ so the proposition follows by checking the values of $r$ on $0$, $1$ and $a/(a+b)$.\end{proof}

\subsection{Flag algebras}
\label{sec:flag-algebras}

As the flag algebra framework introduced by Razborov~\cite{Razborov07}
is now standard (with
e.g.~\cite{Razborov10,BaberTalbot11,deCarliSilvadeOliveiraFilhoSato16,GilboaGlebovHefetzLinial22,JeongParkYang26} giving  detailed treatments),
we only provide the essential definitions required to describe a
flag algebra certificate.

Let $\tau$ be a \emph{type}, that is,  a graph with vertex set $[q]$ for some
$q \in \I N$, where all $q$ vertices are regarded as labelled. A \emph{$\tau$-flag}
is a pair $(F,f)$ consisting of a graph $F$ and an embedding
$f:[q] \to V(F)$ of $\tau$ into $F$.
This represents a partially labelled graph in
which the labelled vertices, called \emph{roots}, induce a copy of
$\tau$.
For $\tau$-flags $(F,f)$ and $(H,h)$, define $\P((F,f),(H,h))$ as the
number of $(v(F)-q)$-subsets $X$ of $V(H) \setminus h([q])$ such that
the $\tau$-flag $(H[X \cup h([q])],h)$ is isomorphic to $(F,f)$, where
isomorphisms must preserve each labelled root (in addition to edges and
non-edges). If $v(F) \le v(H)$, the corresponding \emph{flag density}
is
\[
\p((F,f),(H,h)) := \frac{\P((F,f),(H,h))}{\binom{v(H)-q}{v(F)-q}}.
\]
For $s \ge q$, let $\C F_s^\tau$ denote the set of all $\tau$-flags
with $s$ vertices, considered up to isomorphism. 
In particular, $\cF_s^0$ is the set of isomorphism class
representatives of  graphs of order~$s$, since we let $0$ denote the empty type that has no vertices. We fix an ordering
on $\C F_s^\tau$ for indexing vectors and matrices. For a $\tau$-flag
$(H,h)$ with $v(H) \ge s$, define the column vectors
\begin{equation}\label{eq:TauVector}
  \V V_{(H,h)}^{\tau,s} := \left( \P((F,f),(H,h)) \right)_{(F,f) \in \C F_s^\tau} \quad \text{and} \quad \V v_{(H,h)}^{\tau,s} := \left( \p((F,f),(H,h)) \right)_{(F,f) \in \C F_s^\tau},
\end{equation}  
which list the counts and densities, respectively, of all $s$-vertex
$\tau$-flags in $(H,h)$, following the fixed ordering of
$\C F_s^\tau$.

We now describe the contents of a flag algebra certificate that
establishes an upper bound $\lambda_{F} \le u$ for a given graph $F$ on
$\kappa$ vertices and $u \in \I R$. The certificate specifies an
integer $\N \ge \kappa$ and, for each type $\tau$ on $[q]$ with $1 \le
q \le \N-2$ and $q + \N$ even, a positive semi-definite matrix
$X^\tau$ indexed by $\C F_s^\tau$, where $s := (\N +
q)/2$. More precisely, from each equivalence class $\C C$ of types under
isomorphism as unlabelled graphs, only one representative $\tau \in \C
C$ and its matrix $X^\tau$ are listed, with all-zero matrices assigned
to other types in~$\C C$. Additionally, for every $H \in \C F_\N^0$, a
non-negative real coefficient $c_H$ (referred to as the \emph{slack} at $H$)
is provided, so that for any graph $G$ of order $n \to \infty$, the
following identity holds:
\begin{equation}
  \label{eq:FAMain}
  \begin{split}
    \sum_{H \in \C F_\N^0} (u-p(F, H)) \P(H, G) = & \sum_{\substack{1 \le q \le \N-2 \\ q \equiv N \bmod{2}}} \sum_{f: [q] \hookrightarrow V(G)} \left( \V V_{(G,f)}^{\tau,s} \right)^T X^\tau \V V_{(G,f)}^{\tau,s} \\
        & + \sum_{H \in \C F_\N^0} c_H \P(H, G) + O(n^{\N-1}),
  \end{split}
\end{equation}
where, in the inner sum, $s := (\N + q)/2$ and
$\tau := ([q], f^{-1}(E(G)))$ is the graph on $[q]$ such that the
injection $f: [q] \hookrightarrow V(G)$ embeds $\tau$ into $G$. Note
that the right-hand side of \eqref{eq:FAMain}, excluding the error
term, is non-negative because all matrices $X^\tau$ are positive
semi-definite and all slacks $c_H$ are non-negative. Meanwhile, the
left-hand side equals $(u-p(F, G)) \binom{n}{\N}$, confirming that
$\lambda_{F} \le u$.

Note that for any choice of square matrices $X^\tau$ of
appropriate dimensions and a given $u \in \I R$, there exists a unique
set of slack coefficients $c_H$ (some of which may be negative) that
satisfies~\eqref{eq:FAMain} for every $G$. In essence, for a fixed
type $\tau$ with $q \le \N-2$ vertices and two $\tau$-flags
$H_1, H_2 \in \C F_s^\tau$ where $s := (\N + q)/2$, the sum
$\sum_{f} \P(H_1, (G, f)) \P(H_2, (G, f))$ over all embeddings $f$ of
$\tau$ into an $n$-vertex graph $G$ can be expressed, up to an
$O(n^{\N-1})$ error, as a linear combination of counts of $\N$-vertex
subgraphs in $G$, with coefficients independent of $n$. This arises
because the sum counts pairs of copies of $H_1$ and $H_2$  in $G$ with
shared roots, and each such pair involves at most $\N$ vertices, with
the contribution from each $\N$-set depending solely on the induced
subgraph isomorphism class.
Consequently, for a fixed $\N$, the optimal upper bound $u$ provable
via \eqref{eq:FAMain} corresponds to the value of an explicit, though
typically large, semi-definite program.

\hide{
The following lemma asserts,
informally, that in every nearly extremal graph, the typical vectors
of $\tau$-rooted densities must lie close to the kernel of
$X^\tau$. We use the statement from~\cite[Lemma~2.1]{BodnarPikhurko26}, 
although versions of it appeared in some earlier papers.

\begin{lemma}
  \label{lm:Eigenspaces}
  Suppose that, for some $\N$, we have a flag algebra certificate
  proving that $\lambda_{F}\le u$ as in~\eqref{eq:FAMain}. Let $\tau$ be
  any type present in the certificate, say with $V(\tau)=[q]$, and let
  $s:=(\N+q)/2$. (Thus $s$ is an integer and $s\ge 1$).  Then for
  every $\eps>0$ there are $\delta>0$ and $n_0$ such that, for every
  graph $G$ with $n\ge n_0$ vertices and $p(F, G) \ge u- \delta$,
  there are at most $\eps n^q$ embeddings $f:\tau\hookrightarrow G$ such
  that
  \begin{equation}\label{eq:Eigenspaces}
    \left\| X^\tau \B v_{(G,f)}^{\tau,s}\right\|_\infty\ge \eps.
  \end{equation}
\end{lemma}
}

\subsection{Perfect stability}
\label{sec:perfect-stability}
Let $F$ be a given graph on $\kappa \geq 2$ vertices. 
Following the (more general) definition from
\cite{PikhurkoSliacanTyros19}, we call the $F$-problem
\emph{perfectly $B$-stable} for a pattern $B$ if there is $C>0$ such
that for every graph $G$ of order $n\ge C$ we have
\begin{equation}\label{eq:PerfectStabDef}
  \dedit(G,\blow{B}{})\le C \left(p(F,n)- p(F, G)\right)n^2.
\end{equation}
In particular, it follows that, for all $n\ge C$, every order-$n$ graph
$G$ with $p(F, G) = p(F,n)$ is a blowup of $B$ and then the determination of $p(F,n)$ reduces to  maximising an explicit polynomial of degree at
most $v(F)$ over (integer) part sizes summing up to~$n$.
This definition was motivated by the earlier results in~\cite{Furedi15,NorinYepremyan17,NorinYepremyan18} proving  versions of this property. 

Pikhurko, Sliacan
and Tyros~\cite[Theorem 7.1]{PikhurkoSliacanTyros19} presented a sufficient condition for a flag algebra proof to
give perfect stability and successfully applied it to a number of
problems (including some instances of the graph inducibility problem). Here we use the version of this condition that appears in \cite[Theorem~2.2]{BodnarPikhurko26} (with the latter allowing a pattern $B$ to contain loops).
\hide{Recall that a homomorphism of a graph $F$ to a pattern $B$ is a (not
necessarily injective) map that preserves both edges and non-edges; in
other words it is an assignment of the vertices of $F$ to the parts of
a blowup of $B$ that gives an induced copy of~$F$.
}

\begin{theorem}[{\cite[Theorem~7.1]{PikhurkoSliacanTyros19}} and  {\cite[Theorem~2.2]{BodnarPikhurko26}}]\label{th:PST7.1}
  Let $F$ be a graph with $\kappa \geq 2$ vertices. Let $B$ be a pattern
  (possibly with loops) on $[m]$ and let $\V a\in\I S_m$ be a vector
  with all entries positive. Suppose that all of the following
  statements hold.
  \begin{enumerate}[1)]
  \item \label{it:PST7.11} We have a flag algebra certificate $\C C$
    on $\N$ vertices proving that
    $\lambda_{F} \leq p(F, \blow{B}{\V a})$ as in~\eqref{eq:FAMain}.
  \item \label{it:PST7.12} We have a graph $\tau$ (without loops) with
    at most $\N-2$ vertices such that
    \begin{enumerate}[(a)]
    \item\label{it:PST7.12a} 
    if we restrict the maximisation of
          $p(F,n)$ to graphs without induced copy of $\tau$ then the
          limit strictly decreases, that is,
    \begin{equation}\label{eq:PST7.1a} \lim_{n\to\infty}
            \max\{p(F, G)\mid v(G)=n,\ P(\tau,G)=0\}<
            \lambda_{F};
          \end{equation}
      \hide{, that is
      \begin{equation}\label{eq:PST7.1a} \lim_{n\to\infty}
        \max\{p(F, G)\mid v(G)=n,\ \mbox{$G$ is $\tau$-free}\}<
        \lambda_{F};
      \end{equation}}%
    \item\label{it:PST7.12b} up to an automorphism of $B$ (which by
      definition has to preserve loops and non-loops), there is a unique
      homomorphism from $\tau$ to $B$;
    \item\label{it:PST7.12c} every two distinct vertices $x,y\in V(B)$
      have distinct neighbourhoods in $f(V(\tau))$, that is,
      $\Gamma_B(x)\cap f(V(\tau)) \not=\Gamma_B(y)\cap f(V(\tau))$, for some
      (or, equivalently by Item~\ref{it:PST7.12b}, for every) homomorphism
      $f$ of $\tau$ to $B$;
    \end{enumerate}
  \item\label{it:PST7.13} Every $H\in \C F_{\N}^0$ with $c_H=0$ admits a
    homomorphism to~$B$.
  \end{enumerate} Suppose further that at least one of the following
  statements holds:
  \begin{enumerate}[(i)]
  \item\label{it:PST7.1i} the certificate $\C C$ contains the graph
    $\tau$ as a type and the corresponding matrix $X^\tau$ in $\C C$ is of
    co-rank 1 (that is, its kernel has dimension 1);
  \item\label{it:PST7.1ii} suppose that the vertices of $B$ are either all loops or all non-loops, and  the limiting maximum density of $F$ in growing graphs without any induced copy of the graph $B^\circ$ (which is obtained from $B$ by removing 
  loops) is strictly less than $\lambda_{F}$;
    \item\label{it:PST7.1iii} the pattern $B$ is $F$-minimal.
  \end{enumerate}
  Then the problem of maximising $p(F, G)$ over $n$-vertex graphs $G$
  is perfectly $B$-stable. Moreover, if \Cref{it:PST7.1i} holds (in
  addition to Items~\ref{it:PST7.11}--\ref{it:PST7.13}) then $\V a$ is
  the unique vector $\V x\in \I S_m$ that maximises
  $p(F, \blow{B}{\V x})$ up to the automorphisms of~$B$.
\end{theorem}

\hide{Let us remark that the limit in the left-hand side
of~\eqref{eq:PST7.1a} exists by an easy double-counting argument, see
e.g.\ \cite[Lemma~2.2]{PikhurkoSliacanTyros19}.
Also, note that we
allow non-injective homomorphisms in \Cref{it:PST7.12b} and such maps
may be potentially required (e.g.\ we may need to force that a
specific vertex $x$ of $\tau$ is mapped to a loop, which can be done
by adding a clone $x'$ of $x$ and making them adjacent).}%

\section{Exact results}
\label{sec:exact-results}

This section contains the formal statements of the exact results from Table~\ref{ta:AllResults}. They all are proved using the package \texttt{FlagAlgebraToolbox} (commit 22b8765) which is being developed by Bodn\'ar~\cite{Bodnar26}. The script and certificates can be found in \href{https://github.com/bodnalev/supplementary_files/tree/main/graph_inducibility_6vtx}{\url{https://github.com/bodnalev/supplementary_files/tree/main/graph_inducibility_6vtx}}. The provided notebooks contain both the code that was used by us to generate all certificates (in case the reader would like to experiment with it by modifying it and running on some similar problems) as well as the code that verifies all stated results. General instructions on how to use the package can be found in~\cite{Bodnar26}.

\subsection{Cases where the `plain' application of the package sufficed}

Here we present the cases when the standard application of the flag algebra method gave the sharp upper bound on $\lambda_{F}$ and produced  a certificate which satisfied the perfect stability condition from \Cref{th:PST7.1}. 
For each such case, we list the $6$-vertex graph $F$, the value of $\lambda_{F}$, the pattern $B$, the vector $\V x$ of limiting ratios that maximise $F$-density in blowups of $B$ (which in all these cases happens to be unique). Also, we list a choice for $\tau$ that satisfies \Cref{th:PST7.1} and indicate which one of Items~\ref{it:PST7.1i}--\ref{it:PST7.1iii} is satisfied. If it is Item~\ref{it:PST7.1i} then the full theorem is proved. Otherwise, we provide a proof that the stated vector $\V x$ is the unique maximiser and, in case of Item~\ref{it:PST7.1iii}, also that the pattern $B$ is $F$-minimal. It is quite often the case that our graph $\tau$ is a subgraph of $F$ and then Condition~
\ref{it:PST7.12a} of \Cref{th:PST7.1} is trivially true.

In all these cases, except for \Cref{thm:graph_7} and~\Cref{thm:graph_27} (the graphs $F_7$ and $F_{27}$), it happens that the smallest $\N$ that gave the proof of the tight upper bound on $\lambda_{F_i}$ was also sufficient for verifying perfect stability via \Cref{th:PST7.1}, so we list this value of $\N$. 
(In each of the two exceptional cases $i\in\{7,27\}$, we could determine $\lambda_{F_i}$ using $\N=7$ but needed $\N=8$ for proving perfect stability.)

\label{sec:perf-stab-from-old-theorem}

We omit the trivial case $i=0$, where $F_0=(6,\{\})$ is the edgeless graph on $6$ vertices for which, trivially, $\lambda_{F_0}=1$ with the problem being perfectly $B$-stable for $B=(1,\{\})$. Here, the upper bound can also be formally proved via a flag algebra identity~\eqref{eq:FAMain} by taking each $X^\tau$ to be the all-zero matrix. 

In the following results, uniqueness of the part ratios is understood to be up to the automorphisms of the underlying pattern~$B$. Recall that $(m,E)$ is an abbreviation for the graph with vertex set $[m]$ and edge set $E$. 

\newcommand{\irationalConditionIPartI}[7]{Let
  $F_{#1} := \paren{#2, \set{#3}}$. Then $\lambda_{F_{#1}} =$ #4, where $\alpha_0 = #5\ldots$
  is a root of the polynomial #6. Furthermore, the $F_{#1}$-problem
  is perfectly $B$-stable,}
\newcommand{\irationalConditionIPartII}[7]{with
  $B := (#1, \set{#2})$ by satisfying  Theorem~\ref{th:PST7.1}\,\ref{it:PST7.1i} with $N = #4$ and
  $\tau = \left([#5], \set{#6}\right)$. Also, the unique maximiser of
  $p(F_{#7}, \blow{B}{\V x})$ is $\V x := ($#3$)$.}

\newcommand{\rationalConditionIPartI}[6]{Let  $F_{#1} := \paren{#2, \set{#3}}$. Then $\lambda_{F_{#1}} = #4$#6. Furthermore, the
  $F_{#1}$-problem is perfectly $B$-stable,}
\newcommand{\rationalConditionIPartII}[8]{with
  \checktoolong{B = ([#1], \set{#2}),}{6cm} and the unique maximiser of
  $p(F_{#8}, \blow{B}{\V x})$ is $\V x = ($#3$)$, by satisfying
Theorem~\ref{th:PST7.1}\,\ref{it:PST7.1i} with $N = #4$ and
  $\tau = \left([#5], \set{#6}\right)$.}

\newcommand{\rationalConditionIIPartI}[5]{Let
  $F_{#1} := \paren{#2, \set{#3}}$. Then $\lambda_{F_{#1}} = #4$. Furthermore, the
  $F_{#1}$-problem is perfectly $B$-stable,}
\newcommand{\rationalConditionIIPartII}[7]{with
  $B := (#1, \set{#2})$ by satisfying
Theorem~\ref{th:PST7.1}\,\ref{it:PST7.1ii} with $N = #4$ and $\tau \coloneqq (#5, \set{#6})$. Also,  the unique maximiser of
  $p(F_{#7}, \blow{B}{\V x})$ is $\V x = ($#3$)$.}

\newcommand{\irationalConditionIIPartI}[7]{Let
  $F_{#1} := \paren{#2, \set{#3}}$. Then $\lambda_{F_{#1}} =$ #4, where
  $\alpha_0 = #5\ldots$ is a root of the polynomial #6. Furthermore,
  the $F_{#1}$-problem is perfectly $B$-stable,}
\newcommand{\irationalConditionIIPartII}[7]{with
  $B := (#1, \set{#2})$ by satisfying
  Theorem~\ref{th:PST7.1}\,\ref{it:PST7.1ii} with $N = #4$ and $\tau \coloneqq (#5, \set{#6})$. Also, the unique maximiser of
  $p(F_{#7}, \blow{B}{\V x})$ is $\V x = ($#3$)$.}


\begin{theorem}
    \label{thm:graph_2}
    \irationalConditionIPartI{2}{6}{01, 04}{$- \frac{255}{8}\alpha_0^{2} + \frac{591}{10}\alpha_0 - \frac{129}{5}$}{0.7210}{$\alpha^{3} + 2\alpha^{2} - \frac{32}{5}\alpha + \frac{16}{5}$}{0/1,0/4}{}
    \irationalConditionIPartII{5}{01, 23}{$\frac{1}{4}\alpha_0,\frac{1}{4}\alpha_0,\frac{1}{4}\alpha_0,\frac{1}{4}\alpha_0,- \alpha_0 + 1$}{7}{5}{02,14}{2}
\end{theorem}

\begin{theorem}
    \label{thm:graph_4}
    \rationalConditionIIPartI{4}{6}{01, 02, 04}{40/243}{0/1,0/2,0/4}
    \rationalConditionIIPartII{3}{12}{$1/3,1/3,1/3$}{6}{2}{01}{4}
\end{theorem}
\begin{proof}
 We need to prove only that the given vector $\V x$ is the unique maximiser. Notice that $F \coloneqq F_4$ is the star $T_{1, 3}$ with two isolated vertices. Also, every embedding $f$ of $F$ into $\blow{B}{V_0,V_1,V_2}$ satisfies $f(3),f(5)\in V_0$, $f(0)\in V_i$ and $f(1), f(2), f(4) \in V_{3-i}$ for some $i\in \{1,2\}$ (and, conversely, every such injective map gives an embedding). Thus, 
 we get by~\eqref{eq:t} that 
 \[
 p(F,\blow{B}{\V y}) = \frac{6!}{3!\,2!} y_0^2 (y_1 y_2^3 + y_1^3 y_2) = 60 y_0^2 \paren{y_1^3y_2 + y_1y_2^3},\quad \mbox{for }\V y\in\I S_3.
 \]
 The derivative of $y^3(1-y)+y(1-y)^3$ has a triple root at $1/2$ so $y=1/2$ is the unique argument maximising it on $[0,1]$ (or even on $\I R$). Thus, by the homogeneity of the polynomial $p(F,\B y)$, any optimal vector $\B y
 \in\I S_3$ satisfies $y_1=y_2=(1-y_0)/2$ and we have
 $p(F,\blow{B}{\B y})=\frac{15}2\, y_0^2(1-y_0)^4$. 
 By \Cref{pr:F2}, $y_0=1/3$ is the unique maximiser on $[0,1]$ of the last polynomial.
 \hide{The last polynomial is easily seem to be maximised on $[0,1]$ for and only for $y_0=1/3$ by the geometric-arithmetic means inequality (or by just computing that the roots of its derivative are $0$, $1/3$ and a triple root at 1). We conclude}%
 Thus $\paren{1/3, 1/3, 1/3}\in\I S_3$ is a the unique maximiser of $p(F,\blow{B}{\B y})$, as desired.
\hide{We will only prove that $\V x$ is the unique maximiser. Notice that $F \coloneqq F_4$ is the star $K_{1, 3}$ with two isolated vertices. Let $V_0, V_1$ and $V_2$ be the parts of $B$, where $V_1$ and $V_2$ induce a complete bipartite graph. Notice that every embedding $f$ of $F$ in $B$ satisfies $f(3), f(5) \in V_0$. Also, either $f(0) \in V_1$ and $f(1), f(2), f(4) \in V_2$, or $f(0) \in V_2$ and $f(1), f(2), f(4) \in V_1$. There are $6! / \paren{2! \cdot 3!} = 60$ ways to arrange $6$ vertices in each of these ways. Hence, for real numbers $x_0, x_1, x_2 \geq 0$ with $x_0 + x_1 + x_2 = 1$ we have $p\paren{F, B\paren{x_0, x_1, x_2}} = 60 x_0^2 x_1 x_2^3 + 60 x_0^2 x_1^3 x_2 = 60 x_0^2 x_1 x_2 \paren{x_1^2 + x_2^2} = 60 f_1(x_0, x_1, x_2)$, where $f_1$ is the function in \cref{item:f1-maximiser} of \cref{prop:optimisation-simplex}. We conclude that the unique maximiser of $p\paren{F, B\paren{x_0, x_1, x_2}}$ is $\paren{x_0, x_1, x_2} = \paren{1/3, 1/3, 1/3}$.
}
\end{proof}

\begin{theorem}
    \label{thm:graph_6}
    \rationalConditionIPartI{6}{6}{01, 05, 15}{5/16}{0/1,0/5,1/5}{}
    \rationalConditionIPartII{2}{00}{$1/2,1/2$}{6}{2}{01}{0/0}{6}
\end{theorem}

\begin{theorem}
    \label{thm:graph_7}
    \rationalConditionIPartI{7}{6}{02, 04, 15}{160/729}{0/2,0/4,1/5}{ by flag algebra computations using $7$ vertices}
    \rationalConditionIPartII{9}{02, 01, 12, 34, 35, 45, 67, 68, 78}{$1/9,1/9,1/9,1/9,1/9,1/9,1/9,1/9,1/9$}{8}{6}{02,13,45}{0/2,0/1,1/2,3/4,3/5,4/5,6/7,6/8,7/8}{7}
\end{theorem}

\begin{theorem}
    \label{thm:graph_8}
    \rationalConditionIPartI{8}{6}{02, 13, 45}{10/81}{0/2,1/3,4/5}{}
    \rationalConditionIPartII{3}{00, 11, 22}{$1/3,1/3,1/3$}{6}{2}{}{0/0,1/1,2/2}{8}
\end{theorem}

\begin{theorem}
    \label{thm:graph_12}
    \rationalConditionIPartI{12}{6}{02, 03, 04, 15}{40/243}{0/2,0/3,0/4,1/5}{}
    \rationalConditionIPartII{3}{00, 12}{$1/3,1/3,1/3$}{6}{4}{01,02,03,12,13,23}{0/0,1/2}{12}
\end{theorem}

\begin{theorem}
    \label{thm:graph_14}
    \rationalConditionIPartI{14}{6}{02, 05, 12, 15}{10/81}{0/2,0/5,1/2,1/5}{}
    \rationalConditionIPartII{3}{12}{$1/3,1/3,1/3$}{6}{2}{01}{1/2}{14}
\end{theorem}

\begin{theorem}
    \label{thm:graph_15}
    \rationalConditionIPartI{15}{6}{03, 04, 12, 15}{45/256}{0/3,0/4,1/2,1/5}{}
    \rationalConditionIPartII{4}{02, 13}{$1/4,1/4,1/4,1/4$}{6}{4}{02,13}{0/2,1/3}{15}
\end{theorem}

\begin{theorem}
    \label{thm:graph_17}
    \rationalConditionIPartI{17}{6}{01, 03, 13, 45}{40/81}{0/1,0/3,1/3,4/5}{}
    \rationalConditionIPartII{3}{00, 11, 22}{$1/3,1/3,1/3$}{6}{2}{}{0/0,1/1,2/2}{17}
\end{theorem}

\begin{theorem}
    \label{thm:graph_21}
    \rationalConditionIPartI{21}{6}{01, 02, 05, 12, 15}{128/675}{0/1,0/2,0/5,1/2,1/5}{}
    \rationalConditionIPartII{6}{01, 02, 03, 04, 12, 13, 14, 23, 24, 34}{$2/15,2/15,2/15,2/15,2/15,1/3$}{8}{4}{01,02,03,12,13,23}{0/1,0/2,0/3,0/4,1/2,1/3,1/4,2/3,2/4,3/4}{21}
\end{theorem}

\begin{theorem}
    \label{thm:graph_24}
    \rationalConditionIPartI{24}{6}{02, 03, 04, 12, 15}{5/54}{0/2,0/3,0/4,1/2,1/5}{}
    \rationalConditionIPartII{6}{01, 12, 23, 34, 45, 50}{$1/6,1/6,1/6,1/6,1/6,1/6$}{6}{4}{01,02,13}{0/1,1/2,2/3,3/4,4/5,5/0}{24}
\end{theorem}

\begin{theorem}
    \label{thm:graph_27}
    \rationalConditionIPartI{27}{6}{01, 02, 03, 13, 45}{135/1024}{0/1,0/2,0/3,1/3,4/5}{ by flag algebra computations using $7$ vertices}
    \rationalConditionIPartII{8}{01, 12, 23, 30, 45, 56, 67, 74, 00, 11, 22, 33, 44, 55, 66, 77}{$1/8,1/8,1/8,1/8,1/8,1/8,1/8,1/8$}{8}{6}{03,04,12,15}{0/1,1/2,2/3,3/0,4/5,5/6,6/7,7/4,0/0,1/1,2/2,3/3,4/4,5/5,6/6,7/7}{27}
\end{theorem}

\begin{theorem}
    \label{thm:graph_28}
    \rationalConditionIPartI{28}{6}{02, 03, 14, 15, 45}{15/64}{0/2,0/3,1/4,1/5,4/5}{}
    \rationalConditionIPartII{3}{00, 12}{$1/2,1/4,1/4$}{6}{4}{01}{0/0,1/2}{28}
\end{theorem}

\begin{theorem}
    \label{thm:graph_31}
    \rationalConditionIPartI{31}{6}{02, 03, 12, 13, 45}{10/81}{0/2,0/3,1/2,1/3,4/5}{}
    \rationalConditionIPartII{3}{00, 12}{$1/3,1/3,1/3$}{6}{4}{01}{0/0,1/2}{31}
\end{theorem}

\begin{theorem}
    \label{thm:graph_32}
    \irationalConditionIIPartI{32}{6}{01, 02, 03, 04, 12, 13, 14, 23, 24, 34}{$\frac{20}{3}\alpha_0^{2} - \frac{20}{3}\alpha_0 + \frac{4}{3}$}{0.1677}{$\alpha^{4} - 2\alpha^{3} + \frac{7}{3}\alpha^{2} - \frac{4}{3}\alpha + \frac{1}{6}$}{0/1,0/2,0/3,0/4,1/2,1/3,1/4,2/3,2/4,3/4}
    \irationalConditionIIPartII{2}{00, 11}{$\alpha_0,- \alpha_0 + 1$}{7}{2}{}{32}
\end{theorem}
\begin{proof}
    We have to show only that $\V x$ is the unique maximiser. Notice that $F \coloneqq F_{32}$ is the clique on five vertices with an isolated vertex. Also, every embedding of $F$ into a blowup of $B$ maps the five vertices in the clique of $F$ to the same part and the isolated vertex to the other part. Thus, by~\eqref{eq:t}, we have that $
    p(F,\blow{B}{x, 1 - x}) = \frac{6!}{5!}\, \left(x^5 (1 - x) + x (1 - x)^5\right)$ for $x\in [0,1]$.
 By~\cref{pr:ClaimF3}, it is maximised if and only if $x\in\{y_0,1-y_0\}$, where $y_0$ is defined by~\eqref{eq:y0}.
It is routine to see that $y_0=\alpha_0$, finishing the proof of the uniqueness part of the theorem.
\end{proof}

\begin{theorem}
    \label{thm:graph_36}
    \irationalConditionIIPartI{36}{6}{02, 03, 05, 12, 13, 15}{$\frac{25}{6}\alpha_0^{2} - \frac{25}{6}\alpha_0 + \frac{5}{6}$}{0.1677}{$\alpha^{4} - 2\alpha^{3} + \frac{7}{3}\alpha^{2} - \frac{4}{3}\alpha + \frac{1}{6}$}{0/2,0/3,0/5,1/2,1/3,1/5}
    \irationalConditionIIPartII{4}{01, 23}{$\frac{1}{2}\alpha_0,\frac{1}{2}\alpha_0,- \frac{1}{2}\alpha_0 + \frac{1}{2},- \frac{1}{2}\alpha_0 + \frac{1}{2}$}{7}{3}{01}{36}
\end{theorem}
\begin{proof}
    We need to prove only that $\V x$ is the unique maximiser. Notice that $F \coloneqq F_{36}$ is the complete bipartite graph $T_{2, 3}$ with an isolated vertex. Also, every embedding of $F$ into $\blow{B}{V_0,V_1,V_2,V_3}$ maps all five vertices of $K_{2, 3}$ to $U_i$ for some $i\in\{0,1\}$ and the isolated vertex to $U_{1-i}$, where we let $U_0 := V_0 \cup V_1$ and $U_1 := V_2 \cup V_3$. Furthermore, an injection embeds $T_{2, 3}$ into $U_i=V_j\cup V_h$ if and only if it preserves the bipartition. Hence, by~\eqref{eq:t}, we have \begin{eqnarray*}
    p(F,\blow{B}{\V y}) &=& \frac{6!}{3!\cdot 2!}\, \paren{\paren{y_0^2 y_1^3 + y_0^3 y_1^2} \paren{y_2 + y_3} +  \paren{y_2^2 y_3^3 + y_2^3 y_3^2} \paren{y_0 + y_1}}\\
    &=& 60 \paren{y_0 + y_1} \paren{y_2 + y_3} \paren{y_0^2 y_1^2 + y_2^2 y_3^2},\qquad \mbox{for $\V y\in \I S_4$}.
    \end{eqnarray*}
Take any optimal $\V y$. By \Cref{pr:F2} when we fix $y_0+y_1$ (and thus also $y_2+y_3$), we have that $y_0=y_1$ and $y_2=y_3$. Thus we have to maximise
\[
p(F,\blow{B}{y / 2, y / 2, (1 - y) / 2, (1 - y) / 2}) = \frac{15}{4} \paren{y^5(1-y) + y(1 - y)^5}.
\] 
 for $y\in [0,1]$.
    By~\cref{pr:ClaimF3}, the maximisers are exactly when $y\in\{y_0,1-y_0\}$, where $y_0$ is as in~\eqref{eq:y0}, giving the desired.
\end{proof}

\begin{theorem}
    \label{thm:graph_37}
    \rationalConditionIPartI{37}{6}{23, 24, 25, 34, 35, 45}{80/243}{2/3,2/4,2/5,3/4,3/5,4/5}{}
    \rationalConditionIPartII{2}{00}{$2/3,1/3$}{6}{2}{01}{0/0}{37}
\end{theorem}

\begin{theorem}
    \label{thm:graph_38}
    \rationalConditionIPartI{38}{6}{01, 02, 03, 04, 15, 45}{72/625}{0/1,0/2,0/3,0/4,1/5,4/5}{}
    \rationalConditionIPartII{5}{01, 12, 23, 34, 40}{$1/5,1/5,1/5,1/5,1/5$}{6}{4}{01,02,13}{0/1,1/2,2/3,3/4,4/0}{38}
\end{theorem}

\begin{theorem}
    \label{thm:graph_49}
    \rationalConditionIPartI{49}{6}{03, 04, 05, 12, 13, 45}{5/54}{0/3,0/4,0/5,1/2,1/3,4/5}{}
    \rationalConditionIPartII{6}{00, 11, 22, 33, 44, 55, 01, 12, 23, 34, 45, 50}{$1/6,1/6,1/6,1/6,1/6,1/6$}{6}{4}{01,02,13}{0/0,1/1,2/2,3/3,4/4,5/5,0/1,1/2,2/3,3/4,4/5,5/0}{49}
\end{theorem}

\begin{theorem}
    \label{thm:graph_50}
    \rationalConditionIPartI{50}{6}{02, 03, 05, 12, 13, 45}{5/54}{0/2,0/3,0/5,1/2,1/3,4/5}{}
    \rationalConditionIPartII{6}{01, 12, 23, 34, 45, 50}{$1/6,1/6,1/6,1/6,1/6,1/6$}{6}{4}{01,02,13}{0/1,1/2,2/3,3/4,4/5,5/0}{50}
\end{theorem}

\begin{theorem}
    \label{thm:graph_53}
    \rationalConditionIPartI{53}{6}{04, 05, 12, 13, 23, 45}{5/16}{0/4,0/5,1/2,1/3,2/3,4/5}{}
    \rationalConditionIPartII{2}{00, 11}{$1/2,1/2$}{6}{2}{}{0/0,1/1}{53}
\end{theorem}

\begin{theorem}
    \label{thm:graph_55}
    \irationalConditionIIPartI{55}{6}{01, 02, 03, 05, 12, 13, 15}{$\frac{288}{125}\alpha_0^{2} - \frac{288}{125}\alpha_0 + \frac{288}{625}$}{0.1677}{$\alpha^{4} - 2\alpha^{3} + \frac{7}{3}\alpha^{2} - \frac{4}{3}\alpha + \frac{1}{6}$}{0/1,0/2,0/3,0/5,1/2,1/3,1/5}
    \irationalConditionIIPartII{4}{01, 00, 23, 33}{$\frac{2}{5}\alpha_0,\frac{3}{5}\alpha_0,- \frac{3}{5}\alpha_0 + \frac{3}{5},- \frac{2}{5}\alpha_0 + \frac{2}{5}$}{7}{4}{01,23}{55}
\end{theorem}
\begin{proof}
    We need to prove only that $\V x$ is the unique maximiser. Note that $F \coloneqq F_{55}$ can be obtained from the complete bipartite graph $T_{2, 3}$, adding an edge into the smaller part and  then adding an isolated vertex. Note that every embedding of $F$ in $\blow{B}{V_0,V_1,V_2,V_3}$ maps all vertices of the 5-vertex connected component of $F$ to $U_i$ for some $i\in\{0,1\}$ and the isolated vertex to $U_{1-i}$, where we define $U_0 := V_0 \cup V_1$ and $U_1 := V_2 \cup V_3$. Furthermore, if the connected component of $F$ with five vertices is mapped to $U_i=V_j \cup V_{h}$ (where $V_{h}$ is the clique part), then the three pairwise non-adjacent vertices of the component are mapped to $V_j$, and the remaining two vertices are mapped to $V_{h}$. Thus, by~\eqref{eq:t}, we have for $\V y\in\I S_4$ that
    \[
    p(F,\blow{B}{\V y}) = \frac{6!}{3!\,2!}\left( y_0^2 y_1^3 \paren{y_2 + y_3} + y_2^3 y_3^2 \paren{y_0 + y_1}\right) = 60 \paren{y_0^2 y_1^3 \paren{y_2 + y_3} + y_2^3 y_3^2 \paren{y_0 + y_1}}.
    \]
    Take any maximiser $\V y\in\I S_4$. By \cref{pr:F2} when we fix $y:=y_0 + y_1$, we have $\paren{y_0, y_1} = \paren{2y / 5, 3y / 5}$ and $\paren{y_2, y_3} = \paren{2(1 - y) / 5, 3(1 - y) / 5}$. 
    Thus, we have by routine calculations that
    \[
    p(F,\blow{B}{\V y}) = \frac{1296}{625} \paren{y^5(1-y) + y(1 - y)^5}.
    \]
 By~\cref{pr:ClaimF3},  the maximisers are exactly when $y\in \{y_0,1-y_0\}$, where $y_0$ is given in~\eqref{eq:y0}, as desired.
\end{proof}

\begin{theorem}
    \label{thm:graph_56}
    \rationalConditionIPartI{56}{6}{01, 23, 24, 25, 34, 35, 45}{15/32}{0/1,2/3,2/4,2/5,3/4,3/5,4/5}{}
    \rationalConditionIPartII{2}{00, 11}{$1/2,1/2$}{6}{2}{}{0/0,1/1}{56}
\end{theorem}

\begin{theorem}
    \label{thm:graph_58}
    \rationalConditionIPartI{58}{6}{01, 02, 03, 05, 14, 15, 45}{45/512}{0/1,0/2,0/3,0/5,1/4,1/5,4/5}{}
    \rationalConditionIPartII{4}{01, 12, 23, 11, 22}{$1/4,1/4,1/4,1/4$}{6}{4}{01,02,13}{0/1,1/2,2/3,1/1,2/2}{58}
\end{theorem}

\begin{theorem}
    \label{thm:graph_64}
    \rationalConditionIPartI{64}{6}{03, 04, 05, 12, 14, 15, 45}{5/108}{0/3,0/4,0/5,1/2,1/4,1/5,4/5}{}
    \rationalConditionIPartII{6}{00, 11, 22, 33, 44, 55, 01, 12, 23, 34, 45, 50}{$1/6,1/6,1/6,1/6,1/6,1/6$}{6}{4}{01,02,13}{0/0,1/1,2/2,3/3,4/4,5/5,0/1,1/2,2/3,3/4,4/5,5/0}{64}
\end{theorem}

\begin{theorem}
    \label{thm:graph_68}
    \rationalConditionIPartI{68}{6}{02, 03, 05, 12, 13, 15, 45}{72/625}{0/2,0/3,0/5,1/2,1/3,1/5,4/5}{}
    \rationalConditionIPartII{5}{01, 12, 23, 34, 40}{$1/5,1/5,1/5,1/5,1/5$}{6}{4}{01,02,13}{0/1,1/2,2/3,3/4,4/0}{68}
\end{theorem}

\begin{theorem}
    \label{thm:graph_69}
    \rationalConditionIPartI{69}{6}{02, 03, 04, 12, 13, 15, 45}{72/625}{0/2,0/3,0/4,1/2,1/3,1/5,4/5}{}
    \rationalConditionIPartII{5}{01, 12, 23, 34, 40}{$1/5,1/5,1/5,1/5,1/5$}{6}{4}{01,02,13}{0/1,1/2,2/3,3/4,4/0}{69}
\end{theorem}

\begin{theorem}
    \label{thm:graph_71}
    \rationalConditionIPartI{71}{6}{01, 02, 03, 04, 12, 35, 45}{45/512}{0/1,0/2,0/3,0/4,1/2,3/5,4/5}{}
    \rationalConditionIPartII{4}{01, 12, 23, 00, 33}{$1/4,1/4,1/4,1/4$}{6}{4}{01,02,13}{0/1,1/2,2/3,0/0,3/3}{71}
\end{theorem}

\begin{theorem}
    \label{thm:graph_72}
    \rationalConditionIPartI{72}{6}{01, 02, 03, 04, 12, 15, 35, 45}{5/54}{0/1,0/2,0/3,0/4,1/2,1/5,3/5,4/5}{}
    \rationalConditionIPartII{6}{01, 12, 20, 34, 45, 53, 05, 14, 23}{$\frac{1}{6},\frac{1}{6},\frac{1}{6},\frac{1}{6},\frac{1}{6},\frac{1}{6}$}{7}{3}{01,02,12}{???}{72}
\end{theorem}

\subsection{\texorpdfstring{The special complete partite graph $F_3$}{The special complete partite graph F3}}
\label{sec:perf-stab-complete-partite}

The previous section also contained the cases (namely,  $i\in\{6, 8, 17, 32, 37, 53, 56\}$) when one of the two complementary graphs is complete partite and the `plain' flag algebra approach works. 
Note that we skipped the trivial case $\O{F}_0=K_6$ (for which $\lambda_{F}\le 1$ can be proved as in~\eqref{eq:FAMain} by taking each $X^\tau$ to be zero). 

In the two remaining complete partite cases (namely, $i\in\{1,3\}$ that were resolved in~\cite{LiuMubayiReiher23,LiuMaZhu26,Yuster26}) taking $N=8$ was not sufficient for getting the sharp upper bound. So we switched to the theory of complete partite graphs (that is, forbidding an induced co-cherry), as this does not affect the inducibility constant by the result of Brown and Sidorenko~\cite{BrownSidorenko94} mentioned in \cref{sec:introduction}. 
There are much fewer flags and one can take $N$ larger than $8$, which makes the method stronger. We still could not prove the sharp upper bound on $\lambda_{F_1}$ via flag algebras but were able to resolve the problem for $\O{F}_3=T_{2,2,1,1}$ (which was also solved in~\cite{LiuMaZhu26,Yuster26}).



\begin{theorem}
Let $F_3:=(6,\{01,23\})$. Then $\lambda_{F_3}={25}/{72}$, the $F_3$-problem is perfectly $B$-stable with $B:=(6,\{00,11,22,33,44,55\})$ and the unique maximiser of $p(F_3,B(\V x))$ is 
$x=(\frac16,\frac16,\frac16,\frac16,\frac16,\frac16)$.
\end{theorem}
\begin{proof}
 We complement all graphs and work with $F:=\O{F}_3$, which is complete partite. The upper bound $\lambda_{F}\le {25}/{72}$ is proved via flag algebras with $N=10$ within the theory of complete partite graphs. Moreover, the obtained certificate also satisfies all assumptions of~\cite[Theorem 7.1]{PikhurkoSliacanTyros19} with $\tau=([6], \emptyset)$, including its Item~\ref{it:PST7.1i} (namely, that the co-rank of $X^\tau$ is 1). This theorem, as stated in \cite{PikhurkoSliacanTyros19}, allows also  to specify a forbidden family $\C F$ of induced subgraphs provided the property of being $\C F$-free is preserved under taking blowups (which is the case when we forbid the co-cherry $\O P_3$). We conclude  that perfect stability within the theory of all complete partite graphs holds for the $F$-problem. Specifically, there is $C>0$ such that every complete partite graph $G$ with $n\ge C$ vertices satisfies~\eqref{eq:PerfectStabDef}. Also,~\cite[Theorem 7.1]{PikhurkoSliacanTyros19} implies that $\V x$ is the unique optimal vector. These two conclusions give that, in the limit space of complete partite graphs as defined in~\cite{LiuPikhurkoSharifzadehStaden23}, the balanced complete 6-partite graph is the unique optimiser. Now, perfect stability (for general graphs) follows from~\cite[Theorem~1.1]{LiuPikhurkoSharifzadehStaden23}. In order to apply it, one has to check the following two properties for the uniform blowup $H=\blow{K_6}{V_0,\ldots,V_5}$ of large order $n$: flipping an edge in $H$ decreases the number of induced copies of $F$ by $\Omega(n^4)$ and if we attach a new vertex $v$ to $H$ such that $v$ is complete or empty to every part $V_i$ and creates at least $(25/72-o(1))\binom{n - 1}{5}$ induced copies of $F$ then $v$ is complete to exactly 5 parts of $H$. These properties can be easily checked directly (see e.g.~\cite{LiuMaZhu26}) and are also verified by our script.\end{proof}

\hide{There were two cases when the standard flag algebra approach did not yield the required result. The first such graph is $F_1$ (that consists of a single edge) for which the union of $13$ cliques of equal size gives the value of $\lambda_{F_{1}}$. However, the inducibility constant of  $F_1$  was determined by Liu, Mubayi and Reiher~\cite{LiuMubayiReiher23} with perfect stability established by Liu, Ma and Zhu~\cite{LiuMaZhu26}; so we just refer the reader to these papers. The second such graph is $F_3$, the complement of the $4$-partite Tur\'an graph. 
\sout{In order to prove Theorem~\ref{thm:graph_3}, we had to work with the theory of complete partite graphs (that is, to forbid an induced co-cherry), as this does not affect the inducibility constant by the result of Brown and Sidorenko~\cite{BrownSidorenko94}  mentioned in the \cref{sec:introduction}. Moreover, the certificate also implies perfect stability (within the theory of all graphs) for $F_3$ via the criterion in~{\cite[Theorem 7.1]{PikhurkoSliacanTyros19}} which, in fact, for any theory of graphs closed under subgraphs and blowups (in particular, to the theory of complete partite graphs). There are much fewer such graphs and we could take $\N$ as large as $10$ which was sufficient for the result.} 
The inducibility constant and perfect stability for $F_3$ also follow from the more general results by respectively Yuster~\cite{Yuster26} and Liu, Ma and Zhu~\cite{LiuMaZhu26}.}

\subsection{Cases with quasirandom extremal constructions}\label{sec:quasirandom}

There are $4$ cases (namely, $F_i$ for $i \in \{13, 70, 75, 77\}$) for which the lower bound comes from quasirandom bipartite graphs. 
Although we do not have a good definition of perfect stability for such constructions, we can describe the structure of all almost extremal graphs of order $n\to\infty$ within $o(n^2)$ edits, see Theorem~\ref{th:QR} below. In particular, we are able to prove the uniqueness of an $F_i$-extremal graphon. 
However, getting the exact value of $p(F_i,n)$ seems rather difficult so we do not pursue this direction here.

We were able to find a single argument that works in all these four cases so we treat them together. 
The following definitions will apply in this section:
\[
\begin{array}{lll}
 F_{13}:=(6,\{02,04, 12, 15\}), & p_{13}:=\frac{7-\sqrt{17}}6, & u_{13}:= \frac{11360-2720 \sqrt{17}}{2187},\\
 F_{70}:=(6,\{01, 02, 03, 04, 15, 35, 45\}),& p_{70}:=\frac78, & u_{70}:=\frac{12353145}{67108864},\\
 F_{75}:=(6,\{02, 03, 04, 12, 13, 35, 45\}),& p_{75}:=\frac79, & u_{75}:=\frac{4117715}{86093442},\\
 F_{77}:=(6,\{02, 03, 04, 12, 13, 14, 35, 45\}),& p_{77}:=\frac89,& u_{77}:=\frac{5242880}{43046721}.
\end{array}
\]

For $p\in [0,1]$, let $W_p$ be the graphon which assumes value $p$ on $(V_0\times V_1)\cup (V_1\times V_0)$
and $0$ on $V_0^2\cup V_1^2$, where $V_0:=[0,1/2)$ and $V_1:=[1/2,1]$.
The reader unfamiliar with the theory of graphons can instead consider $W_{n,p}$, a random bipartite graph with parts $V_0$ and $V_1$, each of size $n/2$, in which every pair across is an edge with probability $p$ and we take the limits of densities in typical outcomes as $n\to\infty$. Using the symmetry between $V_0$ and $V_1$, the embedding density $t(F,W_p)$ of a graph $F$ in $W_p$ can be computed as the sum  of 
$p^{e(F)}(1-p)^{|\O F[U_0,U_1]|}/2^{v(F)-1}$ over all (unordered) \emph{bipartitions} $\{U_0,U_1\}$ of $F$ (that is, the sets $U_0$ and $U_1$ are independent in $F$ and partition $V(F)$). By~\eqref{eq:t}, the density $p(F,W_p)$ can be obtained by multiplying this by $v(F)!/|\mathrm{aut}(F)|$.

It can be checked (and is verified by our script) that, for each $i \in \{13, 70, 75, 77\}$, it holds that $p(F_i,W_{p_i})=u_i$, this giving a lower bound on $\lambda_{F_i}$. 
The following theorem states that this lower bound is sharp and gives a description of almost extremal graphs. Recall that a sequence of growing bipartite graphs $(G_n)_{n\in\I N}$ with almost equal parts is
\emph{$p$-quasirandom} if, for every graph $F$, it holds that $\lim_{n\to\infty} p(F,G_n)= p(F,W_p)$. 
By
adapting the proof of Chung, Graham and Wilson~\cite{ChungGrahamWilson89} (see e.g.~\cite[Lemma 14]{CooleyKangPikhurko22}), it is enough to check this property only for $4$-vertex graphs $F$ (in fact, it is enough to compute only two parameters which are linear functions of the $4$-vertex densities, namely the non-induced homomorphism densities of the edge and the 4-cycle).

\begin{theorem}\label{th:QR} Let $i\in \{13,70,75,77\}$ and define $F:=F_i$ and $p:=p_i$. Then the following statements hold.
\begin{enumerate}[1)]
\item\label{it:QRConstant} We have that  $\lambda_{F}=u_{i}$ and there is a flag algebra proof of the upper bound as in~\eqref{eq:FAMain}  with $N:=6$ such that, for every bipartite type $\tau$ with $2$ or $4$ vertices and at least one edge, it holds that the co-rank of the matrix $X^\tau$ is exactly $2^{c(\tau)-1}$, where $c(\tau)$ denotes the number of (connectivity) components of the graph~$\tau$.
\item\label{it:QRStab} Every sequence of graphs $(G_n)_{n\in \mathbb{N}}$ of growing orders with $p(F,G_n)=\lambda_{F}+o(1)$ as $n\to \infty$ consists, up to $o((v(G_n))^2)$ edits in each $G_n$, of balanced bipartite $p$-quasirandom graphs.
\item\label{it:QRExact} For any $\eps>0$, there is $n_0$ such that any graph $G$ of order $n\ge n_0$ with $p(F,G)=p(F, n)$  admits a partition $V(G)=V_0\cup V_1$ such that each vertex of $G$ has $(p/2\pm \eps)n$ neighbours in the other part and at most $\eps n$ neighbours in its own part; moreover, if $i\in \{70,77\}$ then we can additionally assume that each $V_i$ is an independent set.
\end{enumerate} 
\end{theorem}

\begin{proof}
The first part of the theorem is verified by the provided flag algebra certificates and our script. 

In order to show Part~\ref{it:QRStab} it is enough to prove that for every $\eps>0$ there are $\eps_0>0$ and $n_0$ such that every graph $G$ on $n\ge n_0$ vertices with $p(F,G)\ge \lambda_{F}-\eps_0$ admits a partition $V(G)=V_0\cup V_1$ such that $|V_j|=(1/2\pm\eps)n$ and $e(G[V_j])\le \eps \binom{n}{2}$ for $j=0,1$ while $|p(H,G)-p(H,W_p)|\le \eps$ for every $4$-vertex graph $H$. 
\hide{Indeed, since the injective non-induced homomorphism densities of the edge $K_2$ and the $4$-cycle $C_4$ are linear functions (with explicit bounded coefficients) of the induced densities of $4$-vertex graphs,  these two values for $G$  are close to those for $W_p$. As we mentioned in \cref{sec:preliminaries}, this implies the desired quasirandomness property.}

So, let $\eps > 0$ be given. Choose positive constants, each being sufficiently small depending on the previous ones:
\begin{equation*}
\eps \gg \eps_5 \gg \eps_4 \gg \eps_3 \gg \eps_2 \gg \eps_1 \gg \eps_0 \gg 1 / n_0.
\end{equation*}
Let $n \geq n_0$, and let $G$ be an arbitrary $n$-vertex graph satisfying $p(F,G)\ge \lambda_{F}-\eps_0$. 

Recall that for a type $\tau$ on $[q]$, an integer $s>q$, and an embedding $f:\tau\hookrightarrow G$, the vector $\V v_{(G,f)}^{\tau,s}$ lists the densities of $s$-vertex $\tau$-flags in $(G,f)$ as defined in \eqref{eq:TauVector}.
We would like to define the analogues of these rooted densities for the graphon~$W_p$. By the symmetry between the parts of $W_p$ and the homogeneity of each part, it is enough to specify an unordered bipartition $\{A,V(\tau)\setminus A\}$ of $\tau$ (which will play the role of $f$), where $A$ is the set of roots in~$V_0$. For brevity, we specify only one part $A$ instead of the pair $\{A,V(\tau)\setminus A\}$.

Formally, for a type $\tau$ on $[q]$, a set $A\subseteq V(\tau)$ and an integer $s>q$, we define the vector $\V v_{(W_p,A)}^{\tau,s}$ indexed by $\C F_s^\tau$ as follows. If at least one of $A$ or $[q]\setminus A$ spans an edge in $\tau$, then $\V v_{(W_p,A)}^{\tau,s}$ is the zero vector. Otherwise, starting with the graph $\tau$ on $[q]$, add a set $S$ of $s-q$ new vertices. Take a uniform random partition $(S_0,S_1)$ of $S$, and make each pair in $(A\times S_1)\cup(([q]\setminus A)\times S_0)\cup(S_0\times S_1)$ an edge with probability $p$, with all choices being mutually independent. For a $\tau$-flag $H\in\C F_{s}^\tau$, the $H$-th entry of $\V v_{(W_p,A)}^{\tau,s}$ is the probability that the resulting random $\tau$-flag (with roots $0,\ldots,q-1$) is isomorphic to~$H$. 

We can derive the following result about rooted densities in our almost extremal graph~$G$ using the first part of the theorem.

\begin{claim}
\label{cor:convex-comb}
Let $q\in \{2,4\}$ and let $\tau$ be any bipartite type on $[q]$ with at least one edge. Let $s:=(q+N)/2=q/2+3$ and let
\[
\mathcal{L} := \left\{ \B{v}_{(W_p, A)}^{\tau, s} \mid \{A,  V(\tau)\setminus A\} \mbox{ is  a bipartition of the graph $\tau$} \right\}.
\]
Then, for all but at most $\eps_2 \falling{n}{q}$ embeddings $f: \tau \hookrightarrow G$,  there exist real coefficients $\alpha_{\B{v}}=\alpha_{\B{v}}(f)$, for $\B{v}\in \mathcal{L}$, summing to $1$ such that:
\begin{equation}
\label{eq:close_to_convex_combination}
    \left\| \B{v}_{(G,f)}^{\tau, s} - \sum_{\B{v} \in \mathcal{L}} \alpha_{\B{v}} \B{v} \right\|_1 \le \eps_2.
\end{equation}
\end{claim}
\begin{claimproof} 

Let us evaluate the flag algebra identity~\eqref{eq:FAMain} that proves the upper bound $\lambda_{F}\le u_i$ on the graph $G$. Its left-hand side is $(u_i-p(F,G))\binom{n}{N}\le \eps_0\binom{n}{ N}$. Each term on the right-hand side is non-negative, apart from the error term which we can assume to be at least $-\eps_0\binom{n}{ N}$ by choosing $n_0$ sufficiently large. It follows that, for all but at most $\eps_2 \falling{n}{q}$ embeddings $f: \tau \hookrightarrow G$, we have  that
 \begin{equation}\label{eq:LinComb}
\left(\B{v}_{(G,f)}^{\tau, s}\right)^T X^\tau \B{v}_{(G,f)}^{\tau, s}\le \frac{2\eps_0\binom{n
 }{ N}}{\eps_2 \falling{n}{q}\cdot  \binom{n-q}{ s-q}^2}\le \eps_1.
 \end{equation} 
 Thus it is enough to show that~\eqref{eq:LinComb} implies the existence of the desired reals $\alpha_{\V v}$. 
 Since the matrix $X^\tau$ is positive semi-definite and $\eps_1\ll \eps_2$, the inequality in~\eqref{eq:LinComb} implies
 that 
 the vector $\B{v}_{(G,f)}^{\tau, s}$ is within $\eps_2$ in the 1-norm from an element $\V y$ in the kernel of $X^\tau$. We can additionally assume that, like $\V x$, the vector $\V y$ has the sum of its entries equal to $1$.


The same argument as above when we evaluate the flag algebra identity on the graphon $W_p$ (or  large graphs approximating it) shows that, for each used type $\tau$ and each bipartition $\{A,V(\tau)\setminus A\}$ of $\tau$, the vector $\V v_{(W_p,A)}^{\tau,s}$ in fact belongs to the kernel of  $X^\tau$, that is, $\mathcal{L}\subseteq \ker X^\tau$.

By Part~\ref{it:QRConstant} (as verified by our script), the kernel of the matrix $\ker X^\tau$ from the certificate has dimension $2^{c(\tau)-1}$.

Let us show that the dimension of the linear subspace spanned by $\mathcal{L}$ is also $2^{c(\tau)-1}$. This is rather straightforward. We do only the least trivial case when $c(\tau)=3$ and thus, up to an isomorphism, $\tau=(4,\{03\})$. We have four different bipartitions with $A$ being $\{1,2,3\}$, $\{0,1,2\}$, $\{0,1\}$ and $\{0,2\}$. Take the four following $5$-vertex $\tau$-flags $H$ where the neighbourhood of the unlabelled vertex $v$ is $\{1,2,3\}$, $\{0,1,2\}$, $\{0,1\}$ and $\{0,2\}$ respectively. Consider the $4\times4$-matrix $M$ with rows (resp.\ columns) indexed by the above four partitions (resp.\ $\tau$-flags) in the stated ordering, where the $(A,H)$-entry of the matrix is the $H$-th entry of $\V v_{(W_p,A)}^{\tau,s}$. All entries of $M$ above the diagonal are zero because, in each of these cases, $H$ has a 2-edge path via $v$ between some two vertices of $\tau$ from two different parts of the bipartition $\{A,[4]\setminus A\}$. For example, in the first row (for $A=\{1,2,3\}$), each $\tau$-flag $H$ except the first one has the unlabelled vertex $v$ adjacent to $0$ as well as at least one other vertex of~$\tau$. On the other hand, we have $\Gamma_H(v)=A$ on the diagonal and the corresponding entry is $p^{\deg_H(v)}(1-p)^{4-\deg_H(v)}\not=0$. Thus the matrix $M$ is invertible. Furthermore, the dimension of the span of $\mathcal{L}$ cannot be larger than $|\mathcal{L}|\le 2^{c(\tau)-1}=4$, so we have equality here, as desired.

Thus the linear span of $\mathcal{L}$ is equal to $\ker X^\tau\supseteq \mathcal{L}$, since they have the same dimension. 
Returning to the vector $\V y\in \ker X^\tau$ that approximates the vector $\B{v}_{(G,f)}^{\tau, s}$ from~\eqref{eq:LinComb}, we thus can write $\V y$ as $\sum_{\V v\in\mathcal{L}}\alpha_{\V v}\, \V v$ for some reals $\alpha_{\V v}$. Since the entries of each $\V v\in\mathcal{L}$ also sum to $1$, we have that $\sum_{\V v\in\mathcal L}\alpha_{\V v}=1$, as desired.\end{claimproof}

Let $\tau$ be any type with $q\in \{2,4\}$ vertices as in \Cref{cor:convex-comb}. We call an embedding $f:\tau\hookrightarrow G$ \emph{bad} if~\eqref{eq:close_to_convex_combination} fails and \emph{good} otherwise. Thus \Cref{cor:convex-comb} states that there are at most $\eps_2 \falling{n}{q}$ bad embeddings. 
Note that if $\tau$ is connected as a graph then there is only one unordered bipartition 
$\{A,  V(\tau)\setminus A\}$ and thus $\mathcal{L}$ consists of a single vector which we then denote by $\V v_{W_p}^{\tau,s}$. In this case \Cref{cor:convex-comb} states that for all but at most $\eps_2\falling{n}{q}$ embeddings $f:\tau\hookrightarrow G$ we have
 \begin{equation}
 \label{eq:corank1}
\left\|\V v_{(G,f)}^{\tau,s}-\V v_{W_p}^{\tau,s}\right\|_1\le \eps_2.
\end{equation}

Let us first prove that the edge density of $G$ is close to the value in $W_p$.
\begin{claim}\label{cl:edge-density}
We have $|p(K_2,G) - p(K_2,W_p)| \leq 2\eps_2$.
\end{claim}
\begin{claimproof}
The edge density of $G$ is clearly at least $p(F,G)\cdot e(F)/\binom{6}{ 2}>\eps_2$. So we can fix an embedding $g: \gamma \hookrightarrow G$ satisfying~\eqref{eq:corank1}, where $\gamma:=(2,\{01\})$ is the single edge. Now let us select a random pair $uv$ of distinct vertices in $V(G)\setminus g([2])$ uniformly at random, and consider the probability that $uv \in E(G)$. We have
\begin{equation*}
    \pr{uv\in E(G)} = \frac{e(G) \pm \paren{2n - 3}}{\binom{n - 2}{2}} = p(K_2,G) \pm \eps_2.
\end{equation*}
Let $\V \beta$ be the vector indexed by $\C F_4^{\gamma}$ whose $H$-th entry is 1 if the two unlabelled vertices of $H$ are adjacent and 0 otherwise. We have $\pr{uv\in E(G)}$ is  the scalar product $\langle \V \beta, \V v_{(G,g)}^{\gamma,4}\rangle$ which in turn is $\langle \V \beta, \V v_{W_p}^{\gamma,4}\rangle \pm \eps_2$ by~\eqref{eq:corank1}.
\hide{\[
p(K_2,G) = \pr{uv\in E(G)} \pm \eps_2 =\langle \V \beta, \V v_{(G,g)}^{\gamma,4}\rangle \pm \eps_2,
\] 
which is $\langle \V \beta, \V v_{W_p}^{\gamma,4}\rangle \pm 2\eps_2$ by~\eqref{eq:corank1}.}%
Since $\langle \V \beta, \V v_{W_p}^{\gamma,4}\rangle=p(K_2,W_p)$ (by a direct calculation or by applying the above argument to large graphs approximating $W_p$), the claim follows.\end{claimproof}

We also need the following claim. In particular, it implies one of the statements of Part~\ref{it:QRStab}, namely that all 4-vertex densities in $G$ are within $\eps$ of their values in $W_p$.

\begin{claim}\label{cl:four-vertex}
    For every  graph $H$ with $3$ or $4$ vertices, we have $|p(H,G) - p(H,W_p)| \leq 50\eps_2$.
\end{claim}

\begin{claimproof}
First, consider the case when $v(H)=4$ and $H$ contains at least one edge. We will estimate the global density of $H$ by averaging the local densities of flags rooted at edges. Recall that $\gamma=(2,\{01\})$. Let $\V \beta$ be the vector indexed by $\C F^{\gamma}_4$ with its $J$-th coordinate being $\beta_J:=1$ if $J\in\C F_4^\gamma$ is isomorphic to $H$ after the root labels are removed and $\beta_J:=0$ otherwise. 
We now compute, in two different ways, the number of triples $(X, x, y)$, where $G[X]$ is isomorphic to $H$ and $x,y\in X$ are adjacent. We have
\begin{align*}
\binom{n}{4} p(H,G) \cdot 2e(H) &= \sum_{g: \gamma \hookrightarrow G} \sum_{J\in \C F_4^\gamma} \beta_J\cdot \binom{n-2}{2}p(J,(G,g))  \\
&=\binom{n-2}{2}\cdot\sum_{g: \gamma \hookrightarrow G} \left\langle \V \beta,\V v_{(G,g)}^{\gamma,4}\right\rangle\\
&= \binom{n-2}{2}\left( 2\binom{n}{ 2}(p(K_2,W_p)\pm 2\eps_2) \left(\left\langle \V \beta,\V v_{W_p}^{\gamma,4}\right\rangle \pm \eps_2\right)\pm \eps_2 \falling{n}{2}\right)\\
&=2e(H) \binom{n}{4} \paren{p(H,W_p) \pm 5\eps_2},
\end{align*}
 where the third equality uses~\eqref{eq:corank1} and \Cref{cl:edge-density}.
 Thus we proved that $p(H,G) = p(H,W_p) \pm 5\eps_2$ for every $4$-vertex graph $H$ with at least one edge.

Since the sum of densities of $4$-vertex graphs is $1$, we have that the densities of the edgeless 4-vertex graph $\O K_4$ in $G$ and $W_p$ differ by at most $(|\C F_4^0|-1)\cdot 5\eps_2=50 \eps_2$. Furthermore, the density of a $3$-vertex graph $H$ is an affine function of $4$-vertex graph densities excluding $\O K_4$ with each coefficient having absolute value at most 1, so its values on $G$ and $W_p$ also differ by at most~$50\eps_2$.
\end{claimproof}

If we have a cherry with edges $ab$ and $ac$ then we say that $bc$ is its \emph{base} or that the cherry \emph{is based} on $bc$.
Note that all pairs of non-adjacent vertices $u_0u_1$ in $W_p$ are of two types: if the vertices are in two different parts then there are no cherries  based on them, otherwise the density of cherries based on $u_0u_1$ is approximately $p^2/2$. Let us show that a similar classification of non-adjacent pairs is possible in~$G$. Define $B$ (resp.\ $C$) to consist of those $u_0u_1\in E(\O G)$ such that the number of cherries based on $u_0u_1$ is $p^2n/2\pm \eps_4n$ (resp.\ at most $\eps_4n$). 

\begin{claim}\label{cl:cherry-density}
    $\left|E(\O G)\setminus (B\cup C)\right|\le \eps_4\binom{n}{ 2}.$
\end{claim}
\begin{claimproof} Let $\sigma:=(4,\{02,21,13,30\})$ be the 4-cycle type, where it is convenient to have $0$ and $1$ non-adjacent. Call an embedding $f:\O\gamma\hookrightarrow G$ (or an ordered pair $(x_0,x_1)$ of non-adjacent vertices of $G$) \emph{$\sigma$-bad} if it extends to at least  $\eps_3 \falling{n}{2}$ bad embeddings $\sigma\hookrightarrow G$. Since we have in total at most $\eps_2\falling{n}{4}$ bad embeddings of $\sigma$ by \Cref{cor:convex-comb}, the number of $\sigma$-bad pairs $(x_0,x_1)$ is at most $\eps_2\falling{n}{4}/(\eps_3 \falling{n}{2})<(\eps_4/2)\falling{n}{2}$. 

Also, call a vertex of $G$ \emph{$K_3$-bad} if it is in at least $\eps_3 \binom{n}{ 2}$ triangles. By \Cref{cl:four-vertex}, the number of $K_3$-bad vertices is at most $3\, P(K_3,G)/ (\eps_3 \binom{n}{ 2})\le \eps_3n$. Consequently, the number embeddings $\O\gamma\hookrightarrow G$ that use at least one $K_3$-bad vertex is at most $2n$ times this, which is at most, say, $(\eps_4/2) \falling{n}{2}$. 

Thus the number of embeddings $\O\gamma\hookrightarrow G$ which are $\sigma$-bad or contain at least one $K_3$-bad vertex is at most $\eps_4\falling{n}{2}$. Hence it is enough to show that every remaining pair $x_0x_1$ of non-adjacent distinct vertices of $G$ is in $B\cup C$. Take any such $x_0x_1$. 
The number of cherries based on $x_0x_1$ is the size of the common neighbourhood $\Gamma(x_0x_1):=\Gamma_G(x_0)\cap\Gamma_G(x_1)$. Suppose that $|\Gamma(x_0x_1)|> \eps_4n$, as otherwise $x_0x_1\in C$ and we are done. Since e.g.~$x_0$ is not $K_3$-bad, the set $\Gamma(x_0x_1)\subseteq \Gamma(x_0)$ spans at most $\eps_3 \binom{n}{ 2}$ edges. Every choice of non-adjacent distinct $x_2,x_3\in \Gamma(x_0x_1)$ gives an embedding of the 4-cycle $\sigma$. At most $\eps_3 \falling{n}{2}$ of these embeddings can be bad (since $x_0x_1$ is not $\sigma$-bad). Since $\eps_4\gg \eps_3$, we can find $(x_2,x_3)$ giving a good embedding $f$ of $\sigma$. By definition, this means that~\eqref{eq:corank1} holds, that is,  
 \[
 \left\|\V v_{(G,f)}^{\sigma,5}-\V v_{W_p}^{\sigma,5}\right\|_1\le \eps_2.
 \]
 In particular, the densities of cherries on $01\in E(\O \sigma)$ are $\eps_2$-close in these two cases since we can compute them by taking the scalar product with the vector $\V \beta$ indexed by $\C F^{\sigma}_5$ whose $J$-th entry is $1$ if and only if  the (unique) unlabelled vertex of $J$ is adjacent to both $0$ and~$1$. Thus the density of cherries on $x_0x_1$ in $G$ is within $2\eps_2$ from $p^2/2$, the analogous value we compute from~$\V v_{W_p}^{\tau,5}$. We conclude that $x_0x_1\in B$, finishing the proof.
\end{claimproof}

By Claim~\ref{cl:four-vertex}, the density of the cherry $(3,\{01,02\})$ in $G$ is within $50\eps_2$ of $3p^2/4$, its density in $W_p$. On the other hand, the total number of cherries can be lower bounded by the sum of the number of cherries based on $u_0u_1$ for all $u_0u_1\in B$, with each summand being at least $(p^2/2-\eps_4)n$. Thus
\begin{equation}\label{eq:B} |B|\le \frac{(3p^2/4+50\eps_2)\binom{n}{3}}{(p^2/2-\eps_4)n}\le\left(\frac{1}{2}+\eps_5\right)\binom{n}{ 2}.
\end{equation}

In order to finish the proof of Part~\ref{it:QRStab} it is enough to show that, say, $p(K_3,H)\le 4\eps_5$, where
\[
H := (V(G), E(G) \cup C).
\]
Indeed, by \Cref{cl:cherry-density} and~\eqref{eq:B}, the number of edges in $H$ is at least $(\frac12-2\eps_5)\binom{n}{ 2}$. The Erd\H os--Simonovits Stability Theorem~\cite{Erdos67a,Simonovits68} for the Tur\'an problem for $K_3$ combined with the standard supersaturation results (\cite{ErdosSimonovits83}) or the Triangle Removal Lemma~(\cite{RuzsaSzemeredi78}) implies that there is a partition $V_0\cup V_1$ such that, for $j=0,1$, $|V_j|=(1/2\pm\eps)n$ and $H$ (and thus $G$) has at most $\eps \binom n2$ edges inside $V_j$. (The remaining portion of Part~\ref{it:QRStab}, namely that for every $4$-vertex graph its densities in $G$ and $W_p$ differ by at most  $\eps$, follows from \Cref{cl:four-vertex}.)

Thus it remains to bound the number of triangles in $H$. These are of four different kinds, depending on how many edges  of a triangle are in $G$. So it is enough to bound the number of triangles in $H$ of each kind by $\eps_5 \binom n3$. \Cref{cl:four-vertex} takes care of the triangles with all three edges from $G$.

The number of triangles in $H$ with exactly two edges from $G$ is at most $\binom n2\cdot \eps_4n\le \eps_5\binom n3$, since  every pair in $C\subseteq E(\O G)$ is in at most $\eps_4 n$ cherries by definition.

We will also need the following auxiliary claim in the two remaining cases.

\begin{claim}
\label{prop:find_base_of_many_cherries}
Let $\tau$ be a bipartite $4$-vertex type with at least one edge and let  $f$ be
a good embedding of $\tau$ into $G$. Let
$S \subseteq \binom{V(\tau)}{2}$ be a set of non-adjacent pairs in $\tau$ such that
for every bipartition of $\tau$, there is a pair $ab \in S$ with both
$a$ and $b$ in the same part. Then, there is a pair in $f(S)$ that is
not in $C$.
\end{claim}
\begin{claimproof}
Let $x_j:=f(j)$ for $j\in [4]$ and fix scalars $\alpha_{\V v}$ that sum up to 1 and satisfy~\eqref{eq:close_to_convex_combination}. Choose a vertex $x_4$ uniformly at random from $V(G) \setminus \{x_0, x_1, x_2, x_3\}$. Notice that for every pair $ab \in S$, the probability that the random vertex $x_4$ is adjacent to both $x_a$ and $x_b$ can be written as the scalar product $\langle \V \beta^{ab},
\B{v}^{\tau, 5}_{(G,f)} \rangle$, where the $J$-th coordinate of $\V \beta^{ab}$ is $1$ if $\{a,4\},\{b,4\}\in E(J)$ and 0 otherwise (and each $J\in\C F_{\tau}^5$ is assumed to have $V(J)=[5]$ with $4$ being the unique unlabelled vertex). Letting
$\V\beta := \sum_{ab \in S} \V\beta^{ab}$, we
have
\begin{align*}
\sum_{ab \in S} \pr{x_4x_a, x_4x_b \in E(G)} &= \left\langle \V{\beta},
\B{v}^{\tau, 5}_{(G,f)} \right\rangle \\
&=
\left\langle
\V{\beta},
\sum_{\B{v}
\in
                                                                \mathcal{L}}
                                                                \alpha_{\B{v}}
                                                                \B{v}
                                                                \right\rangle
                                                                + \left\langle \V{\beta},
                                                                \B{v}^{\tau,
                                                                5}_{(G,f)}
                                                                - \sum_{\B{v}
                                                                \in
                                                                \mathcal{L}}
                                                                \alpha_{\B{v}}
                                                                \B{v}
                                                                \right\rangle \\
                                                              &\geq \sum_{\B{v} \in \mathcal{L}} \alpha_{\B{v}} \langle
                                                                \V{\beta}, \B{v}
                                                                \rangle - \eps_2
                                                                \|\V{\beta}\|_\infty 
                                                              \ \geq\ 
                                                                \frac{p^2}{2}
                                                                -
                                                                \eps_2,
\end{align*}
 where the last inequality uses our assumption on~$S$.
This implies that there
are at least $6\eps_4 n$ vertices
$x_4 \in V(G) \setminus \set{x_0, x_1, x_2, x_3}$ for which there exists
a pair $ab \in S$ such that $x_4$ is adjacent to both $x_a$ and
$x_b$ in $G$. Hence, at least one pair in $f(S)$ is the base of more than $\eps_4 n$ cherries in $G$, as desired.
\end{claimproof}

Take any triangle $x_0x_1x_2$ in $H$ with exactly one edge from $G$,
say $x_0x_1\in E(G)$ and any vertex $x_3\in V(G)\setminus\{x_0,x_1,x_2\}$. If $G[\{x_0, x_1, x_2, x_3\}]$ is bipartite then let $\tau$ be the type on $[4]$ such that the map $f:[4]\to V(G)$ that sends $i$ to $x_i$ is an embedding and note that $\tau$ has at least one edge (namely $01\in E(\tau)$) and $f$ cannot be good. Indeed, otherwise, by  \Cref{prop:find_base_of_many_cherries} applied to the set $S=\{02, 12\}$, at least one of
$x_0x_2$ and $x_1x_2$ is not in $C$, a contradiction. In other words, for every choice of $x_3$, we have that the induced subgraph $G[\{x_0,x_1,x_2,x_3\}]$ is non-bipartite or that the map $f$ is a bad embedding. In the former case, the number of choices of $(x_0,x_2,x_3,x_4)$ is at most $|\C F_4^0|\cdot 50\eps_2 \falling{n}{4}$ by \Cref{cl:four-vertex}. In the latter case, the number of choices is at most $\eps_2\falling{n}{4}$ by~\Cref{cor:convex-comb}. Thus we have at most $\eps_3 \falling{n}{4}$ choices of $(x_0,x_2,x_3,x_4)$ in total. This, when divided by $2n$, gives an upper bound on the number of triangles in $H$ with one edge from $G$, as desired.

Finally, it remains to prove that the number of  triangles in $H$ with no edges from $G$ is at most $\eps_5\binom{n}{3}$. Let us first start with the following result about
 \[
 L:=\{v\in V(G): \deg_{G}(v)\le \eps_5n\}.
 \]

\begin{claim}
  \label{prop:few_vtx_of_low_deg}
  It holds that $|L|\le \eps_4n$.\end{claim}
\begin{claimproof}
   Let $J:=(3,\{01\})$ be the $3$-vertex graph with exactly one edge. One can easily check that $J$ is induced subgraph of $F$ so $G$ contains many copies of $J$. Suppose to the contrary that $|L|>\eps_4n$. Then a familiar counting argument shows that there is a bipartite $4$-vertex type $\tau\supseteq J$ and a good embedding $f:\tau\hookrightarrow G$ with $f(3)\in L$. Let us now choose the fifth vertex $w$ uniformly at random in $V(G)\setminus f([4])$. Then
   \begin{align*}
       \mathbb{P}[\{w,f(3)\}\in E(G)]&=\left\langle \V{\beta},\B{v}^{\tau, 5}_{(G,f)} \right\rangle\\
       &= \left\langle \V{\beta},\sum_{\B{v}\in\mathcal{L}}\alpha_{\B{v}}\B{v} \right\rangle + \left\langle \V{\beta}, \B{v}^{\tau, 5}_{(G,f)} - \sum_{\B{v}\in \mathcal{L}} \alpha_{\B{v}}\B{v}\right\rangle\\
       &\ge \sum_{\B{v}\in \mathcal{L}}\alpha_{\B{v}}\langle\V{\beta},\B{v}\rangle-\eps_2\|\V{\beta}\|_\infty\ \ge\ \frac{p}{2}-\eps_2,
   \end{align*}
where $J'$-th coordinate of $\V{\beta}$ for $J'\in\C F_{\tau}^5$ is 1 if the unlabelled vertex is adjacent to Root $3$ and 0 otherwise. Thus $f(3)$ has too large degree to be in $L$, a contradiction.
\end{claimproof}

Suppose for contradiction
that the number of  triangles in $C$ is at least $\eps_5
\binom{n}{3}$. By \Cref{prop:few_vtx_of_low_deg}, at most $|L|\cdot \binom{n}{2} < \frac12 \eps_5
\binom{n}{3}$ of them intersect the low-degree set~$L$. If we take one of the remaining triangles, say on $\{x_0,x_1,x_2\}$ and $x_3\in\Gamma_G(x_0)\setminus\{x_1,x_2\}$ then we have at least $\frac12 \eps_5
\falling{n}{3}\cdot (\eps_5n-2)-|\C F_4^0|\cdot 50\eps_2 \falling{n}{4}>\eps_4\falling{n}{4}$ choices of $(x_0,x_1,x_2,x_3)$, each giving an embedding $f:\tau\to G$ of some 4-vertex bipartite type $\tau$ (which has at least one edge by $x_0x_3\in E(G)$). Thus some such type $\tau$ appears for strictly more than $\eps_2 \falling{n}{4}$ embeddings and, by \Cref{cor:convex-comb}, at least one resulting embedding~$f$ is good. Fix one such $\tau$ and $f$. We apply
\Cref{prop:find_base_of_many_cherries} with the set
$S:=\set{01, 02, 12}$ to conclude that at least one pair in $\{f(0),f(1),f(2)\}$ is not in $C$, a contradiction. This concludes
the proof of this case and thus of Part~\ref{it:QRStab}.

Let us turn to Part~\ref{it:QRExact}.
Since its proof is an easy adaptation of the proof of~\cite[Theorem~4.2]{BodnarPikhurko25}, we will be rather brief and use the asymptotic notation $o(1)$ for negligible terms as $n\to\infty$.
Consider an arbitrary graph $G$ with $n\to\infty$ vertices and
$p(F,G)=p(F,n)$.

By Part~\ref{it:QRStab}, there is a partition $V(G)=U_0\cup U_1$ such that each part has $(1/2+o(1))n$ vertices and spans $o(n^2)$ edges and the bipartite graph $G[U_0,U_1]$ is $p$-quasirandom. Let $V_0\cup V_1$ be a partition of $V(G)$ that maximises the number of crossing edges. This number is at least $|E(G[U_0,U_1])|=p(n/2)^2+o(n^2)$ and at most $|E(G)|=p(n/2)^2+o(n^2)$. Thus almost all edges of $G$ lie between $V_0$ and $V_1$. By the quasirandomness, we have up to swapping the parts that $|V_i\triangle U_i|=o(n)$ for $i\in [2]$. This means that all conclusions of Part~\ref{it:QRStab}  also hold for the partition $V_0\cup V_1$ (with the error terms being worse but still $o(1)$).

For a vertex $v\in V(G)$, let $P_v(F,G)$ be the number of induced
copies of $F$ in $G$ that contain $v$. Since
\[
  \frac1n\sum_{v\in V(G)} P_v(F,G)
  =p(F,G)\binom{n-1}{5}
  =(\lambda_F+o(1))\binom{n-1}{5},
\]
the standard cloning argument gives
\[
  P_v(F,G)\ge (\lambda_F-o(1))\binom{n-1}{5}
  \qquad\text{for every }v\in V(G).
\]
Indeed, if some vertex had contribution smaller than this, then replacing it by a clone of a vertex of average
contribution would strictly increase the number of induced copies of $F$ (even if we upper bound the number of destroyed copies that use both vertices by $\binom{n-1}{4}$).

We next consider the contribution of a new
vertex $u$ when added to a $p$-quasirandom biparite graph with parts $U_0$ and $U_1$, each of size $n/2$. Suppose that  $u$ has $x_in$
neighbours in $U_i$ for $i=0,1$.
Then the number of induced copies of $F$ containing $u$ is
\[
  \bigl(q_F(x_0,x_1)+o(1)\bigr)\binom n5,
\]
 for some polynomial $q_F$ of degree at most~$5$. Namely, $q_F$ can be obtained by summing, over the choice of the
vertex of $F$ mapped to $u$ and over all ordered bipartitions of the
remaining five vertices, the corresponding monomials in
$x_0,x_1,1/2-x_0,1/2-x_1,p$ and $1-p$.

The accompanying script computes this polynomial and verifies over exact arithmetic, that for each $i\in\{13,70,75,77\}$,
\[
  q_F(x_0,x_1)\le \lambda_F
  \qquad\text{for all }(x_0,x_1)\in[0,1/2]^2,
\]
with equality only at
\[
  (x_0,x_1)=(0,p/2)\quad\text{and}\quad (x_0,x_1)=(p/2,0).
\] 

Now take any $v\in V(G)$ and put
\[
  x_i(v):=\frac{|\Gamma_G(v)\cap V_i|}{n},\,\quad\mbox{for $i=0,1$}.
\]
The Quasirandom Counting Lemma, applied to various Boolean combinations of
$V_0$, $V_1$, $\Gamma_G(v)$ and $\Gamma_{\O G}(v)$, gives
\[
  P_v(F,G) = \bigl(q_F(x_0(v),x_1(v))+o(1)\bigr)\binom n5.
\]
The above statements imply by the continuity of $q_F$  that
$(x_0(v),x_1(v))$ is $o(1)$-close to
$(0,p/2)$ or $(p/2,0)$.
Since $V_0\cup V_1$ was chosen to maximise the number of crossing
edges, moving any single vertex to the other side cannot increase the
cut. Thus every $v\in V_i$ has at least as many neighbours in
$V_{1-i}$ as in $V_i$ giving that
\begin{equation}\label{eq:QrDegrees}
  |\Gamma_G(v)\cap V_i|=o(n) \qquad \text{and}\qquad|\Gamma_G(v)\cap V_{1-i}|=(p/2+o(1))n
\end{equation}
for every $i \in [2]$ and $v\in V_i$, proving one of the conclusions of Part~\ref{it:QRExact}

It remains to prove, for $i\in\{70,77\}$, that the parts can be chosen
to be independent. Suppose to the contrary that
$uv\in E(G[V_0])$ (a similar argument works for $uv \in E(G[V_1])$). Define their relative common neighbourhood size to be
\[
  q := \frac{|\Gamma_G(u)\cap\Gamma_G(v)\cap V_1|}{n/2}.
\]
Both $u$ and $v$ have
$(p/2+o(1))n$ neighbours in $V_1$ by~\eqref{eq:QrDegrees}, and therefore
\[
  2p-1-o(1)\le q \le p+o(1).
\]
Partition $V_1$ into the four sets
\[
\begin{array}{ll}
  A:=\Gamma_G(u)\cap\Gamma_G(v)\cap V_1, &
  B:=(\Gamma_G(u)\setminus\Gamma_G(v))\cap V_1,\\ 
  C:=(\Gamma_G(v)\setminus\Gamma_G(u))\cap V_1, &
  D:=V_1\setminus(\Gamma_G(u)\cup\Gamma_G(v)),
\end{array}
\]
whose sizes are therefore
\[
  |A|=\frac{q n}{2}+o(n),\quad
  |B|=|C|=\frac{(p-q)n}{2}+o(n),\quad
  |D|=\frac{(1-2p+q)n}{2}+o(n).
\]
Let $G'$ be obtained from $G$ by deleting the edge $uv$. Only copies of
$F$ using both $u$ and $v$ can change. The number of such copies that
also use another edge inside $V_0$ or $V_1$ is $o(n^4)$. Hence the main term of
$P(F,G)-P(F,G')$ is obtained from the contribution of the parts
\[
  V_0\setminus\{u,v\},\quad A,\quad B,\quad C,\quad D,
\]
where between $V_0\setminus\{u,v\}$ and $A\cup B\cup C\cup D$
we have a quasirandom structure with edge probability $p$, the vertex $u$ is complete to $A\cup B$, the vertex
$v$ is complete to $A\cup C$, and the only difference between the two
constructions is whether the pair $uv$ is an edge.

The accompanying script computes asymptotically this difference, as a function of $q$. More precisely, it obtains a polynomial $d_F(q)$, such that
\[
  P(F,G)-P(F,G') =
  \bigl(d_F(q)+o(1)\bigr) \binom n4.
\]

For both $i \in \{70, 77\}$, we obtain that $d_{F_i}(q) < 0$ for every $q$ in the interval $[2p-1, p]$. Thus  it holds for large enough $n$ that
\[
  P(F_i,G)-P(F_i,G')<0
\]
and deleting the edge $uv$
strictly increases the number of induced copies of $F$. This contradicts
the optimality of $G$, giving that $G[V_0]$ is empty.
\end{proof}

Let us remark that Part~\ref{it:QRStab} of \Cref{th:QR} implies that $W_{p_i}$ is the unique (up to weak isomorphism) $F$-extremal graphon. 

The last argument uses the strict negativity of the
polynomial $d_F(q)$ on the whole allowed interval for $q$. For
$i\in\{13,75\}$ the analogous polynomial has a root at the endpoint
$q=p$ (this calculation is included in our script as well), corresponding to the case when the two vertices inside the same part have identical neighbourhoods. Note that the edge-deletion argument does not give an immediate contradiction in these cases, as both $F_{13}$ and $F_{75}$ are twin-free, hence the contribution is $0$ with and without the edge between the two vertices. While a contradiction may be possible by considering different transformations (e.g.\ increasing the symmetric difference between $\Gamma(u)$ and $\Gamma(v)$), this would probably require more careful analysis of lower order terms so we did not pursue this direction.

\subsection{\texorpdfstring{Special case of $F_{65}$}{Special case of F65}}

The case of $F_{65}$, where 
\begin{equation}\label{eq:65}
 F_{65}:=(6,\{02, 04, 05, 12, 14, 15, 45\}),
 \end{equation}
 is somewhat special in the sense that, although we can determine the inducibility constant and it comes from blowups of a (random-free) pattern $B$, we cannot establish perfect stability by applying Theorem~\ref{th:PST7.1} since no graph $\tau$ with at most $7$ vertices can satisfy Items~\ref{it:PST7.12b} and~\ref{it:PST7.12c} of Theorem~\ref{th:PST7.1}.
 Specifically, we have that
 \begin{equation}\label{eq:65B}
 B:=C_6^*\sqcup C_6^*,\quad \mbox{where $C_6^*:=(6,\{00,11,22,33,44,55,01,12,23,34,45,50,03,14,25\})$}.
 \end{equation}
Thus $B$ consists of two disjoint copies of $C_6^*$ which in turn is obtained from the $6$-cycle by adding loops at all six vertices and adding all three main diagonals as edges. (Equivalently, $C_6^*$ is the pattern complement of the graph $K_3\sqcup K_3$.) As it is routine to see, no three roots can identify all six parts of $C_6^*$-blowups and, consequently, no seven roots can identify all 12 parts of $B$-blowups. 

As we will show, the maximising vector is unique to up an automorphism of $B$ (in fact, up to swapping the two copies of $C_6^*$): namely, we take a uniform blowup of each copy of $C_6^*$ but these blowups occupy $\alpha n$ and $(1-\alpha)n$ vertices respectively, where 
\begin{equation}
\label{eq:alpha}
\mbox{$\alpha :=0.16776\ldots$ is a root of
$\alpha^4 - 2\alpha^3 + \frac73\, \alpha^2 - \frac43\, \alpha + \frac16$.}
\end{equation}
 Thus, assuming that the first copy of $C_6^*$ in $B$ has vertex set $[6]$, the maximiser is
 \begin{equation}\label{eq:x}
     \V \alpha:=\left(\frac{\alpha}{6},\frac{\alpha}{6},\frac{\alpha}{6},\frac{\alpha}{6},\frac{\alpha}{6},\frac{\alpha}{6},
     \frac{1-\alpha}{6},\frac{1-\alpha}{6},\frac{1-\alpha}{6},\frac{1-\alpha}{6},\frac{1-\alpha}{6},\frac{1-\alpha}{6}\right).     
 \end{equation}

Nonetheless, we can establish perfect $B$-stability by using some additional properties from the obtained flag algebra certificate. In this task, we first establish the following (weaker) property. Following~\cite{PikhurkoSliacanTyros19}, we say that an $F$-inducibility problem is \emph{robustly $B$-stable} if there is a constant $C>0$ such that every graph $G$ with $n\ge C$ vertices 
 \begin{equation}\label{eq:RobustStab}
 \dedit(G,\blow{B}{})\le C\max\set{n,(\lambda_{F}-p(F,G))n^2}.
 \end{equation}
  Note that we can replace $\lambda_{F}$  by $p(F,n)$ in this definition since $p(F,n)=\lambda_{F}+O(1/n)$ as $n\to\infty$ by
\cite[Lemma~2.2]{PikhurkoSliacanTyros19}.

\begin{theorem}\label{thm:graph_65}
Let $F_{65}$, $B$, $\alpha$ and $\V \alpha$ be defined by~\eqref{eq:65}--\eqref{eq:x}. Let $F:=F_{65}$. Then the following statements hold.
 \begin{enumerate}[1)]
\item\label{it:65a} We have that $\lambda_{F}=\frac{50}{27} \alpha^{2} - \frac{50}{27} \alpha + \frac{10}{27}$, where 
the upper bound can be proved via a flag algebra identity~\eqref{eq:FAMain} with $N=7$ so that
 \begin{enumerate}[(i)]
 \item\label{it:65Sigma} the co-rank of $X^\sigma$ is $2$, where
 $\sigma:=(5,\{01,12,23,30\})$ 
 is the $4$-cycle plus an isolated vertex, and
 \item\label{it:65Slack} every $7$-vertex graph $H$ with slack $c_H=0$ admits a homomorphism to~$B$.
\end{enumerate}
\item \label{it:65Unique} The vector  $\V \alpha$ is the unique (up to an automorphism of $B$) maximiser of $p(F,\blow{B}{\V x})$ in~$\I S_{12}$.
\item\label{it:65Robust} The $F$-inducibility problem is robustly $B$-stable.
\item\label{it:65Perfect} The $F$-inducibility problem is perfectly $B$-stable.
\end{enumerate}
\end{theorem}

\begin{proof}[Proof of \cref{thm:graph_65}]
Item \ref{it:65a} is verified by the provided script and certificate.

Next, we prove Item~\ref{it:65Unique}, namely that $\V\alpha$ is the unique (up to an automorphism of $B$) maximiser. While a direct proof is possible, we decided to present a calculation-free proof based on extra properties of our flag algebra certificate (as some intermediate claims will be used also in the proof of Item \ref{it:65Robust}). 

\begin{claim}\label{cl:65UniqueHom} 
\begin{enumerate}[(a)]
\item\label{it:65C4} The $4$-cycle $C_4$ admits a unique homomorphism to $C_6^*$, up to an automorphism of $C_6^*$. Moreover, its image uniquely identifies all vertices of $C_6^*$.
\item\label{it:65SigmaInB}
 The graph $\sigma$ (which is a 4-cycle plus an isolated vertex) admits a unique homomorphism to $B$, up to an automorphism of~$B$. Moreover, it maps $C_4\subseteq \sigma$ to one copy of $C_6^*$, uniquely identifying (inside $B$) all vertices of this $C_6^*$, and maps the isolated vertex to the other copy of~$C_6^*$.
 \end{enumerate}\end{claim}

\begin{claimproof} 
By looking at the graph complements in Item~\ref{it:65C4}, we have to map the matching of size 2 to the (loopless) graph $K_3\sqcup K_3$; clearly, these two edges must use two different copies of~$K_3$ and all such maps are the same up to an automorphism of the target pattern. Also, with $C_6^*$ labelled as in~\eqref{eq:65B}, the 6 vertices of $C_6^*$ have the following neighbourhoods in the set $[4]$ (which spans a $4$-cycle in $C_6^*$):
 \begin{equation}\label{eq:6Parts}
 \{0,1,3\},\ \{0,1,2\},\ \{1,2,3\},\ \{0,2,3\},\ \{1,3\},\ \{0,2\}.
 \end{equation}
 These sets are pairwise different, so every homomorphic image of $C_4$ in $C_6^*$ uniquely identifies all 6 vertices. This proves the first part.

Every homomorphism of $\sigma$ to $B$ has to map the connected subgraph $C_4\subseteq \sigma$ to a copy of $C_6^*$ in~$B$. Thus the second part follows by the first part, the fact that each set in~\eqref{eq:6Parts} is non-empty, and
the vertex-transitivity of $C_6^*$. \end{claimproof}

Take any maximiser $\V x\in\I S_{12}$ of $p(F,\blow{B}{\V x})$. Let $n\to\infty$ and let $G=\blow{B}{V_0,\ldots,V_{11}}$ be a blowup of $B$, with  $|V_j|=x_jn+O(1)$ for $x_j>0$ and $V_j=\emptyset$ for $x_j=0$. Thus $G$ is almost extremal and both sides of the flag algebra identity~\eqref{eq:FAMain} (that proves the sharp upper bound on $\lambda_F$ with $N=7$) are $o(n^7)$.

Since $C_4$ is an induced subgraph of $F$, we have that $t(C_4,G)\ge t(C_4,F)t(F,G)=\Omega(1)$, that is, the graph $G$ must have $\Omega(n^4)$ induced $4$-cycles. So assume, by applying an automorphism of $B$, that $x_0,x_1,x_2,x_3$ are all non-zero. We must have $x_6+\ldots+x_{11}>0$, say $x_6>0$, as otherwise we would not have any single copy of $F$.

Consider $\Omega(n^5)$ embeddings $f:\sigma\to G$ with $f(i)\in V_i$ for $i\in [4]$ and $f(4)\in V_6$. The corresponding vector $\V v_{(G,f)}^{\sigma,6}$, which is the same for all these $f$, must be $o(1)$-close to the kernel of $X^\sigma$ as otherwise the right-hand side of the flag algebra identity~\eqref{eq:FAMain}
 would be $\Omega(n^7)$. 
Thus the following claim implies
that $x_0=\ldots=x_5$, 
since each of the corresponding parts $V_0,\ldots,V_5$ is uniquely identified by its adjacencies to $f([4])$ by \cref{cl:65UniqueHom}.

\begin{claim}\label{cl:65Kernel}
For every vector $\V v$ in the kernel of $X^\sigma$, its $6$ coordinates that correspond to those $H\in\C F_6^\sigma$ where the intersection of the neighbourhood of the unlabelled vertex with $V(\sigma)=[5]$ is a set listed in~\eqref{eq:6Parts}, are all the same.
\end{claim}
\begin{claimproof}
 We evaluate~\eqref{eq:FAMain}  
 on growing blowups of $B$ where the part sizes are given by the vector $\V \alpha$ defined in~\eqref{eq:x}. There are two kinds of embeddings of $\sigma$, depending on which copy of $C_6^*$ contains the image of $C_4\subseteq \sigma$. They give, in the limit as $n\to\infty$, two non-zero vectors in the kernel of $X^\sigma$, which are different. (For example, the $H$-coordinates of these vectors are $1-\alpha$ and $\alpha$, where $H\in\C F_6^\sigma$ is the $\sigma$-flag in which the unlabelled vertex is isolated.) 
 Since the kernel of $X^\sigma$ has dimension 2 by Item~\ref{it:65Sigma} of the theorem, these two vectors span it. The claim follows since the 6 coordinates of interest are all equal in each vector (namely,
 are  $\alpha/6$ in the first vector and  $(1-\alpha)/6$ in the second vector).
\end{claimproof}

Next, let us show that $x_6=\ldots=x_{11}$. Recall that $x_6+\ldots+x_{11}>0$.
If $G[V_6\cup\ldots\cup V_{11}]$ contains no induced $C_4$ then all embeddings of $F$ into $G$ send the 4-cycle to the first 6 parts and the isolated vertex to the last 6 parts; however, by replacing each of $x_6,\ldots,x_{11}$ by their arithmetic mean, we add $\Omega(n^6)$ new copies of $F$, a contradiction to the maximality of~$\V x$. Thus some 4 indices in $\{6,\ldots,11\}$ that span $C_4$ in $B$ are positive and the same argument as before gives that $x_6=\ldots=x_{11}$.

We conclude that the vector $\V x$ consists of 6 entries $x/6$ followed by 6 entries $(1-x)/6$. Let $F'$ be obtained from $F$ by removing the isolated vertex. Thus $F$ is $T_{2,3}$ with an edge added to the larger part. The number of homomorphisms of $F'$ to $C_6^*$ is equal to $72$, which may be easier to see by passing to complements (when we map $T_{1,2}\sqcup K_2$ to $K_3\sqcup K_3$). Thus the homomorphism density of $F'$ in (uniform blowups of) $C_6^*$ is $72/6^5$. Therefore, by~\eqref{eq:t}, we have
 \[ 
  p(F,G)= \frac{6!}{4}\, t(F,G) =  180  \cdot \frac{72}{6^5} \paren{x^5(1-x)+(1-x)^5x}+o(1).
 \]
 By \Cref{pr:ClaimF3}, this polynomial is maximal if and only if $x\in\{\alpha,1-\alpha\}$. This implies  Item~\ref{it:65Unique}.

Next, let us prove Item \ref{it:65Robust}. Suppose on the contrary that robust $B$-stability does not hold, that is, for every real $C>0$ there is a graph $G$ violating~\eqref{eq:RobustStab}. For brevity, we will use asymptotic notation with respect to $C\to\infty$. Let $n:=v(G)$. Thus $n\ge C$ also tends to infinity and, since~\eqref{eq:RobustStab} fails, we have that $p(F,G)\ge \lambda_F-1/C=\lambda_F+o(1)$, that is, $G$ is almost extremal. By evaluating the flag algebra identity~\eqref{eq:FAMain} 
on the graph $G$ and using Item~\ref{it:65Slack}, we conclude that \begin{equation}\label{eq:65RobustH}
 p(H,G)\le O\left(\max\{1/n,\,\lambda_F-p(F,G)\}\right),\quad\mbox{for all $H\in \C F_7^0$ with no homomorphism to $B$}.
 \end{equation} 
\hide{every $7$-vertex graph $H$ with $c_H>0$ satisfies \XL{$\lambda_F(\blow{B}{\V x}) \rightarrow p(F,\blow{B}{\V x})$}
\begin{equation}\label{eq:65RobustH}
 p(H,G)\le O\left(\max\{1/n,\lambda_F-\lambda_F(G)\}\right),\quad\mbox{for every $H\in \C F_7^0$ with $c_H>0$}.
 \end{equation}
 By Item~\ref{it:65Slack}, the inequality in~\eqref{eq:65RobustH} applies to every $7$-vertex graph $H$ that does not admit a homomorphism to the pattern~$B$. }
 
 Our proof will be based on this, as well as Claim~\ref{cl:65UniqueHom}.
 Informally speaking, we will use a typical copy of $\sigma$ in $G$ to identify the first 6 parts $V_0,\ldots, V_5
\subseteq V(G)$ of a blowup of $B$ and derive by~\eqref{eq:65RobustH}  that ``most'' pairs intersecting these parts have the correct adjacency in $G$. Then we will argue that $V(G)\setminus \cup_{j=0}^5 V_j$ spans many induced copies of $C_4$ and use a typical copy of a such $4$-cycle (extended by a typical vertex of $\cup_{j=0}^5 V_j$ to form a copy of $\sigma$) to identify the remaining 6 parts. 
 
Let us provide details. For every embedding $g:\sigma\to G$, we define 6 disjoint subsets $V_0^g,\ldots,V_5^g\subseteq V(G)$ using the 6 possible attachments in~\eqref{eq:6Parts}; formally, for $j\in [6]$, we let 
\[
V_j^g:=\left\{u\in V(G)\setminus g([5])\colon g^{-1}(\Gamma_G(u))=\Gamma_{C_6^*}(j)\right\}.
\]
Also, we let $R^g:=V(G)\setminus  (g([5])\cup (\cup_{j=0}^5 V_j^g))$. (We leave $g([5])$ outside of these parts for convenience.) Let 
\[
W^g:=\left(E(G)\bigtriangleup \blow{C_6^*}{V_0^g,\ldots,V_5^g}\right) \cup G[\cup_{j=0}^5 V_j^g,R^g].
\]
 In other words, $W^g$ is defined to consist of those pairs in $V(G)\setminus g([5])$ that intersect $\cup_{j=0}^5 V_j^g$ and whose adjacencies in the graph $G$ and the blowup $\blow{(C_6^*\sqcup K_1)}{V_0^g,\ldots,V_5^g,R^g}$ mismatch. We call such pairs \emph{$g$-wrong}. 
 
 By Claim~\ref{cl:65UniqueHom}, every $g$-wrong pair $xy$ together with $g([5])$ spans a subgraph in $G$ that does not admit a homomorphism to~$B$. (In fact, we could have taken $g([4])$, the image of $C_4\subseteq \sigma$, instead of $g([5])$ in the previous statement.) Since $\sigma$ is an induced subgraph of $F$, there are $\Omega(n^5)$ embeddings of $\sigma$ into~$G$. 
  Thus, for a uniformly random embedding $g:\sigma\to G$, the expectation $\I E_g|W_g|$, by~\eqref{eq:65RobustH}, is at most 
  \begin{equation}\label{eq:65|W|}
  O\left(\max\left\{n,\,(\lambda_F-p(F,G)\right)n^2\right\}).
  \end{equation} 

Thus we can find $g$ such that $|W^g|$ is small. This alone is not enough for our forthcoming arguments since we also need exclude the case that $|\cup_{j=0}^5 V_j^g|=o(n)$ (which may happen if e.g.\ the image of $g$ happens to contain an isolated cycle in $G$) when we have not identified any significant part of the structure of $G$. With this in mind, let us show that almost every embedding $g:\sigma\to G$ satisfies that 
\begin{equation}\label{eq:GoodGamma02}
\left|(\Gamma_G(g(0))\cap \Gamma_G(g(2)))\setminus R^g\right|=\Omega(n).
\end{equation}
(Recall that $\{0,2\}$ is a diagonal of $C_4\subseteq \sigma$.)
The inequality in~\eqref{eq:GoodGamma02} can fail in two ways. First, it may happen that $\Gamma_G(g(0))\cap \Gamma_G(g(2))$ has $o(n)$ vertices. We can generate  all such embeddings $g$ by first choosing the image $\{g(0),g(2)\}$, after which we have 
 $o(n)$ choices for $g(1)$ (and also for $g(3)$). Second, the set $U:=\Gamma_G(g(0))\cap \Gamma_G(g(2))\cap R^g$ may have $\Omega(n)$ vertices. Then every vertex $v$ of this set (by not being assigned to $V_1^g\cup V_3^g\cup V_5^g$) has ``wrong'' adjacencies to the set $\{g(1),g(3)\}$.  This means that the graph induced by the 6-set $\{v\}\cup g([5])$ in $G$ does not admit a homomorphism to~$B$. Of course, this applies to every 7-vertex set containing it. In this case, $g$ generates $\Omega(n^2)$ such sets so, by~\eqref{eq:65RobustH}, we have $o(n^5)$ such embeddings, as desired.

Thus we can fix an embedding $g:\sigma\to G$ such that $|W_g|$ is at most, say, twice its expected value and~\eqref{eq:GoodGamma02} holds.
Given this (fixed) map $g$, we define $J:=\blow{C_6^*}{V_0^g,\ldots,V_5^g}$. Let $X:=V(J)=\cup_{j=0}^5 V_j^g$ and $x:=|X|/n$. By~\eqref{eq:GoodGamma02}, we have $x=\Omega(1)$. 

Let us show that $p(C_4,J)=\Omega(1)$. Suppose on the contrary that this is false. Since (up $o(n^2)$ wrong pairs) $G[X]$ is essentially $J$ and $G$ has essentially no edges between $X$ and $V(G)\setminus X$, all but $o(n^6)$ copies of $F$ in $G$ have at most one vertex from $X$. But then all such copies remain when we replace $G[X]$ by a uniform blowup of $C_6^*$ while we get $\Omega(n^6)$ new copies of $F$ (those that have exactly $5$ vertices in $X$). This contradicts the fact that $G$ is  almost extremal. 

By applying Claims~\ref{cl:65UniqueHom} and~\ref{cl:65Kernel} to the $\Omega(n^5)$ embeddings of $\sigma$ into $G$ that map $C_4\subseteq \sigma$ into $J$, we conclude that $|V_j|=(x/6+o(1))n$ for each $j\in [6]$.

Note that $1-x=\Omega(1)$ as otherwise $p(F,G)=o(1)$, a contradiction. Also, the set $Y:=V(G)\setminus X$ induces $\Omega(n^4)$ copies of $C_4$ for otherwise all but $o(n^6)$ copies of $F$ have exactly one vertex in $Y$ and, by replacing $G[Y]$ by a uniform blowup of $C_6^*$, we increase $p(F,G)$ by $\Omega(1)$, a contradiction. 

As before, we can find an embedding $f:\sigma \to G$ which maps $C_4\subseteq \sigma$ inside $Y$ so that $|\cup_{j=0}^5 V_j^{f}|=\Omega(n)$ and the size of the set $W^f$ of $f$-wrong pairs is upper bounded by~\eqref{eq:65|W|}. Additionally, by $G$ having at most $|W^g|$ edges between $Y$ and $V(G)\setminus Y=X$, we can assume that $|(\cup_{j=0}^5 V_j^f)\setminus Y|=O(\max\{1,(\lambda_F-p(F,G))n\})$. For $j\in [6]$, let $V_j:=V_j^g$ and, for $j=6,\ldots,11$, let $V_j:=V_j^f\cap Y$. Also, let $R:=Y\setminus (\cup_{j=0}^5 V_j^f)$. 

Similarly as before, the graph $J':=C_6^*(V_6,\ldots,V_{11})$ has $\Omega(n^4)$ induced $4$-cycles.
In particular, we can find $\Omega(n^3)$ embeddings $P_3\sqcup P_3\to G$ which are also embeddings into $J\sqcup J'$. If we add an isolated vertex to $P_3\sqcup P_3$ then the obtained graph $H$ does not admit a homomorphism to~$B$. Thus, by~\eqref{eq:65RobustH}, $|R|\le O(\max\{1,(\lambda_F-p(F,G))n\})$. It follows that we can make $G$ into a blowup of $B$ by changing at most
 $$
 |W^g|+\left|W^f\cap \binom{Y}{ 2}\right|+n\,|R|=O(\max\{n,(\lambda_F-p(F,G))n^2\}),
 $$
 edges. This contradicts our choice of $G$ when $C$ is sufficiently large, proving Item~\ref{it:65Robust}.

Finally, in order to prove Item~\ref{it:65Perfect}, we prove perfect $\O B$-stability for the $\O F$-inducibility problem
by applying \cite[Theorem 5.13]{PikhurkoSliacanTyros19}. Clearly, taking graph complements preserves perfect stability (and we do this so that the involved pattern has no loops and thus \cite[Theorem 5.13]{PikhurkoSliacanTyros19} can be formally applied). To derive  perfect $\O B$-stability from the theorem,
 it suffices to check that the $\O F$-inducibility problem is robustly $\O B$-stable (which clearly follows from Item~\ref{it:65Robust}) and  the pattern $\O B$ is $\lambda_{\O F}$-minimal (that is, if we delete any vertex from $\O B$ then the maximum limit density of $\O F$ in $\O B$-blowups strictly decreases). By the compactness of $\I S_{12}$, the last property is equivalent to every maximiser of $\lambda_{\O F}(\blow{\O B}{\V x})$ having no zero entries. This holds here by Item~\ref{it:65Unique} since the graph complementation preserves the set of maximisers. Thus, by applying \cite[Theorem 5.13]{PikhurkoSliacanTyros19} and taking graph complements, the $F$-inducibility problem is perfectly $B$-stable.\end{proof}

\section{Conjectures}
\label{sec:conjectures}

In this section we present the still open cases when the upper and lower bounds on $\lambda_{F}$ are very close to each other but we could not round the floating-point matrices returned by computer. It is possible that the exact value of $\lambda_{F}$ can be determined in some of these cases by the `plain' flag algebra method.

The candidate graphons are specified by pairs $(P, E)$, resulting in $W(P, E)$. The vector $P = (p_0, p_1, \ldots, p_{k-1})$ consists of positive reals summing to $1$, representing a measurable partition of the vertex space into $k$ parts $V_0, V_1, \ldots, V_{k-1}$, with the measure of $V_i$ being~$p_i$. The map $E : \{(i, j) : 0 \leq i \leq j < k\} \to [0, 1] \cup \{R\}$ assigns to each unordered pair $\{i, j\}$ the constant value of the graphon on $V_i\times V_j$ (and, by symmetry, on $V_j\times V_i$). Also, $E(i, j) = R$ is allowed only for  diagonal entries (when $i=j$) and indicates that the graphon is \emph{recursive} inside $V_i$, that is,  $V_i$ contains a scaled copy of the entire graphon. 

We systematically looked through all small candidate graphons to find the conjectures. In particular, the family we searched through contained every $W(P, E)$ with $|P| \leq 6$ and $E$ being $\{0, 1, R\}$-valued. In addition, we searched through $W(P, E)$ with $|P| \leq 4$ with off-diagonal $E$ entries in $[0, 1]$ and diagonal entries restricted to $0$ or $1$. Note that for a given linear target $\lambda$, the expression $\lambda (W(P, E))$ is a polynomial in the parameters $(P, E)$ if $E$ does not map to $R$, otherwise it is a quotient of polynomials in $(P, E)$. For each $6$-vertex graph $F$ from this section, we numerically approximated $\max_{P, E} \lambda_F (W(P, E))$ over each template using gradient descent in the parameters.

After this systematic search, we carried out an additional AI-assisted heuristic search for improved candidate constructions for the unresolved cases using ChatGPT 5.5 Pro. The AI assistant was used to suggest candidate templates and numerical parameters; every resulting construction and objective value reported below was independently verified, and this verification is included in the ancillary notebook. This refinement improved several entries in Table~\ref{tab:guess6}, in particular, it improved all constructions except $F_{35}$, $F_{39}$, $F_{40}$, $F_{41}$, $F_{54}$, and $F_{63}$. Additionally, for $F_{23}$ and $F_{67}$, the improved constructions agreed with the numerical flag algebra upper bounds, hence we moved them to the conjectures with the exact algebraic descriptions.

The resulting conjectures are summarised below, while the best guesses are in Section~\ref{sec:Concluding}.

\setlength{\tabcolsep}{3pt}
\renewcommand{\arraystretch}{2.25}

\begin{longtable}{r c c C{2.4cm} c c c}
\caption{Conjectures}\label{tab:conj6}\\
\HIDE{\hline
$i$ & $F_i$ & $\overline{F_i}$ & $W$ & $\leq \lambda$ & $\lambda \leq$ & $N$ \\
\hline
\endfirsthead
\hline
$i$ & $F_i$ & $\overline{F_i}$ & $W$ & $\leq \lambda$ & $\lambda \leq$ & $N$ \\
\hline
\endhead
\hline
\endfoot
\ConjRow{9}{0/1,0/2,0/4,0/5}{0/3,1/2,1/3,1/4,1/5,2/3,2/4,2/5,3/4,3/5,4/5}{0.175872,0.132313,0.035453,0.656362}{}{0/3/{1}, 1/2/{1}}{0.1677211390}{0.1677211390}{34003165/202736311}{N}{8}
\ConjRow{11}{0/1,0/3,0/5,1/5}{0/2,0/4,1/2,1/3,1/4,2/3,2/4,2/5,3/4,3/5,4/5}{0.492901,0.507099}{}{0/0/{1}, 0/1/{0.12115}}{0.1265831105}{0.1265831107}{66640818/526459001}{N}{7}
\ConjRow{19}{0/1,0/2,0/3,0/4,1/5}{0/5,1/2,1/3,1/4,2/3,2/4,2/5,3/4,3/5,4/5}{0.114119,0.189303,0.218049,0.083457,0.395073}{}{0/1/{1}, 1/2/{1}, 0/4/{1}, 2/3/{1}, 3/4/{1}}{0.0873819110}{0.0873819111}{327955157/3753124105}{N}{8}
\ConjRow{20}{0/1,0/2,0/4,1/2,1/5}{0/3,0/5,1/3,1/4,2/3,2/4,2/5,3/4,3/5,4/5}{0.47321,0.52679}{}{0/0/{1}, 0/1/{0.262633}}{0.0543206982}{0.0543206983}{5034053/92672833}{N}{8}
\ConjRow{22}{0/1,0/3,0/4,1/2,1/5}{0/2,0/5,1/3,1/4,2/3,2/4,2/5,3/4,3/5,4/5}{0.166681,0.166681,0.333319,0.333319}{0,1}{0/1/{1}, 1/2/{1}, 0/3/{1}}{0.0617310419}{0.0617310537}{21298978/345028583}{N}{8}
\ConjRow{23}{0/2,0/4,0/5,1/2,1/5}{0/1,0/3,1/3,1/4,2/3,2/4,2/5,3/4,3/5,4/5}{0.073032,0.073032,0.213484,0.426968,0.213484}{}{0/1/{0.832503}, 2/3/{0.796813}, 3/4/{0.796813}}{0.1116212578}{0.1116212592}{}{N}{8}
\ConjRow{29}{0/3,0/4,1/3,1/5,4/5}{0/1,0/2,0/5,1/2,1/4,2/3,2/4,2/5,3/4,3/5}{0.1665567,0.1665567,0.1665567,0.1665567,0.1665567,0.03344,0.03344,0.03344,0.03344,0.03344}{0,1,2,3,4,5,6,7,8,9}{0/1/{1}, 1/2/{1}, 2/3/{1}, 3/4/{1}, 4/0/{1}, 5/6/{1}, 6/7/{1}, 7/8/{1}, 8/9/{1}, 9/5/{1}}{0.0155311392}{0.0155317467}{5836020/375747823}{N}{8}
\ConjRow{34}{0/1,0/2,0/3,0/4,1/2,1/5}{0/5,1/3,1/4,2/3,2/4,2/5,3/4,3/5,4/5}{0.111112,0.222221,0.111112,0.222221,0.111112,0.222221}{0,2,4}{0/1/{1}, 0/2/{1}, 4/3/{1}, 0/4/{1}, 1/2/{1}, 2/3/{1}, 2/4/{1}, 0/5/{1}, 4/5/{1}}{0.0650311042}{0.0650311104}{23290569/358145030}{N}{8}
\ConjRow{57}{0/1,0/3,0/4,0/5,1/4,1/5,4/5}{0/2,1/2,1/3,2/3,2/4,2/5,3/4,3/5}{0.666005,0.333995}{}{0/0/{1}, 0/1/{0.126425}}{0.1297872768}{0.1297872775}{25757455/198459013}{N}{7}
\ConjRow{62}{0/1,0/3,0/4,0/5,1/2,1/3,4/5}{0/2,1/4,1/5,2/3,2/4,2/5,3/4,3/5}{0.222221,0.111112,0.222221,0.111112,0.222221,0.111112}{1,3,5}{0/1/{1}, 1/2/{1}, 2/3/{1}, 3/4/{1}, 4/5/{1}, 5/0/{1}, 1/3/{1}, 1/5/{1}, 3/5/{1}, 0/0/{1}, 2/2/{1}, 4/4/{1}}{0.0650311042}{0.0650311132}{42639211/655674014}{N}{8}
\ConjRow{67}{0/2,0/3,0/5,1/2,1/4,1/5,4/5}{0/1,0/4,1/3,2/3,2/4,2/5,3/4,3/5}{0.237293,0.237293,0.262707,0.262707}{}{0/0/{1}, 0/1/{1}, 0/3/{0.468005}, 1/2/{0.468005}, 2/3/{1}, 1/1/{1}}{0.0390894375}{0.0390894387}{}{N}{8}
\ConjRow{76}{0/1,0/3,0/4,0/5,1/3,1/4,3/5,4/5}{0/2,1/2,1/5,2/3,2/4,2/5,3/4}{0.277411,0.277411,0.277411,0.055922,0.055922,0.055922}{}{0/2/{1}, 0/1/{1}, 1/2/{1}, 3/4/{1}, 3/5/{1}, 4/5/{1}}{0.1490854569}{0.1490854570}{129569313/869094247}{N}{8}

}
\end{longtable}

\begin{conjecture}\label{conj:9}
    Let $F_{9} := (6, \{01, 02, 04, 05\})$. Then \[
    \lambda_{F_{9}} =-\frac{25}{8} \alpha^{6} - \frac{125}{48} \alpha^{4} + \frac{425}{288} \alpha^{2} + \frac{95}{576},
    \] with the optimum attained at the graphon $W(P, E)$ where \[
    P = \left( \beta \gamma, \beta (1-\gamma), (1-\beta)\gamma, (1-\beta)(1-\gamma) \right) 
    \] and $E(0,1)=1, E(2,3)=1$, otherwise zero. Here 
    \begin{align*}
        \beta &= \frac{135}{8} \alpha^{7} + \frac{369}{16} \alpha^{5} + \frac{225}{32} \alpha^{3} - \frac{489}{64} \alpha + \frac{1}{2}, \\ 
        \gamma &= \frac{135}{8} \alpha^{7} + \frac{369}{16} \alpha^{5} + \frac{225}{32} \alpha^{3} - \frac{425}{64} \alpha + \frac{1}{2}
    \end{align*} and $\alpha$ is the root around $\longreal{0.043559135202965838704769}$ of the polynomial $$0 = 1296 \alpha^{8} + 1728 \alpha^{6} + 504 \alpha^{4} - 528 \alpha^{2} + 1.$$
\end{conjecture}

\begin{conjecture}\label{conj:11}
    Let $F_{11} := (6, \{01, 03, 05, 15\})$. Then \begin{align*}
        \lambda_{F_{11}} = & -\frac{1762538900422265}{308668025241} \alpha^{4} - \frac{26188848008374735}{617336050482} \alpha^{3} \\ & + \frac{11676184155575045}{308668025241} \alpha^{2} - \frac{2304679050460405}{205778683494} \alpha +   \frac{3253746756658615}{3704016302892},
    \end{align*}
    with the optimum attained at the graphon $W(P, E)$ where $P = \left( \beta, 1-\beta\right)$ and $E(0,0)=1, E(0,1)=\alpha$, otherwise zero. Here 
    \begin{align*}
        \beta &= \frac{9}{361} \alpha^{4} - \frac{49}{722} \alpha^{3} + \frac{113}{361} \alpha^{2} - \frac{278}{361} \alpha + \frac{210}{361}, \\ 
        0 &= \alpha^{5} - \frac{13}{18} \alpha^{4} + \frac{64}{9} \alpha^{3} - \frac{52}{9} \alpha^{2} + \frac{5}{3} \alpha - \frac{7}{54}
    \end{align*} and $\alpha$ is around $\longreal{0.12115013822476591245544}$.
\end{conjecture}

\begin{conjecture} \label{conj:19}
    Let $F_{19} := (6, \{01, 02, 03, 04, 15\})$. Then $\lambda_{F_{19}} = \alpha \approx 0.087381910951487916660924,$ with the optimum attained at the graphon $W(P, E)$, where 
    \[
    \begin{aligned}
    P &= (\beta_0, \beta_1, \beta_2, \beta_3, \beta_4) \\
      &\approx (\longreal{0.39507240031226174845617}, \longreal{0.11411929017788614010104},\longreal{0.18930253359278366699069}, \longreal{0.21804932261115469857164},\longreal{0.08345645330591374588046}),
    \end{aligned}
    \] and $E(0,1)=1, E(1,2)=1, E(2,3)=1, E(3,4)=1, E(4,0)=1$, otherwise zero. Here $\beta_i$ denotes the unique maximiser, up to the natural symmetries of \[ \sum_{i\in \I Z/5\I Z}
    \beta_i \beta_{i+1} \beta_{i+2} \beta_{i+3} \left( \beta_{i}^2 + \beta_{i+3}^2 \right)
    \] under the constraint $\beta_i \geq 0$ and $\sum_i \beta_i = 1$.
\end{conjecture}
Note that in Conjecture~\ref{conj:19}, the $\alpha, \beta_i$ values are algebraic, however we do not record their exact expressions as the smallest description we found places them in the splitting field of a degree-$221$ polynomial.

\begin{conjecture}\label{conj:20}
    Let $F_{20} := (6, \{01, 02, 04, 12, 15\})$. Then \begin{align*}
        \lambda_{F_{20}} = & -\frac{397582286503493}{168915385692} \alpha^{4} + \frac{56388167496659}{18768376188} \alpha^{3}  \\ & - \frac{717902501474065}{506746157076} \alpha^{2} + \frac{47085600013265}{168915385692} \alpha -  \frac{2364666553933}{126686539269},
    \end{align*}
    with the optimum attained at the graphon $W(P, E)$ where $P = \left( \beta, (1-\beta)\right)$ and $E(0,0)=1, E(0,1)=\alpha$, otherwise zero. Here 
    \begin{align*}
        \beta &= \frac{9}{460} \alpha^{4} - \frac{43}{460} \alpha^{3} + \frac{31}{92} \alpha^{2} - \frac{223}{276} \alpha + \frac{229}{345}, \\ 
        0 &= 27 \alpha^{5} - 75 \alpha^{4} + 207 \alpha^{3} - 185 \alpha^{2} + 66 \alpha - 8
    \end{align*} and $\alpha$ is around $\longreal{0.26263329690919386542210}$.
\end{conjecture}

\begin{conjecture}\label{conj:22}
    Let $F_{22} := (6, \{01, 03, 04, 12, 15\})$. 
    Then $\lambda_{F_{22}} = 60 \alpha^2 (\frac{1}{2} - \alpha)^3$, with the optimum attained at the graphon $W(P, E)$ where 
    $P = \left( \alpha, \alpha, \frac{1}{2} - \alpha, \frac{1}{2} - \alpha \right) $ and 
    $E(0, 1) = 1, E(0, 2) = 1, E(1, 3) = 1, E(0, 0) = R, E(1, 1) = R$, otherwise zero. Here $\alpha$ is the unique root in $(0, 1/2)$ of $4\alpha^6 - 6\alpha + 1 = 0$.
    Numerically,
    \[
        \lambda_{F_{22}}\approx\longreal{0.061731041963616},
        \qquad
        \alpha\approx\longreal{0.16668096300266456}.
    \]
\end{conjecture}

\begin{conjecture}\label{conj:23}
    Let $F_{23} := (6, \{02, 04, 05, 12, 15\})$. 
    Then $\lambda_{F_{23}} \approx \longreal{0.111621257872924},$ with the optimum attained at the graphon $W(P,E)$ where
    \[
        P = (\alpha,\alpha,\beta,2\beta,\beta)
        \approx \left(
        \longreal{0.07303173354167806715739804},
        \longreal{0.07303173354167806715739804},
        \longreal{0.21348413322916096642130098},
        \longreal{0.42696826645832193284260196},
        \longreal{0.21348413322916096642130098}
        \right),
    \]
    and
    \[
        E(0,1)=\gamma,\qquad E(2,3)=\delta,\qquad E(3,4)=\delta,
    \]
    otherwise zero. Here
    \[
        2\alpha+4\beta=1,\qquad
        \alpha \approx\longreal{0.07303173354167806715739804},\qquad
        \gamma \approx\longreal{0.8325025367933598053297472},\qquad
        \delta \approx\longreal{0.7968129023368565201152813}.
    \]
    The values $\alpha,\gamma,\delta$ are algebraic; their smallest description is in a splitting field of a degree-$60$ polynomial.
\end{conjecture}

The following is a corrected version of our previous conjecture on $F_{29}$ with the better construction suggested by Bo Zhang \cite{Zhang26}.

\begin{conjecture}\label{conj:29}
    Let $F_{29} := (6, \{03, 04, 13, 15, 45\})$. 
    Then $\lambda_{F_{29}} = \frac{3(1 + 778\alpha - 1561\alpha^3)}{20306},$ with the optimum attained at the graphon which consists on two disjoint copies of the graphon $W(P, E)$, of total measures $\beta$ and $1 - \beta$, where 
    $P = \left( \frac15, \frac15, \frac15, \frac15, \frac15 \right) $ and 
    $E(0, 1) = 1, E(1, 2) = 1, E(2, 3) = 1, E(3, 4) = 1, E(4, 0) = 1, E(0, 0) = R, E(1, 1) = R, E(2, 2) = R, E(3, 3) = R, E(4, 4) = R$, otherwise $0$. Here $\alpha$ is the unique root in $(0, 1/4)$ of $4683\alpha^2 - 6239\alpha + 778 = 0$, and $\beta = \frac{1 + \sqrt{1 - 4\alpha}}{2}$. Numerically,
    \[
        \lambda_{F_{29}}\approx\longreal{0.01553113928107167769},
        \qquad
        \alpha\approx\longreal{0.139255117611548696},
        \qquad
        \beta\approx\longreal{0.832783536835059351}.
    \]
\end{conjecture}

Notice that in the picture entry corresponding to \cref{conj:29} in \cref{tab:conj6}, namely for $F_{29}$, the two disjoint copies of iterated blowups of $C_5$ should be understood independently, as each part only contains a recursive iterated blowup of $C_5$ and not the whole construction. Observe also that valid copies of $F_{29}$ may be present in a single one of these disjoint subgraphons. Indeed, $5$ vertices forming a $C_5$ may be present in a top-level part, and the remaining vertex in one of the $2$ non-neighbouring parts of the same subgraphon.

\begin{conjecture}\label{conj:34}
    Let $F_{34} := (6, \{01, 02, 03, 04, 12, 15\})$. 
    Then $\lambda_{F_{34}} = 480\alpha^2 (\frac13 - \alpha)^3,$ with the optimum attained at the graphon $W(P, E)$ where 
    $P = \left( \alpha, \frac13 - \alpha, \alpha, \frac13 - \alpha, \alpha, \frac13 - \alpha \right) $ and 
    $E(0, 1) = 1, E(0, 2) = 1, E(0, 3) = 1, E(0, 4) = 1, E(1, 4) = 1, E(2, 3) = 1, E(2, 4) = 1, E(2, 5) = 1, E(4, 5) = 1, E(0, 0) = R, E(2, 2) = R, E(4, 4) = R$, otherwise zero. Here $\alpha$ is the unique root in $(0, 1/3)$ of $6\alpha^6 - 9\alpha + 1 = 0$. Numerically,
    \[
        \lambda_{F_{34}}\approx\longreal{0.065031104299294},
        \qquad
        \alpha\approx\longreal{0.11111236564704499}.
    \]
\end{conjecture}

\hide{
\begin{conjecture}\label{conj:43}
    Let $F_{43} := (6, \{01, 02, 05, 13, 15, 45\})$. 
    Then $\lambda_{F_{43}} = \frac{24}{1555},$ with the optimum attained at the graphon $W(P, E)$ where 
    $P = \left( \frac{1}{6}, \frac{1}{6}, \frac{1}{6}, \frac{1}{6}, \frac{1}{6}, \frac{1}{6} \right) $ and 
    $E(0, 2) = 1, E(1, 2) = 1, E(1, 3) = 1, E(1, 5) = 1, E(2, 3) = 1, E(3, 4) = 1, E(0, 0) = R, E(1, 1) = R, E(2, 2) = R, E(3, 3) = R, E(4, 4) = R, E(5, 5) = R$, otherwise zero. 
\end{conjecture}

We also include (for completeness) the special case $m=6$ of the well-known conjecture of 
Pippenger and Golumbic~\cite{PippengerGolumbic75} that, for every $m\ge 5$, the inducibility constant of the $m$-cycle is $m!/(m^m-m)$; see~\cite{BaloghHuLidickyPfender16,HefetzTyomkyn18,KralNorinVolec19} for partial results towards the general conjecture.

\begin{conjecture}[Pippenger and Golumbic~\cite{PippengerGolumbic75}]\label{conj:51}
    Let $F_{51} := (6, \{02, 04, 12, 13, 35, 45\})$. 
    Then $\lambda_{F_{51}} = \frac{24}{1555},$ with the optimum attained at the graphon $W(P, E)$ where 
    $P = \left( \frac{1}{6}, \frac{1}{6}, \frac{1}{6}, \frac{1}{6}, \frac{1}{6}, \frac{1}{6} \right) $ and 
    $E(0, 3) = 1, E(0, 4) = 1, E(1, 2) = 1, E(1, 5) = 1, E(2, 4) = 1, E(3, 5) = 1, E(0, 0) = R, E(1, 1) = R, E(2, 2) = R, E(3, 3) = R, E(4, 4) = R, E(5, 5) = R$, otherwise zero. 
\end{conjecture}
}

\begin{conjecture}\label{conj:57}
    Let $F_{57} := (6, \{01, 03, 04, 05, 14, 15, 45\})$. Then \begin{align*}
        \lambda_{F_{57}} = & -\frac{36031628837696693165}{198512458926872706} \alpha^{4} + \frac{97077727483421103955}{297768688390309059} \alpha^{3} \\ & - \frac{505067184288007657625}{2382149507122472472} \alpha^{2} + \frac{4815168475066367665}{85076768111516874} \alpha -  \frac{10122272828133955255}{2382149507122472472},
    \end{align*}
    with the optimum attained at the graphon $W(P, E)$ where $P = \left( \beta, (1-\beta)\right)$ and $E(0,0)=1, E(0,1)=\alpha$, otherwise zero. Here 
    \begin{align*}
        \beta &= -\frac{27}{1771} \alpha^{4} - \frac{151}{3542} \alpha^{3} - \frac{1607}{10626} \alpha^{2} - \frac{320}{759} \alpha + \frac{3835}{5313} \\ 
        0 &= 162 \alpha^{5} - 33 \alpha^{4} + 248 \alpha^{3} - 341 \alpha^{2} + 142 \alpha - 13
    \end{align*} and $\alpha$ is around $\longreal{0.12642524715750795560578}$.
\end{conjecture}

\begin{conjecture}\label{conj:62}
    Let $F_{62} := (6, \{01, 03, 04, 05, 12, 13, 45\})$. 
    Then $\lambda_{F_{62}} = 480\alpha^2 (\frac13 - \alpha)^3,$ with the optimum attained at the graphon $W(P, E)$ where 
    $P = \left( \frac13 - \alpha, \alpha, \alpha, \frac13 - \alpha, \frac13 - \alpha, \alpha \right) $ and 
    $E(0, 1) = 1, E(0, 5) = 1, E(1, 2) = 1, E(1, 4) = 1, E(1, 5) = 1, E(2, 3) = 1, E(2, 4) = 1, E(2, 5) = 1, E(3, 5) = 1, E(0, 0) = 1, E(1, 1) = R, E(2, 2) = R, E(3, 3) = 1, E(4, 4) = 1, E(5, 5) = R$, otherwise zero. Here $\alpha$ is the unique root in $(0, 1/3)$ of $6\alpha^6 - 9\alpha + 1 = 0$. Numerically,
    \[
        \lambda_{F_{62}}\approx\longreal{0.065031104299294},
        \qquad
        \alpha\approx\longreal{0.11111236564704499}.
    \]
\end{conjecture}

\begin{conjecture}\label{conj:67}
    Let $F_{67} := (6, \{02, 03, 05, 12, 14, 15, 45\})$. Then
    \[
        \lambda_{F_{67}}
        =
        -\frac{3635}{5832}\alpha^{4}
        + \frac{124565}{46656}\alpha^{3}
        - \frac{154295}{46656}\alpha^{2}
        + \frac{13165}{11664}\alpha
        - \frac{25}{2916},
    \]
    with the optimum attained at the graphon $W(P,E)$ where
    \[
        P=\left(\beta,\beta,\frac{1}{2}-\beta,\frac{1}{2}-\beta\right)
    \]
    and
    \[
        E(0,0)=1,\quad E(0,1)=1,\quad E(1,1)=1,\quad E(2,3)=1,\quad
        E(0,3)=\alpha,\quad E(1,2)=\alpha,
    \]
    otherwise zero. Here
    \[
        \beta
        =
        \frac{7}{12}\alpha^{4}
        - \frac{79}{36}\alpha^{3}
        + \frac{283}{36}\alpha^{2}
        - \frac{775}{72}\alpha
        + \frac{15}{4},
    \]
    and $\alpha$ is the root around $\longreal{0.46800543305264288739873}$ of the polynomial
    \[
        0
        =
        \alpha^{5}
        - \frac{13}{3}\alpha^{4}
        + \frac{47}{3}\alpha^{3}
        - \frac{157}{6}\alpha^{2}
        + \frac{49}{3}\alpha
        - \frac{10}{3}.
    \]
    Numerically,
    \[
        \lambda_{F_{67}}\approx\longreal{0.039089437510310693712390},
        \qquad
        \beta\approx\longreal{0.23729269561009103376264}.
    \]
\end{conjecture}

\begin{conjecture}\label{conj:76}
    Let $F_{76} := (6, \{01, 03, 04, 05, 13, 14, 35, 45\})$. Then \[
    \lambda_{F_{76}} =\frac{200}{81} \alpha^{2} - \frac{200}{81} \alpha + \frac{40}{81},
    \] with the optimum attained at the graphon $W(P, E)$ where \[
    P = \left( \alpha/3, \alpha/3, \alpha/3, (1-\alpha)/3, (1-\alpha)/3, (1-\alpha)/3 \right) 
    \] and $E(0,1)=1, E(1,2)=1, E(2,0)=1, E(3,4)=1, E(4,5)=1, E(5,3)=1$, otherwise zero. Here $\alpha$ is the root around $\longreal{0.16776573020222127904066}$ of the polynomial
    \[
        0 = \alpha^{4} - 2 \alpha^{3} + \frac{7}{3} \alpha^{2} - \frac{4}{3} \alpha + \frac{1}{6}.
    \]
\end{conjecture}

\section{Concluding Remarks}
\label{sec:Concluding}

Here we present the numerical upper and lower bounds that we have in the remaining 27 cases where the value of $\lambda_{F_i}$ is still unknown. We just summarise them in Table~\ref{tab:guess6}.
The stated upper bounds in Table~\ref{tab:guess6} are the values returned by the \texttt{csdp} solver for the `plain' flag algebra SDP with $N=8$.
One should be able to turn the obtained floating-point matrices into mathematical proofs with small loss in the bound. However, we do not do this since we do not expect this to provide any insights. The firth column of Table~\ref{tab:guess6} presents numerical lower bounds on $\lambda_{F_i}$. Very recently,  Zhang \cite{Zhang26} found better constructions in 23 of these cases (in particular, making 6 new conjectures). 
A better lower bound for $\lambda_{F_{26}}$ was also recently discovered by Kyzlasov \cite{kyzlasov_2026_24341542} by considering iterated uniform blowups of the Petersen graph. Kyzlasov also proved that the bound obtained is tight among graphs containing no induced bull or an induced house.

Fox, Huang and Lee~\cite{FoxHuangLee15} and, independently, Yuster~\cite{Yuster19} proved that almost every $F$ of order $\kappa\to\infty$ is a \emph{fractalizer}. There are different possible definitions what a fractalizer is so let us use the broadest definition which states that
the value of $\lambda_F$ is attained by iterated uniform blowups of $F$, where we repeat recursion inside each part (without any claims regarding the uniqueness of the extremal graphon). The known values for $F_{43}$ and $F_{51}$ (and therefore also for their complements) are attained in this way. The results collected here show that, excluding the trivial case of $\O F_0=K_6$, there are no other 6-vertex fractalizers. It is not surprising that the value $\kappa=6$ is too small to indicate the general trend as $\kappa\to\infty$.

The flag algebra method can also be applied to the inducibility of graphs with at most $8$ vertices but we have not done any calculations in that direction.

\begin{longtable}{r c c C{2.4cm} c c c}
\caption{The remaining cases where the inducibility constant is unknown}\label{tab:guess6}\\
\HIDE{\hline
$i$ & $F_i$ & $\overline{F_i}$ & $W$ & $\leq \lambda$ & $\lambda \leq$ & $N$ \\
\hline
\endfirsthead
\hline
$i$ & $F_i$ & $\overline{F_i}$ & $W$ & $\leq \lambda$ & $\lambda \leq$ & $N$ \\
\hline
\endhead
\hline
\endfoot
\GuessRow{5}{0/1,0/4,1/5}{0/2,0/3,0/5,1/2,1/3,1/4,2/3,2/4,2/5,3/4,3/5,4/5}{0.069146,0.232714,0.232714,0.232714,0.232713}{}{0/0/{0.303462}, 0/1/{0.186078}, 0/2/{0.186072}, 0/3/{0.186072}, 0/4/{0.186066}, 1/2/{0.171807}, 1/3/{0.171807}, 1/4/{0.442063}, 2/3/{0.442063}, 2/4/{0.171807}, 3/4/{0.171807}}{0.0999820004}{0.1121327653}{}{N}{8}
\GuessRow{10}{0/1,0/2,0/4,1/5}{0/3,0/5,1/2,1/3,1/4,2/3,2/4,2/5,3/4,3/5,4/5}{0.277092,0.060399,0.222908,0.277092,0.162510}{}{0/2/{0.551926}, 1/2/{1}, 1/3/{0.551926}, 2/4/{1}, 3/4/{0.551927}}{0.0856553997}{0.1026174578}{}{N}{8}
\GuessRow{16}{0/1,0/2,1/3,4/5}{0/3,0/4,0/5,1/2,1/4,1/5,2/3,2/4,2/5,3/4,3/5}{0.133585,0.133585,0.332077,0.133585,0.133585,0.133585}{}{0/0/{0.508964}, 0/3/{1}, 0/4/{1}, 1/1/{0.508986}, 1/4/{1}, 1/5/{1}, 2/2/{1}, 3/3/{0.508908}, 3/5/{1}, 4/4/{0.508923}, 5/5/{0.508984}}{0.0639429360}{0.0685473237}{}{N}{8}
\GuessRow{18}{0/1,0/2,0/4,0/5,1/5}{0/3,1/2,1/3,1/4,2/3,2/4,2/5,3/4,3/5,4/5}{0.165053,0.172986,0.165053,0.166803,0.165053,0.165053}{}{0/0/{1}, 0/1/{1}, 1/1/{0.166807}, 1/2/{1}, 1/4/{1}, 1/5/{1}, 2/2/{1}, 3/3/{0.498663}, 4/4/{1}, 5/5/{1}}{0.0937184727}{0.0966144422}{}{N}{8}
\GuessRow{25}{0/1,0/4,0/5,1/3,4/5}{0/2,0/3,1/2,1/4,1/5,2/3,2/4,2/5,3/4,3/5}{0.333333,0.093699,0.333333,0.239635}{}{0/0/{1}, 0/1/{0.090909}, 0/2/{0.090909}, 0/3/{0.090909}, 1/1/{1}, 1/2/{0.090908}, 1/3/{1}, 2/2/{1}, 2/3/{0.090909}, 3/3/{1}}{0.1038500443}{0.1061720736}{}{N}{8}
\GuessRow{26}{0/1,0/2,0/5,1/3,4/5}{0/3,0/4,1/2,1/4,1/5,2/3,2/4,2/5,3/4,3/5}{0.205250,0.205250,0.303497,0.080752,0.205250}{}{0/0/{1}, 0/2/{0.471923}, 1/1/{1}, 1/2/{0.471921}, 2/3/{1}, 2/4/{0.471924}, 3/3/{0.000511}, 4/4/{1}}{0.0293174175}{0.0432004682}{}{N}{8}
\GuessRow{30}{0/2,0/4,1/3,1/5,4/5}{0/1,0/3,0/5,1/2,1/4,2/3,2/4,2/5,3/4,3/5}{0.241941,0.307071,0.129020,0.321969}{}{0/0/{1}, 0/3/{0.534133}, 1/2/{1}, 1/3/{0.450686}, 2/2/{0.202405}, 3/3/{0.171403}}{0.0284858934}{0.0432388287}{}{N}{8}
\GuessRow{33}{0/1,0/2,0/4,0/5,1/2,1/5}{0/3,1/3,1/4,2/3,2/4,2/5,3/4,3/5,4/5}{0.164448,0.354270,0.373633,0.107649}{}{0/1/{1}, 1/1/{1}, 1/2/{0.582889}, 3/3/{0.599620}}{0.0688329004}{0.0759651323}{}{N}{8}
\GuessRow{35}{0/2,0/4,0/5,1/2,1/3,1/5}{0/1,0/3,1/4,2/3,2/4,2/5,3/4,3/5,4/5}{0.5,0.5}{}{0/1/{0.750000}}{0.0625705719}{0.0668123578}{}{N}{8}
\GuessRow{39}{0/1,0/3,0/4,0/5,1/3,4/5}{0/2,1/2,1/4,1/5,2/3,2/4,2/5,3/4,3/5}{0.08805,0.328608,0.08805,0.328608,0.166685}{4}{0/1/{1}, 0/3/{1}, 1/2/{1}, 2/3/{1}, 0/0/{1}, 1/1/{1}, 2/2/{1}, 3/3/{1}}{0.0627947169}{0.0628954465}{}{N}{8}
\GuessRow{40}{0/1,0/2,0/4,0/5,1/3,4/5}{0/3,1/2,1/4,1/5,2/3,2/4,2/5,3/4,3/5}{0.166915,0.166404,0.166404,0.166915,0.166681,0.166681}{4,5}{0/5/{1}, 1/4/{1}, 2/4/{1}, 3/5/{1}, 4/5/{1}, 0/0/{1}, 1/1/{1}, 2/2/{1}, 3/3/{1}}{0.0617310419}{0.0655364347}{}{N}{8}
\GuessRow{41}{0/1,0/2,0/5,1/4,1/5,4/5}{0/3,0/4,1/2,1/3,2/3,2/4,2/5,3/4,3/5}{0.166663,0.166663,0.166663,0.166663,0.166663,0.166685}{5}{0/1/{1}, 0/2/{1}, 1/4/{1}, 2/3/{1}, 3/4/{1}, 0/0/{1}, 1/1/{1}, 2/2/{1}, 3/3/{1}, 4/4/{1}}{0.0771621481}{0.0772859011}{}{N}{8}
\GuessRow{42}{0/1,0/2,0/3,1/4,1/5,4/5}{0/4,0/5,1/2,1/3,2/3,2/4,2/5,3/4,3/5}{0.171738,0.135070,0.184589,0.104069,0.324730,0.079806}{}{0/0/{1}, 0/1/{1}, 0/2/{1}, 0/5/{1}, 1/1/{1}, 1/2/{1}, 1/3/{0.530157}, 2/2/{1}, 2/3/{0.530105}, 3/4/{1}, 4/5/{1}}{0.0646104850}{0.0657267553}{}{N}{8}
\GuessRow{44}{0/1,0/3,0/4,1/3,1/5,4/5}{0/2,0/5,1/2,1/4,2/3,2/4,2/5,3/4,3/5}{0.167622,0.416189,0.416189}{}{0/0/{0.599385}, 1/1/{1}, 1/2/{0.333334}, 2/2/{1}}{0.0331231529}{0.0408582489}{}{N}{8}
\GuessRow{45}{0/1,0/2,0/4,1/3,1/5,4/5}{0/3,0/5,1/2,1/4,2/3,2/4,2/5,3/4,3/5}{0.404949,0.095052,0.095052,0.404947}{}{0/2/{1}, 0/3/{0.488153}, 1/1/{0.000436}, 1/2/{1}, 1/3/{1}, 2/2/{0.009550}}{0.0404606484}{0.0481709303}{}{N}{8}
\GuessRow{46}{0/1,0/2,0/3,1/3,1/5,4/5}{0/4,0/5,1/2,1/4,2/3,2/4,2/5,3/4,3/5}{0.283329,0.135311,0.135310,0.223026,0.223025}{}{0/0/{1}, 0/1/{0.451660}, 0/2/{0.451660}, 0/3/{0.043754}, 0/4/{0.043752}, 1/1/{1}, 1/4/{1}, 2/2/{1}, 2/3/{1}, 3/3/{1}, 3/4/{0.437784}, 4/4/{1}}{0.0405785145}{0.0602613357}{}{N}{8}
\GuessRow{47}{0/2,0/3,0/4,1/3,1/5,4/5}{0/1,0/5,1/2,1/4,2/3,2/4,2/5,3/4,3/5}{0.208333,0.166667,0.208334,0.208333,0.208333}{}{0/1/{1}, 0/3/{0.666666}, 1/1/{0.008960}, 1/2/{1}, 2/4/{0.666667}, 3/4/{1}}{0.0279081647}{0.0432004546}{}{N}{8}
\GuessRow{48}{0/1,0/2,0/3,1/2,1/3,4/5}{0/4,0/5,1/4,1/5,2/3,2/4,2/5,3/4,3/5}{0.395589,0.120882,0.120882,0.120882,0.120882,0.120882}{}{0/0/{0.903422}, 1/2/{1}, 1/3/{1}, 1/4/{1}, 1/5/{1}, 2/3/{1}, 2/4/{1}, 2/5/{1}, 3/4/{1}, 3/5/{1}, 4/5/{1}}{0.2004503964}{0.2178297373}{}{N}{8}
\GuessRow{52}{0/1,0/3,0/4,0/5,1/3,1/4,1/5,3/5,4/5}{0/2,1/2,2/3,2/4,2/5,3/4}{0.302583,0.132465,0.167555,0.132465,0.132466,0.132465}{}{0/0/{0.829766}, 0/1/{1}, 0/3/{1}, 0/4/{1}, 0/5/{1}, 1/3/{1}, 1/4/{1}, 1/5/{1}, 2/2/{0.900000}, 3/4/{1}, 3/5/{1}, 4/5/{1}}{0.1917521577}{0.2067706555}{}{N}{8}
\GuessRow{54}{0/1,0/2,0/4,0/5,1/2,1/3,1/5}{0/3,1/4,2/3,2/4,2/5,3/4,3/5,4/5}{0.338717,0.408615,0.252668}{}{0/0/{1}, 0/1/{0.738941}, 0/2/{0.738941}}{0.0452784522}{0.0469245022}{}{N}{8}
\GuessRow{59}{0/1,0/3,0/4,0/5,1/3,1/5,4/5}{0/2,1/2,1/4,2/3,2/4,2/5,3/4,3/5}{0.165574,0.278141,0.278150,0.278135}{}{0/0/{0.699830}, 1/1/{0.833137}, 1/2/{0.628554}, 1/3/{0.628553}, 2/2/{0.833128}, 2/3/{0.628554}, 3/3/{0.833140}}{0.0544309552}{0.0582381881}{}{N}{8}
\GuessRow{60}{0/1,0/2,0/4,0/5,1/3,1/5,4/5}{0/3,1/2,1/4,2/3,2/4,2/5,3/4,3/5}{0.106434,0.249998,0.250003,0.249997,0.143568}{}{0/0/{1}, 0/2/{0.798951}, 0/3/{1}, 0/4/{1}, 1/3/{0.798960}, 2/4/{0.798957}, 3/3/{1}, 3/4/{1}, 4/4/{1}}{0.0288465597}{0.0311475322}{}{N}{8}
\GuessRow{61}{0/1,0/2,0/3,0/4,1/3,1/5,4/5}{0/5,1/2,1/4,2/3,2/4,2/5,3/4,3/5}{0.223458,0.294181,0.129452,0.223458,0.129451}{}{0/1/{0.244741}, 0/2/{1}, 0/3/{1}, 0/4/{0.156048}, 1/2/{1}, 1/3/{0.244741}, 1/4/{1}, 2/3/{0.156048}, 2/4/{0.070968}, 3/4/{1}}{0.0371530157}{0.0511213761}{}{N}{8}
\GuessRow{63}{0/1,0/2,0/3,0/5,1/2,1/3,4/5}{0/4,1/4,1/5,2/3,2/4,2/5,3/4,3/5}{0.141435,0.141435,0.141435,0.217129,0.141435,0.217129}{}{0/1/{1}, 0/2/{1}, 0/3/{1}, 1/4/{1}, 1/5/{1}, 2/3/{1}, 2/4/{1}, 4/5/{1}}{0.0688716894}{0.0905122900}{}{N}{8}
\GuessRow{66}{0/1,0/2,0/4,0/5,1/2,1/4,1/5,4/5}{0/3,1/3,2/3,2/4,2/5,3/4,3/5}{0.180582,0.180582,0.110213,0.180582,0.180582,0.167458}{}{0/0/{1}, 0/2/{1}, 0/3/{1}, 0/4/{1}, 1/1/{1}, 1/2/{1}, 1/3/{1}, 1/4/{1}, 2/2/{1}, 2/3/{1}, 2/4/{1}, 3/3/{1}, 4/4/{1}, 5/5/{0.799974}}{0.1115873437}{0.1120734209}{}{N}{8}
\GuessRow{73}{0/1,0/2,0/4,1/2,1/3,3/5,4/5}{0/3,0/5,1/4,1/5,2/3,2/4,2/5,3/4}{0.166667,0.208334,0.208333,0.208333,0.208333}{}{0/0/{0.927668}, 0/3/{1}, 0/4/{1}, 1/1/{1}, 1/2/{1}, 1/3/{0.666668}, 2/2/{1}, 2/4/{0.666666}, 3/3/{1}, 4/4/{1}}{0.0279081647}{0.0289172242}{}{N}{8}
\GuessRow{74}{0/1,0/3,0/4,1/2,1/3,3/5,4/5}{0/2,0/5,1/4,1/5,2/3,2/4,2/5,3/4}{0.166666,0.249999,0.250001,0.166667,0.166667}{}{0/0/{0.999759}, 0/1/{1}, 1/2/{0.499997}, 1/4/{1}, 2/3/{1}, 2/4/{1}, 3/3/{0.003092}, 4/4/{0.194182}}{0.0260416809}{0.0487738848}{}{N}{8}

}
\end{longtable}

\section*{Acknowledgements}
We thank Yixiao Zhang for pointing out the known results on the inducibility of $C_6$ and the net graph. 

We are also grateful Bo Zhang for pointing out that our previously stated lower bounds in Conjectures~\ref{conj:22}, \ref{conj:34} and \ref{conj:62}
were strictly smaller than what the provided constructions give (due to us taking non-optimal part ratios) and for suggesting a better construction for Conjecture~\ref{conj:29}. We have corrected these bounds and updated Conjecture~\ref{conj:29} accordingly.

\section*{Declaration on the use of AI}

ChatGPT Pro 5.5 was used to assist in proofreading and to search for candidate graphon constructions for the remaining unsolved cases. In particular, it improved several entries in Table~\ref{tab:guess6} and contributed the conjectured constructions for $F_{23}$ and $F_{67}$. All mathematical arguments, results, and conclusions were written and verified by the authors, or by a deterministic computer program included in the ancillary files.


\providecommand{\bysame}{\leavevmode\hbox to3em{\hrulefill}\thinspace}

\end{document}